\documentclass[a4paper,11pt]{amsart}%

\usepackage{amsmath}%
\usepackage{amssymb}%
\usepackage{amsfonts}%
\usepackage{amscd}%
\usepackage{mathrsfs}%
\usepackage[all]{xy}%
\usepackage{enumerate}%

\theoremstyle{plain}%
    \newtheorem{thm}{Theorem}[section]%
    \newtheorem{prop}[thm]{Proposition}%
    \newtheorem{lem}[thm]{Lemma}%
	\newtheorem{cor}[thm]{Corollary}%
	\newtheorem{propdef}[thm]{Proposition-Definition}%

\theoremstyle{definition}%
    \newtheorem{defn}[thm]{Definition}%
    \newtheorem{ex}[thm]{Example}%

\theoremstyle{remark}%
    \newtheorem{rem}[thm]{Remark}%

\def\endpiece{xxx}%
\def\makeop[#1]{\xmakeop#1,xxx,}
\def\mkop#1{\expandafter\def\csname #1\endcsname{{\mathrm{#1}}}} %
\def\xmakeop#1,{\def\temp{#1}\ifx\temp\endpiece\else\mkop{#1}\expandafter\xmakeop\fi}%

\newcommand{\Hom}{\mathop{\mathrm {Hom}}\nolimits}

\newcommand{\rank}{\mathop{\mathrm {rank}}\nolimits}
\newcommand{\im}{\mathop{\mathrm {im}}\nolimits}

\newcommand{\Tr}{\mathop{\mathrm {Tr}}\nolimits}

\newcommand{\Gal}{\mathop{\mathrm {Gal}}\nolimits}
\newcommand{\Gr}{\mathop{\mathrm {Gr}}\nolimits}
\newcommand{\Sym}{\mathop{\mathrm {Sym}}\nolimits}

\def\A{{\mathbb{A}}}%
\def\Z{{\mathbb{Z}}}%
\def\Q{{\mathbb{Q}}}%
\def\Zp{{\mathbb{Z}_p}}%
\def\Qp{{\mathbb{Q}_p}}%
\def\C{{\mathbb{C}}}
\def\pa{{\mathrm{par}}}%
\def\Res{{\mathrm{Res}}}%
\def\res{{\mathrm{res}}}%
\def\ideal#1{{\mathfrak{#1}}}%
\def\integer#1{{\mathcal{O}_{#1}}}%
\def\line#1{{\overline{#1}}}%
\def\cal#1{{\mathcal{#1}}}%
\def\mod#1{{\mathrm{mod}\, #1}}%
\def\intii#1#2#3#4#5#6{{\int_{#1}^{#2}\int_{#3}^{#4}\cdots\int_{#5}^{#6}}}%
\def\inti#1#2#3#4#5#6{{\int_{#1}^{#2}\cdots\int_{#3}^{#4}\int_{#5}^{#6}}}%
\def\integral#1#2#3#4{{\int_{#1}^{#2}\cdots \int_{#3}^{#4}}}%
\def\G_m^#1{{\mathbb{G}_m^{#1}}}%
\makeop[Tr,Nr,Ver,Ker,Aut,det,sgn,End,ev]%
\makeop[GL,SL,E,M,K,SK]%
    
\begin{document}

\setcounter{tocdepth}{1} 
\numberwithin{equation}{section}


\title[congruences between Hilbert modular forms]
{congruences between Hilbert modular forms of weight $\mathbf{2}$, and special values of their $L$-functions}
\author{Yuichi Hirano}
\address{Graduate School of Mathematical Sciences, The University of
Tokyo, 8-1 Komaba 3-chome, Meguro-ku, Tokyo, 153-8914, Japan}
\email{yhirano@ms.u-tokyo.ac.jp}
\date{\today}

\begin{abstract}
The purpose of this paper is to show how a congruence between (the Fourier coefficients of) a Hilbert cusp form and a Hilbert Eisenstein series of parallel weight $2$ gives rise to congruences between algebraic parts of critical values of their $L$-functions. 
This is a generalization of a result of V. Vatsal.
\end{abstract}

\maketitle
\setcounter{section}{-1}

\section{Introduction}
%

%
\subsection{Introduction}
%

In this paper, we study a way to obtain congruences between special values of $L$-functions from a congruence between 
a Hilbert cusp form and a Hilbert Eisenstein series of parallel weight $2$. 
Our result is a generalization of the work of V. Vatsal \cite{Vat} for elliptic modular forms.

Let $F$ be a totally real number field of degree $n$ with narrow class number $h_F^+=1$. 
Let $\Delta_F$ denote the discriminant of $F$. 
Let $\mathfrak{n}$ be an integral ideal of $F$ such that 
$\mathfrak{n}$ is prime to $6\Delta_F$. 
Let $p$ be a prime number such that $p\ge n+2$ and $p$ is prime to $6\mathfrak{n}\Delta_F$. 
We fix algebraic closures $\overline{\Q}$ of $\Q$ and $\overline{\Q}_p$ of $\Qp$ and embeddings $\overline{\Q} \hookrightarrow \overline{\Q}_p \hookrightarrow \C$. 
Let $\integer{}$ be the ring of integers of a finite extension $K$ of $\Qp$ 
and $\varpi$ a uniformizer of $\integer{}$. 
Let $M_2(\ideal{n},\integer{})$ (resp. $S_2(\ideal{n},\integer{})$) denote the space of Hilbert modular (resp. cusp) forms of parallel weight $2$ and level $\ideal{n}$ with coefficients in $\integer{}$ (see (\ref{integral MF})).
Let $Y(\ideal{n})$ be the Shimura variety $\Gamma_{1}(\ideal{d}_F[t_1],\ideal{n})\backslash \mathfrak{H}^{\mathrm{Hom}(F, \mathbb{R})}$ defined by $($\ref{adele HMV}$)$, and let $Y(\ideal{n})^{\mathrm{BS}}$ be the Borel--Serre compactification of $Y(\ideal{n})$ $($cf. \S \ref{Borel--Serre}$)$. 
Let $C$ denote the set of all cusps of $Y(\ideal{n})$, and let $D_s$ denote the boundary of $Y(\ideal{n})^{\mathrm{BS}}$ at $s\in C$.
Let $C_{\infty}$ be the subset of $C$ consisting of cusps $\Gamma_0(\ideal{d}_F [t_1],\ideal{n})$\nobreakdash-equivalent to the cusp $\infty$ (where $\Gamma_0(\ideal{d}_F [t_1],\ideal{n})$ is the congruence subgroup defined in \S\ref{Analytic HMV}). 
Let $D_{C_{\infty}}(\ideal{n})$ denote the union of $D_s$ for all $s\in C_{\infty}$.

\begin{thm}[=Theorem \ref{thm:congruence}]\label{main theorem}
Let $\varphi$ and $\psi$ be totally even $($resp. totally odd$)$ $\integer{}$-valued narrow ray class characters of $F$ of conductor $\ideal{m}_{\varphi}$ and $\ideal{m}_{\psi}$ such that $\ideal{m}_{\varphi\psi}=\ideal{m}_{\varphi}\ideal{m}_{\psi}=\ideal{n}$ and $\epsilon$ the character $\sgn^{\Hom(F,\mathbb{R})}$ $($resp. $\textbf{1}$$)$ of the Weyl group $W_G$ $($for the definition, see before Proposition \ref{anti Mellin}$)$. 
Put $\chi=\varphi\psi$. 
Let $\mathbf{E}$ denote the Hilbert Eisenstein series $\mathbf{E}_2(\varphi,\psi)\in M_2(\ideal{n},\integer{})$ associated to the pair $(\varphi,\psi)$ with character $\chi$ $($see Proposition \ref{Hilbert Eisenstein}$)$.
Assume that $\varphi$ is non\nobreakdash-trivial and the algebraic Iwasawa $\mu$-invariants of $\overline{\Q}^{\ker (\varphi)}$ and $\overline{\Q}^{\ker (\psi)}$ $($for the definition, see \cite[\S 13.3]{Was}$)$ are $0$. 
Let $\mathbf{f}\in S_2(\ideal{n},\integer{})$ be a normalized Hecke eigenform for all Hecke operators with character $\chi$. 
We assume the following four conditions$:$
\begin{enumerate}[$(a)$] 
\item $\mathbf{f} \equiv \mathbf{E}\ (\bmod \ \varpi)$ $($for the definition, see before Theorem \ref{thm:congruence}$)$$;$

\item the local components $H^{n}(\partial \left(Y(\ideal{n})^{\mathrm{BS}}\right),\integer{})_{\ideal{m}}$ and $H_c^{n+1}(Y(\ideal{n}),\integer{})_{\ideal{m}}$ are torsion-free, where $\ideal{m}$ is the Eisenstein maximal ideal of $\cal{H}_2(\ideal{n},\integer{})$ defined before Theorem \ref{cong mod}$;$

\item the local component $H^{n}(D_{C_{\infty}}(\ideal{n}),\integer{})_{\ideal{m}_{\mathbf{E}}'}$ is torsion-free, where $\ideal{m}_{\mathbf{E}}'$ is the maximal ideal of $\mathbb{H}_2(\ideal{n},\integer{})'$ defined before Proposition \ref{H^{n-1} vanishing}$;$

\item there exists a prime ideal $\ideal{q}$ of $\ideal{o}_F$ dividing $\ideal{n}$ such that $C(\ideal{q},\mathbf{E})\not\equiv N(\ideal{q})\ (\bmod \varpi)$, 
where $C(\ideal{q},\mathbf{E})$ denotes the $U(\ideal{q})$-eigenvalue $\varphi(\ideal{q})+\psi(\ideal{q})N(\ideal{q})$ of $\mathbf{E}$. 
\end{enumerate}
Then there exist $\Omega_{\mathbf{f}}^{\epsilon}\in \C^{\times}$ and $u \in \integer{}^{\times}$ such that, 
for every narrow ray class character 
$\eta$ of $F$ of conductor $\ideal{m}_{\eta}$ such that $\ideal{n}|\ideal{m}_{\eta}$ and 
$\eta=\epsilon$ on $W_G\simeq \A_{F,\infty}^{\times}/\A_{F,\infty,+}^{\times}$, 
the both values $\tau({\eta}^{-1})D(1,\mathbf{f},\eta)/(2\pi \sqrt{-1})^n\Omega_{\mathbf{f}}^{\epsilon}$ and $\tau(\eta^{-1})D(1,\mathbf{E},\eta)/(2\pi \sqrt{-1})^n$ belong to $\integer{}(\eta)$ and the following congruence holds$:$ 
\begin{align*}
&\tau(\eta^{-1})\frac{D(1,\mathbf{f},\eta)}{(2\pi\sqrt{-1})^n\Omega_{\mathbf{f}}^{\epsilon}}
= u \tau(\eta^{-1})\frac{D(1,\mathbf{E},\eta)}{(2\pi\sqrt{-1})^n} 
\ \ \text{in}\ \ \integer{}(\eta)/\varpi.
\end{align*}
Here $\tau(\eta^{-1})$ denotes the Gauss sum attached to $\eta^{-1}$ $($see $($\ref{Gauss})$)$, $D(1,\ast,\eta)$ is given by the Dirichlet series in the sense of G. Shimura $($see $($\ref{L-function}$)$$)$, and 
$\integer{}(\eta)$ denotes the ring of integers of the field generated by $\im(\eta)$ over $K$. 
\end{thm}

This result is a generalization of the result of Vatsal \cite{Vat} in the case where $F=\Q$ and weight $k=2$. 
However, the methods to prove the main theorem have some limitations, such as the need for the torsion-freeness of the compact support cohomology and the boundary cohomology. 
In the case where $F$ is a real quadratic field, the torsion-freeness is satisfied under some conditions (see Proposition \ref{prop:ab torsion-free} and \ref{prop:bou torsion-free}).
We also give an example of a congruence between a Hilbert cusp form and a Hilbert Eisenstein series satisfying all the assumptions of the main theorem (Example \ref{Example cong}).

\begin{rem}\label{rem:Fe--Wa}
(1)
The assumption on the algebraic Iwasawa $\mu$-invariants of $\overline{\Q}^{\ker (\varphi)}$ and $\overline{\Q}^{\ker (\psi)}$ is satisfied 
if $\overline{\Q}^{\ker (\varphi)}$ and $\overline{\Q}^{\ker (\psi)}$ are abelian extensions over $\Q$ by Ferrero\nobreakdash-Washington theorem (see, for example, \cite[\S 7.5, Theorem 7.15]{Was}).

(2)
The assumption $\ideal{n}|\ideal{m}_{\eta}$ is used in a cohomological description of the special value $\tau(\eta^{-1})D(1,\mathbf{E},\eta)/(2\pi\sqrt{-1})^n$ as follows. 
Since $\varphi$ is non-trivial, $\mathbf{E}$ vanishes at cusps $\Gamma_0(\ideal{d}_F [t_1],\ideal{n})$ equivalent to the cusp $\infty$.
The assumption $\ideal{n}|\ideal{m}_{\eta}$ 
allows us to describe the special value in terms of Mellin transforms relevant to cusps $\Gamma_0(\ideal{d}_F [t_1],\ideal{n})$-equivalent to $\infty$ (Proposition \ref{Modular symbol} and \ref{Modular symbol anti-hol}).
\end{rem}

We give an outline of the proof of the main theorem (Theorem \ref{main theorem}=Theorem \ref{thm:congruence}) below in order to clarify its complicated structure, the methods used, and the places where the assumptions are necessary.
The proof consists of five steps.

\textbf{Step 1.} To prove Mellin transforms for the relative cohomology classes of $\mathbf{E}$ and ${\mathbf{f}}$.

Since $\varphi$ is non-trivial, $\mathbf{E}=\mathbf{E}_2(\varphi,\psi)$ vanishes at every $s\in C_{\infty}$.
Therefore we can define the relative cohomology class $[\omega_{\mathbf{E}}]_{\mathrm{rel}}$ (resp. $[\omega_{\mathbf{f}}]_{\mathrm{rel}}$) associated to $\mathbf{E}$ (resp. ${\mathbf{f}}$) in $H^n(Y(\ideal{n})^{\mathrm{BS}},D_{C_{\infty}}(\ideal{n});\C)$, whose image in $H^n(Y(\ideal{n}),\C)$ is the cohomology class $[\omega_{\mathbf{E}}]$ (resp. $[\omega_{\mathbf{f}}]$) associated to $\mathbf{E}$ (resp. $\mathbf{f}$). 
Then we prove that the value $\tau(\eta^{-1})D(1,\mathbf{E},\eta)/(2\pi\sqrt{-1})^n$ (resp. $\tau(\eta^{-1})D(1,\mathbf{f},\eta)/(2\pi\sqrt{-1})^n$) is expressed as a linear combination of the images of $[\omega_{\mathbf{E}}]_{\mathrm{rel}}$ (resp. $[\omega_{\mathbf{f}}]_{\mathrm{rel}}$) under the evaluation maps with integral coefficients (Proposition \ref{Modular symbol} and \ref{Modular symbol anti-hol}), where we use the assumption that weight $k=2$. The proof is based on the method of T. Oda \cite{Oda}, H. Hida \cite{Hida94}, and T. Ochiai \cite{Ochi} for a Hilbert cusp form.  
By the assumption $\ideal{n}|\ideal{m}_{\eta}$, we can generalize the Mellin transform to a Hilbert modular form vanishing at every $s\in C_{\infty}$ as mentioned in Remark \ref{rem:Fe--Wa} (2).

\textbf{Step 2.} To prove the integrality of the restriction of the cohomology class associated to a Hilbert modular form to the boundary.

For $\mathbf{h} \in M_2(\ideal{n},\integer{})$ and $s\in C$, we prove that the image of the cohomology class $[\omega_{\mathbf{h}}]$ of $\mathbf{h}$ belonging to $H^n(Y(\ideal{n}),\C)=H^n(Y(\ideal{n})^{\mathrm{BS}},\C)$ under the restriction map to $H^n(D_s,\C)$ is integral, that is, $\mathrm{res}([\omega_{\mathbf{h}}]) \in H^n(D_s,\integer{})/(\text{$\integer{}$-torsion})$ (see Proposition \ref{const}).
For the proof, we express the image of $[\omega_{\mathbf{h}}]$ in the group cohomology $H^n(\Gamma_{1}(\ideal{d}_F[t_1],\ideal{n}),\C)$ in terms of multiple integrals (see Proposition-Definition \ref{def:cocycle}) (following the method of H. Yoshida \cite{Yo} in the case where $F$ is a real quadratic field), and we explicitly compute its restriction to the boundary 
(generalizing the method of G. Stevens \cite{Ste1} in the case where $F=\Q$).
The author does not know any other means to prove the integrality, for instance, using de Rham cohomology.

\textbf{Step 3.} To prove the integrality of $[\omega_{\mathbf{E}}]_{\mathrm{rel}}$ and $[\omega_{\mathbf{f}}]_{\mathrm{rel}}/\Omega_{\mathbf{f}}$.

For the Eisenstein series $\mathbf{E}$, we first prove the rationality of $[\omega_{\mathbf{E}}]$ and $[\omega_{\mathbf{E}}]_{\mathrm{rel}}$, that is, $[\omega_{\mathbf{E}}]\in H^n(Y(\ideal{n}),K)$ (Proposition \ref{prop:rational'}) and $[\omega_{\mathbf{E}}]_{\mathrm{rel}}\in H^n(Y(\ideal{n})^{\mathrm{BS}},D_{C_{\infty}}(\ideal{n});K)$ (Proposition \ref{prop:rational});
the latter follows from the former and a vanishing result on $H^{n-1}(D_{C_{\infty}}(\ideal{n}),\C)$ (Proposition \ref{H^{n-1} vanishing}).
The proof of the vanishing result is based on an explicit computation of the action of Hecke operators at places dividing the level $\ideal{n}$ on $H^{n-1}(D_{C_{\infty}}(\ideal{n}),\C)$, where we use the assumptions that $h_F^+=1$ and weight $k=2$. 
Moreover, under the assumption on the algebraic Iwasawa $\mu$\nobreakdash-invariants and the assumptions (b), (c), and (d), we prove the integrality of $[\omega_{\mathbf{E}}]$ and $[\omega_{\mathbf{E}}]_{\mathrm{rel}}$, that is, $[\omega_{\mathbf{E}}]\in H^n(Y(\ideal{n}),\integer{})/(\text{$\integer{}$-torsion})$ and $[\omega_{\mathbf{E}}]_{\mathrm{rel}}\in H^n(Y(\ideal{n})^{\mathrm{BS}},D_{C_{\infty}}(\ideal{n});\integer{})/(\text{$\integer{}$-torsion})$, and further the mod $\varpi$ non-vanishing of $[\omega_{\mathbf{E}}]$ and $[\omega_{\mathbf{E}}]_{\mathrm{rel}}$ (Corollary \ref{thm:integral}).
The proof is based on the method of C. Skinner (in the case where $F=\Q$) and T. Berger \cite{Be} (in the case where $F$ is an imaginary quadratic field);
our result follows from the Mellin transform for the relative cohomology class $[\omega_{\mathbf{E}}]_{\mathrm{rel}}$ (mentioned in Step 1), the integrality $\res([\omega_{\mathbf{E}}])\in H^n(\partial\left(Y(\ideal{n})^{\text{BS}}\right),\integer{})/(\text{$\integer{}$-torsion})$ (mentioned in Step 2), and the Iwasawa main conjecture for totally real number fields (proved by A. Wiles \cite{Wil}). 

For the cusp form $\mathbf{f}$, we prove that there exists $\Omega_{\mathbf{f}}\in \C^{\times}$ such that the class $[\omega_{\mathbf{f}}]/\Omega_{\mathbf{f}}$ belongs to $H_{\pa}^n(Y(\ideal{n}),\integer{})/(\text{$\integer{}$-torsion})$ and its reduction modulo $\varpi$ does not vanish. 
The key ingredients of the proof are the Hecke-equivariance of the canonical homomorphism $H_{\pa}^n(Y(\ideal{n}),K) \hookrightarrow H_{\pa}^n(Y(\ideal{n}),\C)$ induced by the fixed embedding $K\hookrightarrow \C$ and the Eichler\nobreakdash-Shimura-Harder isomorphism $H_{\pa}^n(Y(\ideal{n}),\C)[\epsilon]\simeq S_2(\ideal{n},\C)$ as Hecke modules (see (\ref{+,+ decomp})), where the left-hand side is the $\epsilon$\nobreakdash-isotypic part.
The proof of the Eichler-Shimura-Harder isomorphism is based on an explicit computation of the action of $W_G$ on the space of invariant differential forms, where we use the assumption $h_F^+=1$. 

\textbf{Step 4.} To prove the main theorem.

We prove the congruence between the special values $\tau(\eta^{-1})D(1,\mathbf{E},\eta)/(2\pi\sqrt{-1})^n$ and $\tau(\eta^{-1})D(1,\mathbf{f},\eta)/(2\pi\sqrt{-1})^n$. 
It follows from the Mellin transforms for the relative cohomology classes $[\omega_{\mathbf{E}}]_{\mathrm{rel}}$ and $[\omega_{\mathbf{f}}]_{\mathrm{rel}}$ (mentioned in Step 1) and the congruence between $[\omega_{\mathbf{E}}]_{\mathrm{rel}}$ and $[\omega_{\mathbf{f}}]_{\mathrm{rel}}/\Omega_{\mathbf{f}}$ (mentioned in Step 5).

\textbf{Step 5.} To prove the congruence between $[\omega_{\mathbf{E}}]_{\mathrm{rel}}$ and $[\omega_{\mathbf{f}}]_{\mathrm{rel}}/\Omega_{\mathbf{f}}$. 

We prove the congruence between the integral cohomology classes $[\omega_{\mathbf{E}}]$ and $[\omega_{\mathbf{f}}]/\Omega_{\mathbf{f}}$ in the parabolic cohomology $\widetilde{H}_{\pa}^n(Y(\ideal{n}),\integer{})/\varpi$ (Theorem \ref{partial one}).
Moreover, we lift the congruence to the relative cohomology classes $[\omega_{\mathbf{E}}]_{\mathrm{rel}}$ and $[\omega_{\mathbf{f}}]_{\mathrm{rel}}/\Omega_{\mathbf{f}}$ in $\widetilde{H}^n(Y(\ideal{n})^{\text{BS}},D_{C_{\infty}}(\ideal{n});\integer{})/\varpi$ by using Proposition \ref{H^{n-1} vanishing} (mentioned in Step 3). 
The proof of the former is based on multiplicity one results for the $\mathbf{E}$-part $\widetilde{\overline{M}}_{\mathbf{E}}$ and $\mathrm{Fil}^n(\widetilde{\overline{M}}_{\mathbf{f}})$ of the $\mathbf{f}$-part $\widetilde{\overline{M}}_{\mathbf{f}}$, and integral $p$-adic Hodge theory. 
Here $\widetilde{\overline{M}}_{\mathbf{E}}$ (resp. $\widetilde{\overline{M}}_{\mathbf{f}}$) is the quotient $\widetilde{M}_{\mathbf{E}}/\varpi$ (resp. $\widetilde{M}_{\mathbf{f}}/\varpi$) of the torsion-free $\mathbf{E}$-part (resp. $\mathbf{f}$-part) $\widetilde{M}_{\mathbf{E}}$ (resp. $\widetilde{M}_{\mathbf{f}})$ of the integral log\nobreakdash-crystalline cohomology by $\varpi$ (for the definition, see \S \ref{Pf of main}, \S\ref{main result}). 
By the theory of Eisenstein cohomology and the $q$-expansion principle over $\C$, we prove that the dimension of $\widetilde{\overline{M}}_{\mathbf{E}}$ over $\integer{}/\varpi$ is $1$ and $\widetilde{\overline{M}}_{\mathbf{E}}=\mathrm{Fil}^n(\widetilde{\overline{M}}_{\mathbf{E}})$ (Proposition \ref{Eis HT}), where we use the assumption (d) on Hecke eigenvalues at places dividing the level.
Since the dimension of $\mathrm{Fil}^n(\widetilde{\overline{M}}_{\mathbf{f}})$ over $\integer{}/\varpi$ is also $1$, we obtain $\widetilde{\overline{M}}_{\mathbf{E}}=\mathrm{Fil}^n(\widetilde{\overline{M}}_{\mathbf{E}})\simeq \mathrm{Fil}^n(\widetilde{\overline{M}}_{\mathbf{f}})$ (Proposition \ref{Fil f mod pi} and \S \ref{Pf of main}), where the second isomorphism follows from the assumptions (b) and (c), and the congruence (a) between Hecke eigenvalues including at places dividing the level. Now the assertion follows from integral $p$-adic Hodge theory (\S \ref{Pf of main}).
The result of this step may be regarded as an analogue of the multiplicity one theorem for modulo $p$ parabolic cohomology in the case where the residual Galois representation $\overline{\rho}_\textbf{f}\ (=\rho_{\textbf{f}}\ (\bmod \varpi))$ associated to $\textbf{f}$ is reducible. 
When $\overline{\rho}_{\textbf{f}}$ is irreducible, under some assumptions, the multiplicity one theorem has been proved by M. Dimitrov \cite{Dim2} for a general totally real number field. \\

The organization of this paper is as follows. 

In \S \ref{Hilbert}, 
we summarize results on Hilbert modular varieties and Hilbert modular forms in the analytic and algebraic settings. 
Moreover, we state basic properties of Hilbert Eisenstein series (Proposition \ref{Hilbert Eisenstein} and \ref{Const of Eisenstein}), which are of great utility in the following sections. 

In \S \ref{subsection:Mellin}, we give a cohomological description of special values of $L$-functions associated to a Hilbert modular form vanishing at cusps $\Gamma_0(\ideal{d}_F [t_1],\ideal{n})$-equivalent to $\infty$ (Proposition \ref{Modular symbol} and \ref{Modular symbol anti-hol}). The evaluation of the associated cohomology class on Hilbert modular cycles produces special values of $L$-functions. 

In \S \ref{subsection:Rationality}, we prove the integrality of the restriction of the cohomology class associated to a Hilbert Eisenstein series to the boundary of the Borel--Serre compactification of the Hilbert modular variety (Proposition \ref{const}). 

In \S \ref{subsection:Eisenstein cohomology and ESH},
we recall the theory of Eisenstein cohomology and the Eichler\nobreakdash--Shimura--Harder homomorphism. 
We prove the Eichler\nobreakdash--Shimura--Harder isomorphism (\ref{+,+ decomp}) for the $\epsilon$-part.

In \S \ref{subsection:Rationality and Integrality}, we generalize Stevens's result \cite{Ste2} on the integrality of the cohomology class associated to an elliptic modular form. 
We prove the integrality of the cohomology class $[\omega_{\mathbf{E}}]$ associated to $\mathbf{E}$ (Corollary \ref{thm:integral}).

In \S \ref{congruence}, 
we prove the main theorem (Theorem \ref{main theorem}=Theorem \ref{thm:congruence}). 
The key ingredient of our proof is the congruence between the integral cohomology classes of $\mathbf{f}$ and $\mathbf{E}$ in the parabolic cohomology, whose proof is postponed to \S \ref{comparison} (Theorem \ref{partial one}).
Combining with Proposition \ref{Modular symbol} and \ref{Modular symbol anti-hol}, we obtain the main theorem.

In \S\ref{comparison}, we prove the congruence between the integral cohomology classes of $\mathbf{f}$ and $\mathbf{E}$ in the parabolic cohomology (Theorem \ref{partial one}) by combining the theory of Eisenstein cohomology and integral $p$\nobreakdash-adic Hodge theory.

\tableofcontents

%
\subsection{Notation}
%

Let $\widehat{\Z}$ denote $\prod_{l}\Z_l$, where $l$ runs over all rational primes. 
We abbreviate $\A_{\Q}$, the ring of adeles of $\Q$, to $\A$. 
We fix a rational prime number $p>3$.
We fix algebraic closures $\line{\Q}$ of $\Q$ and $\line{\Q}_p$ of the field of $p$-adic numbers $\Qp$, 
and embeddings $\iota_p:\line{\Q} \to \line{\Q}_p$ and $\line{\Q}_p \rightarrow \C$,
where $\mathbb{C}$ denotes the field of complex numbers.

Let $F$ be a totally real number field unramified at $p$, $n$ the degree $[F\colon\Q]$ of the extension $F/\Q$, and $\mathfrak{o}_F$ the ring of integers of $F$.
For a place $v$ of $F$ (resp. a non-zero prime ideal $\ideal{q}$ of $\ideal{o}_F$), let $F_v$ (resp. $F_{\ideal{q}}$) denote the completion at $v$ (resp. the $\ideal{q}$-adic completion) of $F$.
Let $\ideal{o}_{F_{\ideal{q}}}$ denote the ring of integers of $F_{\ideal{q}}$, and $\widehat{\ideal{o}}_F$ the product of $\ideal{o}_{F_{\ideal{q}}}$ over all non-zero prime ideals $\ideal{q}$ of $\ideal{o}_F$.
Let $J_F$ denote the set of embeddings of $F$ into $\mathbb{R}$.
For $a\in F$ and $\iota\in J_F$, let $a^{\iota}$ denote $\iota(a)$.
We have $F\otimes_{\Q}\mathbb{R} \simeq \mathbb{R}^{J_F}$, and write $(F\otimes_{\Q}\mathbb{R})_+^{\times}$ for the subgroup of $(F\otimes_{\Q}\mathbb{R})^{\times}$ corresponding to $(\mathbb{R}_+^{\times})^{J_F}$, where $\mathbb{R}_+^{\times}$ denotes the multiplicative group of positive real numbers.
As usual, $\A_F$ denotes the ring of adeles of $F$, which is the product of the finite part $\A_{F,f}(\simeq \widehat{\ideal{o}}_F \otimes_{\ideal{o}_F}F)$ and the infinite part $\A_{F,\infty}(\simeq F\otimes_{\Q}\mathbb{R})$.
For $x\in \A_F$ and a place $v$ of $F$, $x_0$, $x_{\infty}$, and $x_v$ denote the finite component $\in \A_{F,f}$, the infinite component $\in \A_{F,\infty}$, and the $v$-component $\in F_v$ of $x$, respectively. 
For $x \in \A_F$, a subset $X$ of $\A_F$, and a non-zero ideal $\ideal{n}$ of $\ideal{o}_F$, we write $x_{\ideal{n}}$ and $X_{\ideal{n}}$ for the images of $x$ and $X$ in $\prod_{\ideal{q}|\ideal{n}} F_{\ideal{q}}$, where
$\ideal{q}$ denotes a non-zero prime ideal of $\ideal{o}_F$. 
Let $N$ denote the norm map $\mathrm{Nr}_{F/\Q}$ of the extension $F/\Q$, 
$\ideal{d}_F \subset \ideal{o}_F$ the different of $F$, and $\Delta_F$ the discriminant $N(\ideal{d}_F)$ of $F$, which is prime to $p$ by assumption. 
Let $\mathrm{Cl}_F^+$ denote the narrow ideal class group of $F$. 
We have an isomorphism 
$F^{\times}\backslash \A_F^{\times}/\widehat{\ideal{o}}_F^{\times}(F\otimes_{\Q}\mathbb{R})_+^{\times} \xrightarrow{\simeq} \mathrm{Cl}_F^+$ sending the class of $x\in \A_F^{\times}$ to the class of the fractional ideal
$[x]:=\prod_{\ideal{q}} \ideal{q}^{\mathrm{ord}_{\ideal{q}}(x_{\ideal{q}})}$, where $\ideal{q}$ runs over the set of all non-zero prime ideals of $\ideal{o}_F$. Let $D$ be an element of $\A_F^{\times}$ such that $[D] = \ideal{d}_F$ and $D_{\infty}=1$.

For a non-zero ideal $\ideal{b}$ of $\ideal{o}_F$, 
let $\mathrm{Cl}_F^+(\ideal{b})$ denote the narrow ray class group of $F$ modulo $\ideal{b}$.
By a narrow ray class character of $F$ modulo $\ideal{b}$, we mean a homomorphism $\chi : \mathrm{Cl}_F^+(\ideal{b}) \to \C^{\times}$.
The conductor of $\chi$ is the smallest divisor $\ideal{m}_{\chi}$ of $\ideal{b}$ such that $\chi$ factors through $\mathrm{Cl}_F^+(\ideal{m}_{\chi})$.
For a narrow ray class character $\chi$ of $F$ modulo $\ideal{b}$, there exists $r = (r_{\iota})_{\iota\in J_F}\in (\Z/2\Z)^{J_F}$ such that
\[
\chi((\alpha))=\sgn(\alpha)^r \ \ \ \  \text{for all}\ \ \ \ \alpha\in F^{\times}\ \ \ \ \text{satisfying}\ \ \ \ \alpha\equiv 1\ (\bmod \ideal{b}).
\]
Here $\sgn(x)$ for $x \in \mathbb{R}^{\times}$ denotes the sign of $x$ and $\sgn(\alpha)^r=\prod_{\iota\in J_F}\sgn(\alpha^{\iota})^{r_{\iota}}$, where we identify $J_F$ with the set of infinite places of $F$. 
We call $r$ the sign of $\chi$. 
We say that $\chi$ is totally even (resp. totally odd) if $r_{\iota} = 0$ (resp. $r_{\iota} = 1$) for all $\iota \in J_F$.

For an algebraic group $H$ defined over $\Q$, 
$H(\mathbb{R})$ is abbreviated to $H_{\infty}$ and 
$H_{\infty,+}$ denotes the connected component of $H_{\infty}$ containing the unit. 
Let $G$ be the reductive algebraic group $\Res_{F/\Q}(\mathrm{GL}_{2/F})$ over $\Q$, where $\Res_{F/\Q}$ denotes the Weil restriction of scalars. 
We have $G_{\infty}= \GL_2(\mathbb{R})^{J_F}$, $G_{\infty,+}= \GL_2(\mathbb{R})_+^{J_F}$, and $G(\A)=\GL_2(\A_F)$. 
Let $B$ denote the Borel subgroup of $G$ consisting of upper triangular matrices, and let $U$ denote its unipotent radical.

%
\subsection{Acknowledgement}
%

I would like to express my gratitude to Professor Takeshi Tsuji 
for providing helpful comments and suggestions, and 
pointing out mathematical mistakes during the course of my study. 
In particular, the work in $\S$\ref{comparison} would have been impossible without his insight and guidance.

%
\section{Hilbert modular varieties and Hilbert modular forms}\label{Hilbert}
%

%
\subsection{Analytic Hilbert modular varieties}\label{Analytic HMV}
%

In this subsection, we recall the definition of analytic Hilbert modular varieties. 
For more detail, refer to \cite[\S 1.1]{Dim2}. 

Let $\mathfrak{H}$ be the upper half plane $\{ z \in \mathbb{C} \mid \mathrm{Im}(z) > 0 \}$. 
The group $\GL_2(\mathbb{R})_+$ acts on $\mathfrak{H}$ by linear fractional transformations. 
We can extend the action to $\GL_2(\mathbb{R})$ by defining the action of $\begin{pmatrix}-1&0\\ 0&1\\\end{pmatrix}$ on $\mathfrak{H}$ by $z\mapsto -\bar{z}$. 
We define the action of $G_{\infty}=\GL_2(\mathbb{R})^{J_F}$ on $\mathfrak{H}^{J_F}$ by $(g_{\iota})_{\iota\in J_F}\cdot(z_{\iota})_{\iota\in J_F}=(g_{\iota}z_{\iota})_{\iota\in J_F}$.
Let $\underline{\textbf{i}}=(\sqrt{-1},\cdots,\sqrt{-1}) \in \mathfrak{H}^{J_F}$. 
Let $K_{\infty}$ and $K_{\infty,+}$ be the stabilizers of $\underline{\textbf{i}}$ in $G_{\infty}$ and $G_{\infty,+}$, respectively. 

For a non-zero ideal $\mathfrak{n}$ of $\mathfrak{o}_F$, we define the open compact subgroup $K_1(\ideal{n})$ of $G(\mathbb{A}_f)$ by 
\begin{align*}
&K_1(\mathfrak{n})=\left\{\begin{pmatrix}a&b\\c&d\\\end{pmatrix}\in G(\widehat{\Z})\bigg\vert c\in \mathfrak{n},d-1\in \mathfrak{n}\right\}. 
\end{align*}
The adelic Hilbert modular variety of level $K_1(\ideal{n})$ is defined by
\begin{align}\label{adele HMV}
Y(\ideal{n})&=G(\Q)\backslash G(\mathbb{A})/K_1(\ideal{n}) K_{\infty,+}\\
&\nonumber=G(\Q)_+\backslash G(\mathbb{A})_+/K_1(\ideal{n}) K_{\infty,+},
\end{align}
where $G(\A)_+=G(\A_f) G_{\infty,+}$ and $G(\Q)_+=G(\Q)\cap G_{\infty,+}$.

Then $Y(\ideal{n})$ is a disjoint union of finitely many arithmetic quotients of $\mathfrak{H}^{J_F}$ as follows. 
Let $\ideal{a}$ be a fractional ideal of $F$.
We consider the following congruence subgroups of $G(\Q)_+$:
\begin{align}\label{cong subgroup}
\Gamma_0(\ideal{a},\mathfrak{n})
&= \left\{ \begin{pmatrix}
a&b\\
c&d
\end{pmatrix} \in \begin{pmatrix}\mathfrak{o}_F&\ideal{a}^{-1}\\ \ideal{a}\mathfrak{n}&\mathfrak{o}_F\end{pmatrix} \bigg| ad-bc \in \mathfrak{o}_{F,+}^{\times} \right\},\\
\nonumber \Gamma_1(\ideal{a},\mathfrak{n})
&=\left\{ \begin{pmatrix}
a&b\\
c&d
\end{pmatrix} \in \Gamma_0(\ideal{a},\mathfrak{n}) \bigg| d\equiv  1 \bmod \mathfrak{n} \right\},\\
\nonumber \Gamma_1^1(\ideal{a},\mathfrak{n})
&=\Gamma_1(\ideal{a},\mathfrak{n})\cap \SL_2(F), 
\end{align}
where $\mathfrak{o}_{F,+}^{\times}$ denotes the subgroup of $\mathfrak{o}_F^{\times}$ consisting of totally positive units. 

Let $\mathrm{Cl}_F^+$ be the narrow ideal class group of $F$ and 
$h_F^+$ the narrow class number $\sharp\mathrm{Cl}_F^+$ of $F$. 
Choose and fix $t_1,\cdots,t_{h_F^+} \in \A_F^{\times}$ such that $t_{i,\infty}=1$ and the corresponding fractional ideals $[t_1],\cdots,[t_{h_F^+}]$ form a complete set of representatives of $\mathrm{Cl}_F^+$. 
We put 
\[
x_i=\begin{pmatrix}D^{-1}t_i^{-1}&0\\0&1\\\end{pmatrix}\in G(\A)_+. 
\]

We define the analytic Hilbert modular varieties $Y_i$ by 
\begin{align}\label{analytic HMV}
Y_i=\overline{\Gamma_1(\ideal{d}_F[t_i],\mathfrak{n})}\backslash \mathfrak{H}^{J_F}, 
\end{align}
where  $\overline{\Gamma}$ denotes $\Gamma/(\Gamma\cap F^{\times})$ for a congruence subgroup $\Gamma$ of $G(\Q)_+$. 
Then, by the strong approximation theorem, we have the following description of $Y(\ideal{n})$:
\begin{align}\label{adele var}
Y(\ideal{n})\simeq \coprod_{1\le i \le h_F^+} Y_i
\end{align}
given by sending the class of $x_i g\in Y(\ideal{n})$ to the class of $g \underline{\mathbf{i}}\in Y_i$ for $g\in G_{\infty,+}$.

We also need the following varieties: 
\begin{align}\label{adele var^1}
Y^1(\ideal{n})=\coprod_{1\le i \le h_F^+} Y_i^1, \ \ 
Y_i^1=\overline{\Gamma_1^1(\ideal{d}_F[t_i],\mathfrak{n})}\backslash \mathfrak{H}^{J_F}.
\end{align}

%
\subsection{Analytic Hilbert modular forms}\label{Analytic HMF}
%

In this subsection, we fix notation concerning the spaces of Hilbert modular forms, following \cite[\S 1.2]{Dim2}. 

Let $k$ be an integer $\ge 2$ and $\ideal{n}$ a non-zero ideal of $\ideal{o}_F$. 
Let $t=\sum_{\iota\in J_F}\iota \in \Z[J_F]$.

Let $M_k(\ideal{n},\C)$ (resp. $S_k(\ideal{n},\C)$) denote the $\C$-vector space $G_{kt,J_F}(K_1(\ideal{n}))$ (resp. $S_{kt,J_F}(K_1(\ideal{n}))$) of holomorphic Hilbert modular (resp. cusp) forms of weight $kt$ and of level $K_1(\ideal{n})$ defined in \cite[Definition 1.2]{Dim2}. 
Let $\chi$ be a Hecke character of $F$ of type $-(k-2)t$ whose conductor divides $\mathfrak{n}$. 
Let $M_k(\ideal{n},\chi,\C)$ (resp. $S_k(\ideal{n},\chi,\C)$) denote the subspace $G_{kt,J_F}(K_1(\ideal{n}),\chi)$ (resp. $S_{kt,J_F}(K_1(\ideal{n}),\chi)$) of $G_{kt,J_F}(K_1(\ideal{n}))$ (resp. $S_{kt,J_F}(K_1(\ideal{n}))$) of elements with character $\chi$ defined in \cite[Definition 1.3]{Dim2}.

For a fractional ideal $\ideal{a}$ of $F$, 
let $M_{k}(\Gamma_{1}(\ideal{a},\mathfrak{n}),\C)$ (resp. $S_{k}(\Gamma_{1}(\ideal{a},\mathfrak{n}),\C)$) denote the space $G_{kt,J_F}(\Gamma_{1}(\ideal{a},\mathfrak{n});\C)$ (resp. $S_{kt,J_F}(\Gamma_{1}(\ideal{a},\mathfrak{n});\C)$) of holomorphic Hilbert modular (resp. cusp) forms of weight $kt$ and of level $\Gamma_{1}(\ideal{a},\mathfrak{n})$ defined in \cite[Definition 1.4]{Dim2}.

Then we have canonical isomorphisms (cf. \cite[p.323]{Hida91} and \cite[(2.6a)]{Hida88}):
\begin{align}\label{isomorphism between spaces of HMF}
M_{k}(\mathfrak{n},\C) \simeq \bigoplus_{1\le i \le h_F^+}M_{k}(\Gamma_{1}(\ideal{d}_F [t_i],\mathfrak{n}),\C),\ \ \ \ \ \ 
S_{k}(\mathfrak{n},\C) \simeq \bigoplus_{1\le i \le h_F^+}S_{k}(\Gamma_{1}(\ideal{d}_F [t_i],\mathfrak{n}),\C).
\end{align}

%
\subsection{Hecke operators on analytic Hilbert modular forms}\label{Hecke}
%

Let $\ideal{n}$ be a non-zero ideal of $\ideal{o}_F$. 
In this subsection, we recall the definition of the Hecke operators acting on $M_k(\mathfrak{n},\C)$ and $S_k(\mathfrak{n},\C)$, following \cite[\S 1.10]{Dim2}. 

Let $\Delta(\mathfrak{n})$ be the following semigroup: 
\begin{align*}
\Delta(\mathfrak{n})&=
G(\A_f)\cap 
\left\{\begin{pmatrix}a&b\\c&d\\\end{pmatrix} \in M_2(\widehat{\mathfrak{o}}_F) \bigg\vert c \in \mathfrak{n}\widehat{\mathfrak{o}}_F,\  d \in \ideal{o}_{F_{\ideal{q}}}^{\times}\ \text{whenever}\ \ideal{q} \vert \mathfrak{n}  \right\},
\end{align*}
where $\ideal{q}$ is a non-zero prime ideal of $\ideal{o}_F$.
For $y \in \Delta(\mathfrak{n})$, we define the action of the double coset $K_1(\mathfrak{n}) y K_1(\mathfrak{n})$ on $M_k(\ideal{n},\C)$ (resp. $S_k(\ideal{n},\C)$) by 
\begin{align}\label{Hecke op analytic}
\textbf{f}\vert [K_1(\mathfrak{n}) y K_1(\mathfrak{n})](x)=\sum_{i} \textbf{f}(xy_i^{-1}),
\end{align}
where $K_1(\mathfrak{n}) y K_1(\mathfrak{n})=\coprod_i K_1(\mathfrak{n}) y_i$. 
By the definition of $M_k(\ideal{n},\C)$ and $S_k(\ideal{n},\C)$, the right\nobreakdash-hand side is independent of  the choice of the representative set $\{y_i\}_i$. 

We define the Hecke operator $T(\varpi_{\ideal{q}}^e)$ (resp. $S(\varpi_{\ideal{q}}^e)$) for a non-negative integer $e$, a non-zero prime ideal $\ideal{q}$ of $\mathfrak{o}_F$ (resp. prime ideal $\ideal{q}$ of $\mathfrak{o}_F$ prime to $\ideal{n}$), and a uniformizer $\varpi_{\ideal{q}}$ of $\mathfrak{o}_{F_{{}_{\ideal{q}}}}$ by the action of the double coset
$K_1(\ideal{n})\begin{pmatrix}\varpi_{\ideal{q}}^e&0\\0&1\\\end{pmatrix} K_1(\ideal{n})$ (resp.
$K_1(\ideal{n})\begin{pmatrix}\varpi_{\ideal{q}}^e&0\\0&\varpi_{\ideal{q}}^e\end{pmatrix} K_1(\ideal{n})$). 
We note that these operators are independent of the choice of $\varpi_{\ideal{q}}$.
We put $T(\ideal{q}^e)=T(\varpi_{\ideal{q}}^e)$ and $S(\ideal{q}^e)=S(\varpi_{\ideal{q}}^e)$ (resp. $U(\ideal{q}^e)=T(\varpi_{\ideal{q}}^e)$) for a non-negative integer $e$ and a non-zero prime ideal $\ideal{q}$ prime to $\ideal{n}$ 
(resp. prime ideal $\ideal{q}$ dividing $\ideal{n}$). 
We define 
$T(\ideal{m})=\prod_{\ideal{q}\nmid\ideal{n}}T(\ideal{q}^{e(\ideal{q})})$ and 
$S(\ideal{m})=\prod_{\ideal{q}\nmid\ideal{n}}S(\ideal{q}^{e(\ideal{q})})$
for all non-zero ideal 
$\ideal{m}=\prod_{\ideal{q}\nmid\ideal{n}}\ideal{q}^{e(\ideal{q})}$ of $\ideal{o}_F$ prime to $\ideal{n}$ and 
$U(\ideal{m})=\prod_{\ideal{q}|\ideal{n}}U(\ideal{q}^{e(\ideal{q})})$
for all non-zero ideal $\ideal{m}=\prod_{\ideal{q}|\ideal{n}}\ideal{q}^{e(\ideal{q})}$ of $\ideal{o}_F$, where $\ideal{q}$ is a non-zero prime ideal. 

The definition of the Hecke operators acting on $M_k(\Gamma_1(\ideal{a},\ideal{n}),\C)$ and $S_k(\Gamma_1(\ideal{a},\ideal{n}),\C)$ and their relation (via (\ref{isomorphism between spaces of HMF})) to the adelic ones recalled above are explicitly given in \cite[\S 2]{Shi}. 

For a subalgebra $A$ of $\C$, let 
$\mathbb{H}_k(\mathfrak{n},A)$ 
(resp. $\mathscr{H}_k(\mathfrak{n},A)$) be the commutative $A$-subalgebra of $\End_{\C}(M_k(\ideal{n},\C))$ (resp. $\End_{\C}(S_k(\ideal{n},\C))$) generated by $T(\ideal{m})$, $S(\ideal{m})$ 
for all non-zero ideal 
$\ideal{m}=\prod_{\ideal{q}\nmid\ideal{n}}\ideal{q}^{e(\ideal{q})}$ of $\ideal{o}_F$ prime to $\ideal{n}$ and 
$U(\ideal{m})$
for all non-zero ideal $\ideal{m}=\prod_{\ideal{q}|\ideal{n}}\ideal{q}^{e(\ideal{q})}$ of $\ideal{o}_F$. 

%
\subsection{Dirichlet series associated to Hilbert modular forms}\label{Diri}
%

In this subsection, we recall the definition and properties of the Dirichlet series associated to a Hilbert modular form, following \cite[\S 2]{Shi}.

Let $\textbf{h} \in M_k(\ideal{n},\C)$ and $h_i\in M_k(\Gamma_1(\ideal{d}_F[t_i],\ideal{n}),\C)$ such that $\textbf{h}=(h_i)_{1\le i \le h_F^+}$ under the isomorphism (\ref{isomorphism between spaces of HMF}). 
Then $h_i$ has the Fourier expansion of the form 
\begin{align}\label{i-th Fourier exp}
h_i(z)
=c_{\infty}([t_i]^{-1},\textbf{h})N([t_i])^{k/2}
+\sum_{0\ll \xi \in [t_i]}c(\xi[t_i]^{-1},\bold{h})N(\xi)^{k/2}e_F(\xi z)
\end{align}
given by \cite[(2.18)]{Shi} and \cite[Proposition 4.1]{Hida88}. 
Here the notion $\gg 0$ means totally positive, $\ideal{m} \mapsto c(\ideal{m},\textbf{h})$ is a function on the set of all fractional ideals of $F$ vanishing outside the set of integral ideals, and 
$e_F$ denotes the additive character of $F\backslash \A_F$ characterized by $e_F(x_{\infty})=\exp(2\pi \sqrt{-1}x_{\infty})$ for $x_{\infty}\in \A_{F,\infty}$.
We put 
\[
\text{$a_{\infty}(0,h_i)=c_{\infty}([t_i]^{-1},\textbf{h})N([t_i])^{k/2}$ and $a_{\infty}(\xi,h_i)=c(\xi[t_i]^{-1},\bold{h})N(\xi)^{k/2}$}
\]
for any $0\ll \xi\in [t_i]$. 
We also put 
\begin{align}
\label{Hecke constant}
C_{\infty,i}(0,\textbf{h})&=N([t_i])^{-k/2}a_{\infty}(0,h_i),\\
\label{Hecke eigenvalue}
C(\mathfrak{m},\textbf{h})&=N(\ideal{m})^{k/2}c(\mathfrak{m},\textbf{h})
\end{align}
for all non-zero ideals $\ideal{m}$ of $\ideal{o}_F$. 
Let $\eta$ be a finite Hecke character of $F$.
The Dirichlet series in the sense of Shimura \cite[(2.25)]{Shi} is defined by 
\begin{align}\label{L-function}
\sum_{\mathfrak{m}}C(\mathfrak{m},\textbf{h})
\eta(\mathfrak{m})\mathrm{N}(\mathfrak{m})^{-s} \ \ \text{for}\ \ s\in \C,
\end{align}
where $\mathfrak{m}$ runs over the set of all non-zero ideals of $\ideal{o}_F$. 
It converges absolutely if $\mathrm{Re}(s)$ is sufficiently large and extends to a meromorphic function on the complex plane (see, for example, \S \ref{twist Mellintrans} in this paper). 
For each $\textbf{h}\in M_{k}(\ideal{n},\C)$, let $D(s,\textbf{h},\eta)$ denote this meromorphic function. 
If $\eta$ is the trivial character, we simply write $D(s,\textbf{h})$ for $D(s,\textbf{h},\eta)$.

%
\subsection{Duality theorem between Hecke algebras and Hilbert modular forms}\label{Duality theorem}
%

Recall that, for $\textbf{h}=(h_i)_{1 \le i \le h_F^+} \in M_k(\ideal{n},\C)$, $h_i \in M_k(\Gamma_1(\ideal{d}_F[t_i],\ideal{n}),\C)$ has the Fourier expansion of the form (\ref{i-th Fourier exp}). 
For a subring $A$ of $\C$, we put
\begin{align*}
M_k(\Gamma_{1}(\ideal{d}_F [t_i], \mathfrak{n}),A)
&=M_k(\Gamma_{1}(\ideal{d}_F [t_i], \mathfrak{n}),\C) \cap A[[e_F(\xi z):\text{$\xi=0$ or $0\ll \xi \in F$}]],\\
S_k(\Gamma_{1}(\ideal{d}_F [t_i], \mathfrak{n}),A)
&=S_k(\Gamma_{1}(\ideal{d}_F [t_i], \mathfrak{n}),\C) \cap A[[e_F(\xi z):\text{$\xi=0$ or $0\ll \xi \in F$}]],
\end{align*}
\begin{align}\label{integral MF}
&M_k(\ideal{n},A)=\bigoplus_{1\le i \le h_F^+} M_k(\Gamma_{1}(\ideal{d}_F [t_i], \mathfrak{n}),A),&
&S_k(\ideal{n},A)=\bigoplus_{1 \le i\le h_F^+} S_k(\Gamma_{1}(\ideal{d}_F [t_i], \mathfrak{n}),A).&
\end{align}

By \cite[Theorem 4.11]{Hida88} and \cite[Theorem 2.2 (ii)]{Hida91}, 
the space $S_k(\ideal{n},A)$ (resp. $M_k(\ideal{n},A)$) is stable under the action of $\mathscr{H}_k(\ideal{n},A)$ (resp. $\mathbb{H}_k(\ideal{n},A)$). 
\begin{thm}[Duality theorem]\label{Duality}
Assume that $p$ is prime to the discriminant $\Delta_F$ of $F$. 
Let $K$ be a finite extension of the field $\Phi_p$ defined in Proposition \ref{const}, $\integer{}$ its ring of integers, and $\kappa$ the residue field. 
Assume that $\kappa$ contains the residue fields for all primes $\ideal{p}$ of $\ideal{o}_F$ over $p$. 
Then, for $A=K$ or $\integer{}$, the following $A$-bilinear map is a perfect pairing:
\[
\langle\ ,\ \rangle : \mathbb{H}_2(\ideal{n},A) \times M_2(\ideal{n},A) \to A ; (t,\mathbf{f}) \mapsto C(\mathfrak{o}_F,\mathbf{f} \vert t).
\]
\end{thm}
\begin{proof}
The duality theorem between $\mathscr{H}_2(\ideal{n},A)$ and $S_2(\ideal{n},A)$ is well-known (\cite[Theorem 5.1]{Hida88}).
We follow the arguments in the proof of \cite[Theorem 5.1]{Hida88} and \cite[Theorem 2.2 (iii)]{Hida91}. 
In the case $A=K$, the proof is the same as that of \cite[Theorem 2.2 (iii)]{Hida91}.

Suppose that $A=\integer{}$. 
It suffices to prove that the $\integer{}$-linear homomorphism 
$M_2(\ideal{n},\integer{})
\to
\Hom_{\integer{}}\left(\mathbb{H}_2(\ideal{n},\integer{}),\integer{} \right)$ 
induced by the pairing is an isomorphism.
If $\phi : \mathbb{H}_2(\ideal{n},\integer{}) \to \integer{}$ is an $\integer{}$\nobreakdash-linear map, then we can extend it to a $K$-linear map 
$\phi : \mathbb{H}_2(\ideal{n},K) \to K$. 
Thus, by the duality theorem for a field $K$, 
we get $\textbf{f}\in M_2(\ideal{n},K)$ such that 
$\langle t,\textbf{f}\rangle =\phi(t)$ for all $t\in \mathbb{H}_2(\ideal{n},\integer{})$. 
For a non-zero ideal $\ideal{m}=\prod_{\ideal{q}\nmid \ideal{n}}\ideal{q}^{e(\ideal{q})}\prod_{\ideal{q}\mid \ideal{n}}\ideal{q}^{e(\ideal{q})}$ of $\ideal{o}_F$, we put $V(\ideal{m})=\prod_{\ideal{q}\nmid \ideal{n}}T(\ideal{q})^{e(\ideal{q})}\prod_{\ideal{q}\mid \ideal{n}}U(\ideal{q})^{e(\ideal{q})}$.
Then we have 
$C(\ideal{m},\textbf{f})
=C(\mathfrak{o}_F,\textbf{f}\vert  V(\ideal{m}))
=\langle V(\ideal{m}),\textbf{f} \rangle
=\phi(V(\ideal{m})) \in \integer{}$. 
Here the first equality follows from \cite[(2.20)]{Shi}.
Suppose that the constant term of $\textbf{f}$ does not belong to $\integer{}$, that is, $a_{\infty}(0,f_i) \notin \integer{}$ for some $i$. 
Let $r \in \Z$ be the positive integer such that $\varpi^r a_{\infty}(0,f_i) \in \integer{}^{\times}$. 
Then the $q$-expansion of $\varpi^r f_i$ is equal to $\varpi^r a_{\infty}(0,f_i)$ modulo $\varpi$. 
By \cite{An--Go}, 
the kernel of the $q$-expansion map 
on the sum of the spaces of Hilbert modular forms with coefficients in $\overline{\mathbb{F}}_p$ of parallel weight $k$ $(k\in \Z,k\ge 0)$ defined in Definition \ref{definition:GHMF} is generated by $H_{p-1}-1$, 
where $H_{p-1}$ is the Hasse invariant, which is a Hilbert modular form of level $1$ and of parallel weight $p-1$. 
Therefore we have 
$\varpi^r f_i-\varpi^r a_{\infty}(0,f_i)\ (\bmod \varpi)=\alpha (H_{p-1}-1)$ for some $\alpha \in \integer{}/\varpi$. 
This is a contradiction
because the weight of $H_{p-1}$ is $(p-1)t>2t$. 
\end{proof}

%
\subsection{Hilbert Eisenstein series}
%

In this subsection, we recall the definition and properties of the Hilbert Eisenstein series, following \cite[$\S$3]{Shi}. 

Let $\varphi$ (resp. $\psi$) be a narrow ray class character of $F$ (cf. Notation), whose conductor is denoted by $\ideal{m}_{\varphi}$ (resp. $\ideal{m}_{\psi}$). 
Let $q$ (resp. $r$) $\in (\Z/2\Z)^{J_F}$ be the sign of $\varphi$ (resp. $\psi$). 
We may regard $\varphi$ (resp. $\psi$) as a function on the set of all non-zero ideals of $\ideal{o}_F$ by defining $\varphi(\ideal{m})=0$ (resp. $\psi(\ideal{m})=0$) if $\ideal{m}$ is not prime to $\ideal{m}_{\varphi}$ (resp. $\ideal{m}_{\psi}$). 
Then a function $\sgn(x)^r \psi(x\mathfrak{h}^{-1})$ of $x\in \mathfrak{h}$ depends only on $x$ modulo $\ideal{m}_{\psi}\ideal{h}$ for a fractional ideal $\mathfrak{h}$ of $F$. 

Let $\tau(\psi)$ be the Gauss sum attached to $\psi$ defined by 
\begin{align}\label{Gauss}
\tau(\psi)=\sum_{x\in \mathfrak{m}_{\psi}^{-1}\mathfrak{d}_F^{-1}/\mathfrak{d}_F^{-1}}
\sgn(x)^r \psi(x\mathfrak{m}_{\psi}\mathfrak{d}_F)e_F(x). 
\end{align}

The following is obtained by \cite[Proposition 3.4]{Shi} and \cite[Proposition 2.1]{Da--Da--Po}:
\begin{prop}\label{Hilbert Eisenstein}
Let $k$ be an integer $\ge 2$ such that $(k,\cdots,k)\equiv q+r \ (\mod (2\Z)^{J_F})$. 
Then there exists 
$\mathbf{E}_k(\varphi,\psi)=(E_k(\varphi,\psi)_i)_{1\le i\le h_F^+}\in M_{k}(\ideal{m}_{\varphi}\ideal{m}_{\psi},\varphi\psi,\C)$, called a Hilbert Eisenstein series, satisfying the following properties. 

$(1)$ $D(s,\mathbf{E}_k(\varphi,\psi))=L(s,\varphi)L(s-k+1,\psi).$

$(2)$ $C(\ideal{m},\mathbf{E}_k(\varphi,\psi))=\sum_{\ideal{c}\mid\ideal{m}}\varphi\left(\frac{\ideal{m}}{\ideal{c}}\right)\psi(\ideal{c})N(\ideal{c})^{k-1}$ 
for each integral ideal $\ideal{m}$ of $F$. 

$(3)$ If $\ideal{m}_{\varphi}\neq \ideal{o}_F$, then 
$C_{\infty,i}(0,\mathbf{E}_k(\varphi,\psi))=0$. 
If $\ideal{m}_{\varphi}=\ideal{o}_F$, then 
\[
C_{\infty,i}(0,\mathbf{E}_k(\varphi,\psi))
=2^{-n}\varphi^{-1}([t_{i}])
L(1-k,\varphi^{-1}\psi).
\]
\end{prop}

\begin{prop}\label{Const of Eisenstein}
Assume that $[F:\Q]>1$, $h_F^+=1$, and $\ideal{d}_F[t_1]=\mathfrak{o}_F$. 
Under the same notation and assumptions as Proposition \ref{Hilbert Eisenstein}, the constant term $a_c(0,E_k(\varphi,\psi)_1)$ of $\mathbf{E}_k(\varphi,\psi)=E_k(\varphi,\psi)_1$ at a cusp $c\in \mathbb{P}^1(F)$ is given by the following$:$ 
fix $\alpha=\begin{pmatrix}x&\beta\\y&\delta\end{pmatrix}\in \SL_2(\mathfrak{o}_F)$ such that $c=\alpha(\infty)$. 
If $y\notin \ideal{m}_{\psi}$ and $\psi\neq \textbf{1}$, then 
$a_c(0,E_k(\varphi,\psi)_1)=0$.  
If $y\in \ideal{m}_{\psi}$ or $\psi=\textbf{1}$, then
\begin{align*}
a_c(0,E_k(\varphi,\psi)_1)
&=\frac{N(\ideal{d}_F)^{-k/2}}{2^n}\frac{\tau(\varphi\psi^{-1})}{\tau(\psi^{-1})}
\left(\frac{N(\ideal{m}_{\psi})}{N(\ideal{m}_{\varphi\psi^{-1}})}\right)^{k}
\sgn(-y)^q\varphi(-y\ideal{m}_{\psi}^{-1})\sgn(-x)^r \psi^{-1}(-x)\\
&\ \ \ \ \times 
\left(\prod_{\ideal{q}|\ideal{m}_{\varphi}\ideal{m}_{\psi},\ideal{q}\nmid \ideal{m}_{\varphi\psi^{-1}}}
(1-\varphi\psi^{-1}(\ideal{q})N(\ideal{q})^{-k})\right)
L(1-k,\varphi^{-1}\psi). 
\end{align*}
\end{prop}
\begin{rem}
The assumption $h_F^+=1$ allows us to simplify some computations.
In the general case, the constant terms at all cusps are computed by T. Ozawa \cite{Oza}.
\end{rem}
\begin{proof}
We follow the arguments in the proof of \cite[Proposition 2.1]{Da--Da--Po} and \cite[Chapter III, Theorem 4.9]{Fre}. 
We put $\ideal{a}=\ideal{m}_{\varphi}$ and $\ideal{b}=\ideal{m}_{\psi}$. 
First, we recall the construction of the Eisenstein series $\textbf{E}_k(\varphi,\psi)$ given in \cite[$\S$3]{Shi} and \cite[Proposition 2.1]{Da--Da--Po}. 
Let $U$ be the subgroup $\{u\in \mathfrak{o}_F^{\times} \mid N(u)^k=1, u\equiv 1 \ \mod{\mathfrak{ab}}\}$ of $\mathfrak{o}_F^{\times}$, which is of finite index. 
For $z \in \mathfrak{H}^{J_F}$ and $s\in \C$ with $\mathrm{Re}(2s+k)>2$, we define
\begin{align*}
E_k(\varphi,\psi)_1(z,s)
=&N([t_1])^{1-k/2}[\mathfrak{o}_F^{\times}:U]^{-1}\Gamma(k)^nN(\mathfrak{b})^{-1}\tau(\psi)
\sum_{\mathfrak{h}\in \mathrm{Cl}_F}
\sum_{a \in \mathfrak{h}/\mathfrak{ah}}
\sum_{t\in \mathfrak{b}^{-1}\mathfrak{d}_F^{-1}[t_1]^{-1}\mathfrak{h}/\mathfrak{d}_F^{-1}[t_1]^{-1}\mathfrak{h}}\\
\nonumber&\times 
\sgn(a)^q \varphi(a\mathfrak{h}^{-1})\sgn(-t)^r\psi(-t\ideal{b}\ideal{d}_F[t_1]\mathfrak{h}^{-1})N(\mathfrak{h})^{k-1}\\
\nonumber&\times 
E_{k,U}(z,s;a,t;\mathfrak{ah},\mathfrak{d}_F^{-1}[t_1]^{-1}\mathfrak{h}),
\end{align*}
where $\mathrm{Cl}_F$ is the ideal class group of $F$ and 
\begin{align*}
&E_{k,U}(z,s;a,t;\mathfrak{ah},\mathfrak{d}_F^{-1}[t_1]^{-1}\mathfrak{h})\\
&=\Delta_F^{1/2}N(\mathfrak{d}_F^{-1}[t_1]^{-1}\mathfrak{h})(-1)^{kn}(2\pi \sqrt{-1})^{-kn}\sum_{(a',b')\in R,(a',b')\neq (0,0)}(a'z+b')^{-k}|a'z+b'|^{-2s}. 
\end{align*}
Here the last sum runs over a complete set $R$ of representatives for the quotient of $(a+ \mathfrak{ah}) \times (t+\mathfrak{d}_F^{-1}[t_1]^{-1}\mathfrak{h})$ by the diagonal multiplication of $U$. 
This series converges if $\mathrm{Re}(2s+k)>2$ and is analytically continued to a holomorphic function on the whole complex plane (\cite[p.656]{Shi}). 
The Eisenstein series $E_k(\varphi,\psi)_1(z)$ is defined to be $E_k(\varphi,\psi)_1(z,0)$, which is a holomorphic function of  $z\in \mathfrak{H}^{J_F}$ (\cite[p.656]{Shi}).
For $z \in \mathfrak{H}^{J_F}$, we have
\begin{align*}
E&_{k,U}(z,s;a,t;\mathfrak{ah},\mathfrak{d}_F^{-1}[t_1]^{-1}\mathfrak{h})|\alpha=E_{k,U}(\alpha z,s;a,t;\mathfrak{ah},\mathfrak{d}_F^{-1}[t_1]^{-1}\mathfrak{h})(yz+\delta)^{-k}\\
&=\Delta_F^{1/2}N(\ideal{h})(-2\pi \sqrt{-1})^{-kn}
\sum_{(a',b')\in R,(a',b')\neq (0,0)}(a'\alpha z+b')^{-k}(yz+\delta)^{-k}|a'\alpha z+b'|^{-2s}\\
&=\Delta_F^{1/2}N(\ideal{h})(-2\pi \sqrt{-1})^{-kn}
\sum_{(a',b')\in R,(a',b')\neq (0,0)}((a'x+b'y)z+(a'\beta+b'\delta))^{-k}|a'\alpha z+b'|^{-2s}.
\end{align*}
Note that only the terms for $(a',b')$ with $a'x+b'y=0$ in the series contribute to the constant term of $E_k(\varphi,\psi)_1|\alpha$ at $\infty$. 
We put $C=\Delta_F^{1/2}\Gamma(k)^n[\mathfrak{o}_F^{\times}:U]^{-1}N(\ideal{d}_F)^{-1}(-2\pi\sqrt{-1})^{-kn}$. 

(1) First suppose that $y\notin \ideal{b}$. 
Since $\ideal{d}_F[t_1]=\mathfrak{o}_F$ and $b'y=-a'x \in (y)\ideal{b}^{-1}\ideal{h}\cap \ideal{h}$, we see that $b'\ideal{b}\mathfrak{h}^{-1}$ is not prime to $\ideal{b}$ and hence $\sgn(-b')^r\psi^{-1}(-b'\ideal{b}\mathfrak{h}^{-1})=0$ if $\ideal{b}\neq 1$. 
Thus the constant term $a_c(0,E_k(\varphi,\psi)_1)$ is equal to $0$ if $\ideal{b}\neq 1$. 

Consider the case $\ideal{b}=1$. 
The constant term of $E_k(\varphi,\psi)_1|\alpha$ at $\infty$ is equal to the value of
\begin{align}\label{const b=1}
C\cdot N([t_1])^{-k/2}
\sum_{\ideal{h}\in \text{Cl}_F}
\sum_{\substack{(a',b')\in R,(a',b')\neq (0,0)\\ a'x+b'y=0}}
\sgn(a')^q \varphi(a'\ideal{h}^{-1})N(\ideal{h})^k
(a'\beta+b'\delta)^{-k-2s}
\end{align}
at $s=0$. 
We note that the map $(a',b') \mapsto a'\beta+b'\delta$ from the set of pairs $(a', b')$ in (\ref{const b=1}) to $\ideal{h}-\{0\}$ is bijective. 
Indeed, the inverse map is given by $d\mapsto (-dy,dx)$. 
Thus the value of (\ref{const b=1}) st $s=0$ is equal to the value of
\begin{align}\label{const x nonzero}
C\cdot N([t_1])^{-k/2}
\sum_{\ideal{h}\in \text{Cl}_F}
\sum_{\substack{d\in R',d\neq 0}}
\sgn(-dy)^q \varphi(-dy\ideal{h}^{-1})N(\ideal{h})^kd^{-k-2s}
\end{align}
at $s=0$. 
Here the last sum runs over a complete set $R'$ of representatives for the quotient of $\ideal{h}$ by the multiplication of $U$.
Since the map $(\ideal{h},d)\mapsto d\ideal{h}^{-1}$ from the set of pairs $(\ideal{h},d)$ in (\ref{const x nonzero}) to the set of all non-zero ideals of $\ideal{o}_F$ is a surjective $[\mathfrak{o}_F^{\times}:U]$-to-$1$ map, the value of 
(\ref{const x nonzero}) at $s=0$ is equal to 
\begin{align*}
C\cdot N([t_1])^{-k/2}
\sgn(-y)^q\varphi(-y)
[\mathfrak{o}_F^{\times}:U]L(k,\varphi). 
\end{align*}
Therefore, the functional equation for the Hecke $L$-functions (see, for example, \cite[Theorem 3.3.1]{Mi}) implies that the constant term $a_{c}(0,E_k(\varphi,\psi)_1)$ is equal to
\begin{align*}
a_{x/y}(0,E_k(\varphi,\psi)_1)
&=\frac{N(\ideal{d}_F)^{-k/2}}{2^n}\tau(\varphi)
N(\ideal{m}_{\varphi})^{-k}
\sgn(-y)^q\varphi(-y)
L(1-k,\varphi^{-1}).
\end{align*}

(2) Next suppose that $y\in \ideal{b}$. 
The constant term of $E_k(\varphi,\psi)_1|\alpha$ is equal to the value of 
\begin{align}\label{const y in b'}
C&\cdot N([t_1])^{-k/2}N(\mathfrak{b})^{-1}\tau(\psi)
\sum_{\mathfrak{h}\in \mathrm{Cl}_F}N(\ideal{h})^k
\sum_{\substack{(a',b')\in R,(a',b')\neq (0,0) \\ a'x+b'y=0}}\\
\nonumber&\times \sgn(a')^q \varphi(a'\mathfrak{h}^{-1})
\sgn(-b')^r\psi^{-1}(-b'\ideal{b}\mathfrak{h}^{-1})
(a'\beta+b'\delta)^{-k-2s}
\end{align}
at $s=0$. 
We note that the map $(a',b')\mapsto a'\beta+b'\delta$ from the set of pairs $(a',b')$ in (\ref{const y in b'}) to $\ideal{b}^{-1}\ideal{h}-\{0\}$ is bijective. 
Indeed, the inverse map is given by $d \mapsto (-dy,dx)$. 
Thus the value of (\ref{const y in b'}) at $s=0$ is equal to the value of
\begin{align}\label{const y in b}
C&\cdot N([t_1])^{-k/2}N(\ideal{b})^{-1}\tau(\psi)
\sum_{\ideal{h}\in \text{Cl}_F}
\sum_{\substack{d\in R'',d\neq 0 }}\\
\nonumber&\times \sgn(-dy)^q \varphi(-dy\ideal{h}^{-1})
\sgn(-dx)^r \psi^{-1}(-dx\ideal{b}\ideal{h}^{-1})
N(\ideal{h})^k
N(d)^{-k-2s}
\end{align}
at $s=0$. 
Since the map $(\ideal{h},d)\mapsto d\ideal{b}\ideal{h}^{-1}$ from the set of pairs $(\ideal{h},d)$ in (\ref{const y in b}) to the set of all non-zero ideals of $\ideal{o}_F$ is a surjective $[\mathfrak{o}_F^{\times}:U]$-to-$1$ map, 
the value of (\ref{const y in b}) at $s=0$ is equal to 
\begin{align*}
C&\cdot N([t_1])^{-k/2}N(\ideal{b})^{-1}\tau(\psi)
\sgn(-y)^q\varphi(-y)\sgn(-x)^r\psi^{-1}(-x)\\
&\times\varphi(\ideal{b}^{-1})N(\ideal{b})^k
[\mathfrak{o}_F^{\times}:U]L(k,\varphi\psi^{-1})
\prod_{\ideal{q}|\ideal{m}_{\varphi}\ideal{m}_{\psi},\ideal{q}\nmid \ideal{m}_{\varphi\psi^{-1}}}
(1-\varphi\psi^{-1}(\ideal{q})N(\ideal{q})^{-k}). 
\end{align*}
Therefore, in the same way as above, our assertion follows from the functional equation for the Hecke $L$-functions. 
\end{proof}


%
\subsection{Geometric Hilbert modular varieties}\label{subsection:HMV}
%

In this subsection,
we fix notation concerning the integral models of Hilbert modular varieties and their compactifications, following \cite{Dim2} and \cite{Dim--Ti}. 

We fix a non-zero ideal $\mathfrak{n}$ of $\mathfrak{o}_F$. 
We put $\Delta=\mathrm{N}(\mathfrak{nd}_F)$. 
Let $\mu_{\mathfrak{n}}$ be the closed subscheme of $\mathbb{G}_m \otimes_{\Z} \mathfrak{d}_F^{-1}$ given by the $\ideal{n}$-torsion points of $\mathbb{G}_m \otimes_{\Z} \mathfrak{d}_F^{-1}$.
Let $\mathfrak{c}$ be a fractional ideal of $F$.
We consider the contravariant functor $\cal{F}_{\ideal{c}}^1$ (resp. $\cal{F}_{\ideal{c}}$) from the category of $\Z[1/\Delta]$\nobreakdash-schemes to the category of sets sending a scheme $S$ to the set of isomorphism classes of triples $((A,\iota),\lambda,\alpha)$ (resp. $((A,\iota),[\lambda],\alpha)$),
where $(A,\iota)$ is a Hilbert-Blumenthal abelian variety (HBAV for short) over $S$ endowed with a $\mathfrak{c}$-polarization $\lambda$ (resp. the $\mathfrak{o}_{F,+}^{\times}$-orbit $[\lambda]$ of a $\mathfrak{c}$-polarization $\lambda$) and a $\mu_{\mathfrak{n}}$\nobreakdash-level structure $\alpha$ 
(for the definitions, see, for example, \cite[\S 1.3]{Dim2}). 

Throughout the paper, we assume that 
\begin{align}\label{tor-free}
\text{$(\ideal{n},6\Delta_F)=1$ and the quotient of the group $\Gamma_1(\mathfrak{c},\mathfrak{n})$ by its center is torsion--free.}
\end{align} 

Then the functor $\cal{F}_{\ideal{c}}^1$ is representable by a quasi-projective, smooth, geometrically connected $\Z[1/\Delta]$-scheme $M_{\mathfrak{c}}^1=M(\Gamma_1^1(\mathfrak{c},\mathfrak{n}))$ of relative dimension $n=[F:\Q]$ (\cite[Theorem 4.1]{Dim--Ti} and \cite[Lemma 2.1]{Dim3}). 

The functor $\cal{F}_{\ideal{c}}$ has a coarse moduli scheme $M_{\mathfrak{c}}=M(\Gamma_1(\mathfrak{c},\mathfrak{n}))$, which is the quotient of $M_{\ideal{c}}^1$ by $\mathfrak{o}_{F,+}^{\times}/\mathfrak{o}_{F,\mathfrak{n}}^{\times 2}$ (\cite[Corollary 4.2]{Dim--Ti}).
Here $\mathfrak{o}_{F,+}^{\times}$ denotes the subgroup of $ \mathfrak{o}_F^{\times}$ consisting of totally positive units, $\mathfrak{o}_{F,\mathfrak{n}}^{\times}$ denotes the subgroup of $\mathfrak{o}_F^{\times}$ consisting of elements congruent to $1$ modulo $\mathfrak{n}$, and 
the finite group $\mathfrak{o}_{F,+}^{\times}/\mathfrak{o}_{F,\mathfrak{n}}^{\times 2}$ acts on $M_{1,\mathfrak{c}}$ by $[\varepsilon] \cdot ((A,\iota),\lambda,\alpha)=((A,\iota),\iota(\varepsilon)\circ\lambda,\alpha)$ for $[\varepsilon] \in \mathfrak{o}_{F,+}^{\times}/\mathfrak{o}_{F,\mathfrak{n}}^{\times 2}$. 
Also, $M_{\mathfrak{c}}$ is a quasi-projective, smooth, geometrically connected $\Z[1/\Delta]$\nobreakdash-scheme of relative dimension $n=[F:\Q]$. 
We put 
\begin{align*}
M^1=\coprod_{1\le i \le h_F^+} M_{[t_i]}^1, \ \ \ \ 
M=\coprod_{1 \le i \le h_F^+} M_{[t_i]}, 
\end{align*}
where $\{[t_i]\}_{1\le i\le h_F^+}$ is the complete set of representatives of $\mathrm{Cl}_F^+$ fixed in \S \ref{Analytic HMV}.  

Let $M_{\mathfrak{c}}^{1,\mathrm{tor}}$ (resp. $M_{\mathfrak{c}}^{\mathrm{tor}}$) denote the toroidai compactification of $M_{\mathfrak{c}}^1$ (resp. $M_{\mathfrak{c}}$) constructed in \cite{Dim}. 
The scheme $M_{\mathfrak{c}}^{1,\mathrm{tor}}$ (resp. $M_{\mathfrak{c}}^{\mathrm{tor}}$) is smooth and proper over $\Z[1/\Delta]$ and 
the boundary $M_{\mathfrak{c}}^{1,\mathrm{tor}}-M_{\mathfrak{c}}^1$ (resp. $M_{\mathfrak{c}}^{\mathrm{tor}}-M_{\mathfrak{c}}$) is a relative simple normal crossing divisor of $M_{\mathfrak{c}}^{1,\mathrm{tor}}$ (resp. $M_{\mathfrak{c}}^{\mathrm{tor}}$) (\cite[Theorem 7.2]{Dim}). 
We put 
\begin{align*}
M^{1,\mathrm{tor}}=\coprod_{1 \le i \le h_F^+} M_{[t_i]}^{1,\mathrm{tor}}, \ \ \ \ 
M^{\mathrm{tor}}=\coprod_{1 \le i \le h_F^+} M_{[t_i]}^{\mathrm{tor}}. 
\end{align*}

%
\subsection{Geometric Hilbert modular forms}\label{subsection:Geometric HMF}
%

In this subsection, we recall the definition of the geometric Hilbert modular form, following \cite[\S 1]{Dim2} and \cite[\S 2]{Ti--Xi}.

We keep the notation in \S \ref{subsection:HMV}.
Throughout the paper, we assume that 
\begin{align}\label{prime to p}
\text{for each $i$, $[t_i]$ is prime to $p$. }
\end{align}
Let $\pi:\mathcal{A} \to M_{[t_i]}^1$ denote the universal HBAV. 
The morphism $\pi$ extends to a morphism of semi\nobreakdash-abelian schemes 
$\overline{\pi}:\mathcal{G} \to M^{1,\mathrm{tor}}_{[t_i]}$ 
(\cite[Theorem 6.4]{Dim--Ti}). 
Let $\cal{A}^{\text{tor}}$ denote the toroidal compactification of $\cal{A}$ constructed in \cite{Dim--Ti}.
Let $\underline{\omega}_{1,[t_i]}$ (resp. $\cal{H}_{1,[t_i]}^1$) denote the sheaf $\underline{\omega}_{\cal{G}/M^{1,\mathrm{tor}}_{[t_i]}}$ (resp. $\cal{H}_{\text{log-dR}}^1(\cal{A}^{\mathrm{tor}}/M^{1,\mathrm{tor}}_{[t_i]})$) defined in \cite[\S 1.9]{Dim2}, which is a locally free $\integer{}_{M^{1,\mathrm{tor}}_{[t_i]}}\otimes_{\Z}\mathfrak{o}_F$-module of rank $1$ (resp. rank $2$). 
We put $\underline{\nu}_{1,[t_i]}=\wedge^2_{\integer{}_{M^{1,\mathrm{tor}}_{[t_i]}}\otimes_{\Z}\mathfrak{o}_F}\cal{H}_{1,[t_i]}^1$.

Let $\widetilde{F}$ be the Galois closure of $F$ in $\overline{\Q}$, $F'$ the field generated by elements $\varepsilon^{t/2}$ for all $\varepsilon \in \mathfrak{o}_{F,+}^{\times}$ over $\widetilde{F}$, and $\mathfrak{o}_{F'}$ the ring of integers of $F'$. 
For a $\Z[1/\Delta]$-scheme $S$, 
we write $S_{\mathfrak{o}_{F'}}$ for the base change of $S$ to $\mathrm{Spec}(\mathfrak{o}_{F'}[1/\Delta])$. 
As explained in \cite[\S 1.6, \S 1.9]{Dim2}, 
$\underline{\omega}_{1,[t_i]}$ (resp. $\cal{H}_{1,[t_i]}^1$) descends to a locally free $\integer{}_{M_{{}_{[t_i],\mathfrak{o}_{F'}}}^{\mathrm{tor}}}\otimes_{\Z} \mathfrak{o}_F$\nobreakdash-module of rank $1$ (resp. rank $2$) over $M_{\ideal{c},\ideal{o}_{F'}}^{\mathrm{tor}}$, which is denoted by $\underline{\omega}_{[t_i]}$ (resp. $\cal{H}_{[t_i]}^1$). 
We put $\underline{\nu}_{[t_i]}=\wedge^2_{\integer{}_{M^{\mathrm{tor}}_{[t_i],\ideal{o}_{F'}}}\otimes_{\Z}\mathfrak{o}_F}\cal{H}_{[t_i]}^1$.

Let $D_{[t_i]}^1$ (resp $D_{[t_i]}$) denote the boundary $M_{[t_i]}^{1,\mathrm{tor}}-M_{[t_i]}^1$ (resp. $M_{[t_i]}^{\mathrm{tor}}-M_{[t_i]}$) of $M_{[t_i]}^{1,\mathrm{tor}}$ (resp. $M_{[t_i]}^{\mathrm{tor}}$). 
For $X=M_{[t_i]}^{1,\mathrm{tor}}$ (resp. $M_{[t_i]}^{\mathrm{tor}}$) and $D=D_{[t_i]}^1$ (resp $D_{[t_i]}$), let $(X,L)$ denote the log scheme 
in the sense of \cite[(1.5)(1)]{Kato}. 
For a $\Z[1/\Delta]$-algebra $R$ and a $\Z[1/\Delta]$-log scheme $(S,L)$, 
let $(\text{Spec}(R),\text{triv})$ denote the scheme $\text{Spec}(R)$ endowed with the trivial log structure and 
let $(S,L)_R$ denote the base change of $(S,L)$ to $(\text{Spec}(R),\text{triv})$. 
Let $\Omega_{X_R/R}^j(\mathrm{log}(D))$ denote the logarithmic differential module obtained by the log smooth morphism $(X,L)_R \rightarrow (\text{Spec}(\Z[1/\Delta]),\text{triv})_R$ in the sense of \cite[(1.7)]{Kato}. 

Let $k$ be an integer $\ge 2$. 
For $\underline{\omega}=\underline{\omega}_{1,[t_i]}$ (resp. $\underline{\omega}_{[t_i]}$) and $\underline{\nu}=\underline{\nu}_{1,[t_i]}$ (resp. $\underline{\nu}_{[t_i]}$), we put 
\[
\underline{\omega}^{(k-2)}=\underline{\omega}^{k-2}\otimes \underline{\nu}^{-(k-2)t/2}. 
\]
\begin{defn}\label{definition:GHMF}(\cite[\S1.5]{Dim2} and \cite[\S2.12]{Ti--Xi}).
Let $R$ be an $\mathfrak{o}_{F'}[1/\Delta]$-algebra and $k$ an integer $\ge 2$.
Let $D_{[t_i]}^1$ (resp $D_{[t_i]}$) denote the boundary $M_{[t_i]}^{1,\mathrm{tor}}-M_{[t_i]}^1$ (resp. $M_{[t_i]}^{\mathrm{tor}}-M_{[t_i]}$) of $M_{[t_i]}^{1,\mathrm{tor}}$ (resp. $M_{[t_i]}^{\mathrm{tor}}$). 
We define the spaces of Hilbert modular forms of weight $kt$ and of level 
$\Gamma_1^1(\ideal{d}_F[t_i],\mathfrak{n})$ and 
$\Gamma_1(\ideal{d}_F[t_i],\mathfrak{n})$ with coefficients in $R$ to be 
\begin{align*}
M_{k}(\Gamma_1^1(\ideal{d}_F[t_i],\mathfrak{n}),R)
&=H^0(M_{[t_i],R}^1, \underline{\omega}_{1,[t_i]}^{(k-2)}\otimes \Omega_{M^1_{[t_i],R}}^n),\\
M_{k}(\Gamma_1(\ideal{d}_F[t_i],\mathfrak{n}),R)
&=H^0(M_{[t_i],R}, \underline{\omega}_{[t_i]}^{(k-2)}\otimes \Omega_{M_{[t_i],R}}^n),
\end{align*}
respectively. 
If $F\neq\Q$, then, by the Koecher's principle, we have 
$M_{k}(\Gamma_1^1(\ideal{d}_F[t_i],\mathfrak{n}),R)
=H^0(M_{[t_i],R}^{1,\mathrm{tor}}, \underline{\omega}_{1,[t_i]}^{(k-2)}\otimes \Omega_{M_{[t_i],R}^{1,\mathrm{tor}}}^n(\mathrm{log}(D^1_{[t_i]})))$ \ and $M_{k}(\Gamma_1(\ideal{d}_F[t_i],\mathfrak{n}),R)=H^0(M_{[t_i],R}^{\mathrm{tor}},\underline{\omega}_{[t_i]}^{(k-2)}\otimes \Omega_{M_{[t_i],R}^{\mathrm{tor}}}^n(\mathrm{log}(D_{[t_i]})))$ (\cite[\S 2.12]{Ti--Xi}). 
We define the space of Hilbert cusp forms of weight $kt$ and of level 
$\Gamma_1^1(\ideal{d}_F[t_i],\mathfrak{n})$ and 
$\Gamma_1(\ideal{d}_F[t_i],\mathfrak{n})$ with coefficients in $R$ to be 
\begin{align*}
S_{k}(\Gamma_1^1(\ideal{d}_F[t_i],\mathfrak{n}),R)
&=\im \left(H^0(M_{[t_i],R}^{1,\mathrm{tor}}, \underline{\omega}_{1,[t_i]}^{(k-2)}\otimes \Omega_{M_{[t_i],R}^{1,\mathrm{tor}}}^n)\to H^0(M_{[t_i],R}^1, \underline{\omega}_{1,[t_i]}^{(k-2)}\otimes \Omega_{M_{[t_i],R}^1}^n)\right),\\
S_{k}(\Gamma_1(\ideal{d}_F[t_i],\mathfrak{n}),R)
&=\im \left(H^0(M_{[t_i],R}^{\mathrm{tor}}, \underline{\omega}_{[t_i]}^{(k-2)}\otimes \Omega_{M_{[t_i],R}^{\mathrm{tor}}}^n) \to H^0(M_{[t_i],R}, \underline{\omega}_{[t_i]}^{(k-2)}\otimes \Omega_{M_{[t_i],R}}^n)\right),
\end{align*}
respectively. 
We put 
\begin{align*}
M_{k}(M^1,R)&=\bigoplus_{1\le i \le h_F^+} M_{k}(\Gamma_1^1(\ideal{d}_F[t_i],\mathfrak{n}),R),\ \ 
M_{k}(M,R)=\bigoplus_{1\le i \le h_F^+} M_{k}(\Gamma_1(\ideal{d}_F[t_i],\mathfrak{n}),R),\\
S_{k}(M^1,R)&=\bigoplus_{1\le i \le h_F^+} S_{k}(\Gamma_1^1(\ideal{d}_F[t_i],\mathfrak{n}),R),\ \ \ \ 
S_{k}(M,R)=\bigoplus_{1\le i \le h_F^+} S_{k}(\Gamma_1(\ideal{d}_F[t_i],\mathfrak{n}),R).
\end{align*}
\end{defn}

%
\subsection{Hecke correspondences}\label{subsection:Hecke correspondence}
%
%

In this subsection, we fix notation concerning the Hecke correspondence $T(\ideal{a})$ and $U({\ideal{a}})$, following \cite[\S 2.4]{Dim2} and \cite[\S 1.11]{Ki--La}.

Let $\ideal{a}$ be a non-zero ideal of $\ideal{o}_F$. 
We fix a pair $(i,j)$ such that $[t_i]\ideal{a}=[t_j]$ in $\text{Cl}_F^+$. 
We consider the functor $\cal{F}_{\ideal{a},i,j}^1$ from the category of $\Z[1/\Delta]$-schemes to the category of sets sending a scheme $S$ to the set of isomorphism classes of quintuples $((A,\iota),\lambda,\alpha,C,\beta)$.
Here $(A,\iota)$ is a HBAV over $S$, endowed with a $[t_i]$-polarization $\lambda$ and a $\mu_{\ideal{n}}$-level structure $\alpha$, $C$ is an $\mathfrak{o}_{F}$-stable closed subscheme of $A[\ideal{a}]$, which is disjoint from $\alpha(\mu_{\ideal{n}})$ and \'etale locally isomorphic to the constant group scheme $\mathfrak{o}_F/\ideal{a}$ over $\mathfrak{o}_F$, and $\beta$ is the $\mathfrak{o}_{F,\ideal{n}}^{\times,2}$-orbit of isomorphisms $([t_i]\ideal{a},([t_i]\ideal{a})_+)\simeq ([t_j],[t_j]_+)$, 
where $\ideal{c}_+=\ideal{c}\cap (F\otimes\mathbb{R})_+^{\times}$ is the totally positive cone for a fractional ideal $\ideal{c}$ of $F$.

The functor $\cal{F}_{\ideal{a},i,j}^1$ is representable by $M_{\ideal{a},i,j}^1$ constructed in \cite[\S1.9]{Ki--La}. 
Put $M_{\ideal{a}}^1=\prod_{1 \le i \le h_F^+}M_{\ideal{a},i,j}^1$. 
Then the two projections 
$\pi_1,\pi_2:M_{\ideal{a}}^1 \to M^1$ given in \cite[\S1.9]{Ki--La}
induce algebraic correspondences $T(\ideal{a})$ and $U(\ideal{a})$ on $M^1$. 
We define the Hecke correspondence $T(\ideal{a})$ and $U(\ideal{a})$ on $M^{1,\text{tor}}$ by taking a toroidal compactification of $M_{\ideal{a}}^1$ (see the proof of \cite[Corollary 2.7]{Dim2}).

Thus, we get an action of $T(\ideal{a})$ and $U(\ideal{a})$ on the spaces $M_{k}(M^1,R)$ and $S_{k}(M^1,R)$ (\cite[\S2.4]{Dim2} and \cite[\S1.11]{Ki--La}) and hence 
we obtain an action of $T(\ideal{a})$ and $U(\ideal{a})$ on $M_{k}(M,R)$ and $S_{k}(M,R)$ by using the projection 
$\frac{1}{[\mathfrak{o}_{F,+}^{\times}:\mathfrak{o}_{F,\ideal{n}}^{\times 2}]}\sum_{[\varepsilon]\in \mathfrak{o}_{F,+}^{\times}/\mathfrak{o}_{F,\ideal{n}}^{\times 2}}[\varepsilon] :
M_{k}(M^1,R) \to M_{k}(M,R)$. 
According to \cite[\S2.4]{Dim2} and \cite[\S1.11.8]{Ki--La}, 
this action over $\C$ coincides with the usual Hecke operator as (\ref{Hecke op analytic}).

\section{Mellin transform}\label{subsection:Mellin}
%
The purpose of this section is to give a cohomological description of special values of the $L$-functions defined in (\ref{L-function}) associated to a Hilbert modular form vanishing at cusps $\Gamma_0(\ideal{d}_F[t_1],\ideal{n})$-equivalent to the cusp $\infty$ (Proposition \ref{Modular symbol} and \ref{Modular symbol anti-hol}), where we use the assumption $h_F^+=1$.

%
\subsection{Borel--Serre compactification}\label{Borel--Serre}
%

In this subsection, we recall the Borel--Serre compactification of $Y_i$ defined by (\ref{analytic HMV}). 
For more detail, refer to \cite[\S 2.1]{Ha} and \cite[\S 1.8]{Hida93}. 

We fix an integer $i$ with $1\le i\le h_F^+$ and abbreviate $\Gamma_{1}(\ideal{d}_F [t_i], \mathfrak{n})$ to $\Gamma$. 

Let $(\mathfrak{H}^{J_F})^{\text{BS}}$ denote the Borel--Serre compactification of $\mathfrak{H}^{J_F}$, which is a locally compact manifold on which $\GL_2(F)$ acts. 
We can describe the boundary of $(\mathfrak{H}^{J_F})^{\text{BS}}$ at the cusp $\infty$ as follows. 
Put $X=\{(y,x)\in (F\otimes \mathbb{R})_+^{\times} \times (F\otimes \mathbb{R})
\mid \prod_{\iota\in J_F}y_{\iota}=1\}$. 
Then we have 
\begin{align}\label{H^{J_F}}
\mathfrak{H}^{J_F} \xrightarrow{\simeq} X \times \mathbb{R}_+^{\times}; (x_{\iota}+\sqrt{-1}y_{\iota})_{\iota\in J_F}
\mapsto \left(
\left(\left(\prod_{\iota\in J_F}y_{\iota}\right)^{-\frac{1}{n}}y_{\iota},x_{\iota}\right)_{\iota\in J_F},
\prod_{\iota\in J_F} y_{\iota}\right),
\end{align}
which is compatible with the action of $\Gamma_{\infty}$. 
Here $\Gamma_{\infty}$ denotes the stabilizer of $\infty$ in $\Gamma$, which acts trivially on the second factor of the right-hand side. 
The compactification of $\mathfrak{H}^{J_F}$ at the cusp $\infty$ is given by $X \times (\mathbb{R}_+^{\times}\cup \{\infty\})$. 

Let $Y_i^{\text{BS}}$ denote the Borel--Serre compactification $\overline{\Gamma}\backslash(\mathfrak{H}^{J_F})^{\text{BS}}$ of $Y_i$. 
Then $Y_i^{\text{BS}}$ is a compact manifold and its  boundary at a cusp $s$, which is denoted by $D_{s}$, is given by $\overline{\Gamma_{s}}\backslash \alpha(X \times \{\infty\})$, where $\Gamma_s$ denotes the stabilizer of $s$ in $\Gamma$ and $\alpha\in\SL_2(F)$ such that $s=\alpha(\infty)$.

%
\subsection{Fundamental domain}\label{Fundamental domain}
%

In this subsection, we construct a relative homology class, which is related to special values of the $L$-functions attached to a Hilbert modular form. 

We keep the notation in \S \ref{Borel--Serre}. 
Let $E$ be a subgroup of $\mathfrak{o}_{F,+}^{\times}$ of finite index and $\varepsilon_1,\cdots,\varepsilon_{n-1}$ a $\Z$-basis of $E$. 
We note that a fundamental domain of $(\mathbb{R}_{+}^{\times})^{J_F}/E$ is given by 
\begin{align*}
\Omega_E=\prod_{1 \le j \le n-1} \{\varepsilon_j^{r_j}\mid r_j\in [0,1)\}\times \mathbb{R}_+^{\times} 
\hookrightarrow X\times \mathbb{R}_+^{\times} 
\xrightarrow{\simeq} \mathfrak{H}^{J_F};\\
(\varepsilon_1^{r_1},\cdots,\varepsilon_{n-1}^{r_{n-1}},-\mathrm{log}(r_n))
\mapsto 
((\varepsilon^{r},0),-\mathrm{log}(r_n))
\mapsto 
\sqrt{-1}y^{\frac{1}{n}}
\varepsilon^r,
\end{align*}
where 
$r=(r_1,\cdots,r_{n-1})\in [0,1)^{n-1}$, 
$(\varepsilon^{\iota})^r=\prod_{1 \le j \le n-1}(\varepsilon_j^{\iota})^{r_j}$ for $\iota\in J_F$, 
$\varepsilon^r=((\varepsilon^{\iota})^r)_{\iota\in J_F}$, and $y=-\mathrm{log}(r_n)$. 
We put 
\[
\overline{\Omega}_E=\prod_{1 \le j \le n-1}\{\varepsilon_j^{r_j}\mid r_j\in [0,1]\}\times (\mathbb{R}_{\ge 0}\cup \{\infty\}). 
\]

We define a singular $n$-cube $\ell_i: [0,1]^n \to \overline{\Omega}_E\to (\mathfrak{H}^{J_F})^{\text{BS}}$ by 
\[
(r_1,\cdots,r_n)
\mapsto 
(\varepsilon_1^{r_1},\cdots,\varepsilon_{n-1}^{r_{n-1}},-\mathrm{log}(r_n))
\mapsto 
\sqrt{-1}y^{\frac{1}{n}}
\varepsilon^r.
\]
Let $c_{E,i}$ denote the composition of $\ell_i$ and the canonical map $(\mathfrak{H}^{J_F})^{\text{BS}} \twoheadrightarrow Y_i^{\text{BS}}$.

%
\subsection{Mellin transform}\label{twist Mellintrans}
%

In this subsection, we give a Mellin transform for a Hilbert modular form $\in M_k(\Gamma_{1}(\ideal{d}_F[t_i],\ideal{n}),\C)$ vanishing at cusps $\Gamma_0(\ideal{d}_F[t_1],\ideal{n})$-equivalent to the cusp $\infty$. 
For the proof, we must need to prove analytic properties of the $L$\nobreakdash-functions, 
which are obtained by P-B. Garrett for a Hilbert cusp form $\in S_k(\Gamma(\ideal{n}),\C)$ of level $\Gamma(\ideal{n})$ (\cite[\S 1.9, p.37, Theorem]{Ga}). 
Here $\Gamma(\ideal{n})$ is the principal congruence subgroup of level $\ideal{n}$.
In order to do it, we strictly follow the argument in the method of Garrett. 
We keep the notation in \S\ref{Diri}, \S \ref{Borel--Serre}, and \S\ref{Fundamental domain}. 
\begin{prop}\label{abs conv}
Let $h\in M_k(\Gamma_{1}(\ideal{d}_F[t_i],\ideal{n}),\C)$. 
For $s\in \C$ such that $\mathrm{Re}(s)\gg 0$, the integral 
\[
\int_{\text{$\mathrm{image}$ $\mathrm{of}$ $c_{{}_{E,i}}$}} y^{(s-1)n} w(\widetilde{h})
\]
converges absolutely and extends to a meromorphic function on the complex plane, which is holomorphic at $s=1$. 
Here $w(h)$ is defined by $(\ref{n-form})$ and $\widetilde{h}(z)=h(z)-a_{\infty}(0,h)$. 
\end{prop}
\begin{proof}
For $s\in \C$ such that $\mathrm{Re}(s)\gg 0$, the integral above is just equal to the following:
\begin{align}\label{int 1}
\int_{[0,1]^{n-1}} \int_{\sqrt{-1} \mathbb{R}_{+}^{\times}} y^{(s-1)} w(\widetilde{h})
=\int_{[0,1]^{n-1}} \left(\int_{\sqrt{-1}}^{\sqrt{-1} \infty}+\int_{0}^{\sqrt{-1}}\right)y^{(s-1)} w(\widetilde{h}).
\end{align}
We calculate the second term. 
Put $\sigma=\left(\begin{pmatrix}0&1\\-1&0\\\end{pmatrix}\right)_{\iota\in J_F}\in G_{\infty,+}$. 
By the pull-back formula, we have 
\begin{align}\label{int 2}
&\int_{[0,1]^{n-1}}\int_{0}^{\sqrt{-1}} y^{(s-1)} w(\widetilde{h})
=-\int_{[-1,0]^{n-1}} \int_{\sqrt{-1}}^{\sqrt{-1} \infty} y^{(1-s)} 
\sigma \bullet w(\widetilde{h \vert \sigma})\\
\nonumber&\ \ \ \ \ \ -\int_{[-1,0]^{n-1}} \int_{\sqrt{-1}}^{\sqrt{-1} \infty} y^{(1-s)} \sigma \bullet w(a_{\infty}(0,h \vert \sigma))
-\int_{[0,1]^{n-1}} \int_{0}^{\sqrt{-1}} y^{(s-1)} w(a_{\infty}(0,h)).
\end{align}
An elementary calculation shows that the second (resp. third) term of (\ref{int 2}) is absolutely convergent for 
$\text{Re}(s) > k$ (resp. $\text{Re}(s)\ge 1$) and holomorphic at $s=1$. 
Thus, for the proof, it suffices to show that the first terms of (\ref{int 1}) and (\ref{int 2}) are absolutely convergent and holomorphic at $s=1$. 
Hence we reduce it to showing that the integral
\begin{align}\label{int}
\int_{[a,b]^{n-1}} \int_{1}^{\infty}
y^{(s-1)} \widetilde{h}(\sqrt{-1}y^{\frac{1}{n}}\varepsilon^{r}) y^m drdy
\end{align}
is absolutely convergent and holomorphic at $s=1$ for any $a,b\in\mathbb{R}$ with $a\le b$ and non\nobreakdash-negative integer $m$. 
There is a constant $M>0$ such that $N(\xi)> M$ 
for any $0\ll\xi\in [t_i]$. 
Then there is a constant $\varepsilon>0$ such that $N(\xi)>M+\varepsilon$ 
for any such $\xi$. 
Thus, by the same argument as in \cite[p.29]{Ga}, we have an estimate 
\begin{align*}
\exp\left(\pi n M^{\frac{1}{n}}y^{\frac{1}{n}}\right)\bigg|\widetilde{h}(\sqrt{-1}y^{\frac{1}{n}}\varepsilon^{r})\bigg|\le \sum_{0\ll\xi\in [t_i]}|a_{\infty}(\xi,h)|\exp\left(-\pi\left(2-\left(\frac{M}{M+\varepsilon}\right)^{\frac{1}{n}}\right)\Tr(\xi y^{\frac{1}{n}}\varepsilon^r)\right). 
\end{align*}
Since $\widetilde{h}(z)$ is absolutely convergent, so is the latter series. 
Hence there are constants $C,C'>0$ such that 
\begin{align*}
\bigg|\widetilde{h}(\sqrt{-1}y^{\frac{1}{n}}\varepsilon^{r})\bigg|
\le C\exp\left(-C'y^{\frac{1}{n}}\right)
\end{align*}
for any $y\ge 1$ and $r\in[a,b]^{n-1}$. 
Therefore the integral (\ref{int}) is dominated by 
\begin{align*}
\int_{[a,b]^{n-1}} \int_{1}^{\infty}
\exp(-C'y^{\frac{1}{n}}) y^{\text{Re}(s)-1+m} drdy
\end{align*}
and hence is absolutely convergent and a holomorphic function of $s\in\C$. 
\end{proof}

We assume that $h_F^+=1$. 
We fix a Hilbert cusp form $\textbf{f}$ and the Hilbert Eisenstein series $\textbf{E}_2(\varphi,\psi)$ given in Proposition \ref{Hilbert Eisenstein} satisfying the following conditions:
\begin{align}\label{zero point}
&\textbf{f}\in S_{2}(\ideal{n},\chi,\C)\text{ and }\\
\nonumber&\textbf{E}_2(\varphi,\psi)\in M_{2}(\ideal{n},\chi,\C)\text{ vanishes at cusps $\Gamma_{0}(\ideal{d}_F[t_1],\ideal{n})$-equivalent to the cusp $\infty$.} 
\end{align}

We simply write $\textbf{h}=\textbf{f}$ or $\textbf{E}_2(\varphi,\psi)$. 
We express the special values of the Dirichlet series $D(1,\textbf{h},\eta)$ as a Mellin transform for $\textbf{h}$ (cf. \cite[\S16]{Oda}, \cite[\S7, \S8]{Hida94}, and \cite[\S3]{Ochi}). 

Let $\eta$ be a $\overline{\Q}$-valued narrow ray class character of $F$ whose conductor is denoted by $\ideal{m}_{\eta}$ such that $\ideal{m}_{\eta}$ is prime to $\ideal{d}_F[t_1]$ and $\ideal{n}|\ideal{m}_{\eta}$. 
Let $(\ideal{m}_{\eta}^{-1}\ideal{d}_F^{-1}[t_1]^{-1}/\ideal{d}_F^{-1}[t_1]^{-1})^{\times}$ (resp. $(\ideal{m}_{\eta}^{-1}/\mathfrak{o}_F)^{\times}$) be the subset of $\ideal{m}_{\eta}^{-1}\ideal{d}_F^{-1}[t_1]^{-1}/\ideal{d}_F^{-1}[t_1]^{-1}$ (resp. $\ideal{m}_{\eta}^{-1}/\mathfrak{o}_F$) consisting of elements whose annihilator is $\ideal{m}_{\eta}$. 
We fix a non-canonical isomorphism of $\mathfrak{o}_F$-modules 
$\ideal{m}_{\eta}^{-1}\ideal{d}_F^{-1}[t_1]^{-1}/\ideal{d}_F^{-1}[t_1]^{-1} \simeq \ideal{m}_{\eta}^{-1}/\mathfrak{o}_F \simeq \mathfrak{o}_F/\ideal{m}_{\eta}$ 
and a non-canonical bijection induced from it 
$(\ideal{m}_{\eta}^{-1}\ideal{d}_F^{-1}[t_1]^{-1}/\ideal{d}_F^{-1}[t_1]^{-1})^{\times} \simeq (\ideal{m}_{\eta}^{-1}/\mathfrak{o}_F)^{\times}\simeq (\mathfrak{o}_F/\ideal{m}_{\eta})^{\times}$. 
Hence we may canonically identify $(\ideal{m}_{\eta}^{-1}\ideal{d}_F^{-1}[t_1]^{-1}/\ideal{d}_F^{-1}[t_1]^{-1})^{\times}/\mathfrak{o}_{F,+}^{\times}$ with a subgroup of $\text{Cl}_F^+(\ideal{m}_{\eta})$ under the canonical extension 
\begin{align*}
1\to (\mathfrak{o}_F/\ideal{m}_{\eta})^{\times}/\mathfrak{o}_{F,+}^{\times} \to \text{Cl}_F^+(\ideal{m}_{\eta}) \to \text{Cl}_F^+\to 1. 
\end{align*}

Let $\eta_1$ be the function on 
$(\ideal{m}_{\eta}^{-1}\ideal{d}_F^{-1}[t_1]^{-1}/\ideal{d}_F^{-1}[t_1]^{-1})^{\times}/\mathfrak{o}_{F,+}^{\times}$ defined by $\eta_1(\bar{b})=\eta(\bar{b}\ideal{m}_{\eta}\ideal{d}_F[t_1])$. 
We note that $\eta_1(\xi \bar{b})=\eta(\xi)\eta_1(\bar{b})$ for any $\bar{b}\in (\ideal{m}_{\eta}^{-1}\ideal{d}_F^{-1}[t_1]^{-1}/\ideal{d}_F^{-1}[t_1]^{-1})^{\times}/\mathfrak{o}_{F,+}^{\times}$ and $0\ll\xi\in [t_1]$ prime to $\ideal{m}_{\eta}$. 
Let $E$ denote $\mathfrak{o}_{F,\ideal{m}_{\eta},+}^{\times}:=\{e\in \mathfrak{o}_{F,+}^{\times}\mid e\equiv 1\ (\bmod\ \ideal{m}_{\eta})\}$. 

We fix a complete set $S$ (resp. $T$) of representatives of $(\ideal{m}_{\eta}^{-1}\ideal{d}_F^{-1}[t_1]^{-1}/\ideal{d}_F^{-1}[t_1]^{-1})^{\times}/\mathfrak{o}_{F,+}^{\times}$ in $\ideal{m}_{\eta}^{-1}\ideal{d}_F^{-1}[t_1]^{-1}$
(resp. $\mathfrak{o}_{F,+}^{\times}/E$ in $\mathfrak{o}_{F,+}^{\times}$) 
satisfying the condition that 
\begin{align}\label{b equiv infty}
\text{every cusp $b\in S$ is $\Gamma_{0}(\ideal{d}_F [t_1],\ideal{n})$-equivalent to the cusp $\infty$.}
\end{align} 
Here we note that the existence of such set follows from the assumption $\ideal{n}|\ideal{m}_{\eta}$. 
Indeed, fix a generator $m$ (resp. $c$) of $\ideal{m}_{\eta}$ (resp. $\ideal{d}_F[t_1]$) and a set $S'$ of representatives of $(\mathfrak{o}_F/\ideal{m}_{\eta})^{\times}$ satisfying the condition that each $x\in S'$ is prime to $mc$. 
Then $\{x/mc\mid x\in S'\}$ is a complete set of representatives of $(\ideal{m}_{\eta}^{-1}\ideal{d}_F^{-1}[t_1]^{-1}/\ideal{d}_F^{-1}[t_1]^{-1})^{\times}/\mathfrak{o}_{F,+}^{\times}$. 
The assumption $\ideal{n}|\ideal{m}_{\eta}$ implies $mc\in \ideal{n}\ideal{d}_F [t_1]$ and hence there is 
$\begin{pmatrix}x&\ast\\ mc&\ast\end{pmatrix}\in\Gamma_{0}^1(\ideal{d}_F [t_1],\ideal{n})$.

Let $\bar{b}$ denote the image of $b\in S$ in $(\ideal{m}_{\eta}^{-1}\ideal{d}_F^{-1}[t_1]^{-1}/\ideal{d}_F^{-1}[t_1]^{-1})^{\times}/\mathfrak{o}_{F,+}^{\times}$ under the canonical map. 
We have
\begin{align}\label{twist MF}
&N([t_1])^{s-k/2}
\sum_{b\in S} \sum_{u \in T}
\eta_1(\bar{b})^{-1}h_1(z+b u)\\
\nonumber
&=N([t_1])^{s-k/2}
\sum_{0\ll\xi \in [t_1]}a_{\infty}(\xi,h_1)
\sum_{b\in S} \sum_{u\in T}
\eta_1(\bar{b})^{-1}e_F(\xi b u)e_F(\xi z)\\
\nonumber
&=\tau(\eta^{-1})
N([t_1])^{s-k/2}
\sum_{0\ll\xi \in [t_1]}a_{\infty}(\xi,h_1)\eta(\xi[t_1]^{-1})e_F(\xi z). 
\end{align}
Here the last equality follows from \cite[(3.11)]{Shi}.
By taking $\Omega_E=\coprod_{u\in T}u^{-1}\Omega_{\mathfrak{o}_{F,+}^{\times}}$, we have
\begin{align*}
&N([t_1])^{s-k/2}
\sum_{b\in S}
\eta_1(\bar{b})^{-1}
\int_{\Omega_E}h_1(z+b)y^{(s-1)t}dz_{{}_{J_F}}\\
\nonumber&= N([t_1])^{s-k/2}
\sum_{b\in S}
\eta_1(\bar{b})^{-1}
\sum_{u\in T}
\int_{u^{-1}\Omega_{\mathfrak{o}_{F,+}^{\times}}}h_1(z+b)y^{(s-1)t}dz_{{}_{J_F}}\\
\nonumber&= N([t_1])^{s-k/2}
\sum_{b\in S}
\eta_1(\bar{b})^{-1}
\sum_{u\in T}
\int_{\Omega_{\mathfrak{o}_{F,+}^{\times}}}h_1(z+b u)y^{(s-1)t}dz_{{}_{J_F}}\\
\nonumber&=\int_{\Omega_{\mathfrak{o}_{F,+}^{\times}}}
N([t_1])^{s-k/2}
\sum_{b\in S}\sum_{u\in T}
\eta_1(\bar{b})^{-1}
h_1(z+b u)y^{(s-1)t}dz_{{}_{J_F}}.
\end{align*}
Here we note that 
each integral is well-defined by using Proposition \ref{abs conv}, our assumption (\ref{zero point}), and the condition (\ref{b equiv infty}). 
By using the expansion of (\ref{twist MF}), 
for $\mathrm{Re}(s)\gg 0$, we have
\begin{align*}
&N([t_1])^{s-k/2}
\sum_{b\in S}
\eta_1(\bar{b})^{-1}
\int_{\sqrt{-1}(F\otimes \mathbb{R})_+^{\times}/E}h_1(z+b)y^{(s-1)t}dz_{{}_{J_F}}\\
\nonumber&=\tau(\eta^{-1})
N([t_1])^{s-k/2}
\sum_{0\ll\xi\in [t_1]}a_{\infty}(\xi,h_1)\eta(\xi[t_1]^{-1})
\int_{\Omega_{\mathfrak{o}_{F,+}^{\times}}}e_F(\xi z)y^{(s-1)t}dz_{{}_{J_F}}\\
\nonumber&=\tau(\eta^{-1})
\sum_{0\ll\xi\in [t_1]}
\frac{a_{\infty}(\xi,h_1)\eta(\xi[t_1]^{-1})N([t_1])^{-k/2}}{N(\xi[t_1]^{-1})^s}
\int_{\Omega_{\mathfrak{o}_{F,+}^{\times}}}e_F(\xi z)(\xi y)^{(s-1)t}\bigwedge_{\iota\in J_F} d\xi^{\iota}z_{\iota}\\
\nonumber&=\tau(\eta^{-1})
\sum_{\xi\mathfrak{o}_{F,+}^{\times}}
\frac{a_{\infty}(\xi,h_1)\eta(\xi[t_1]^{-1})N([t_1])^{-k/2}}{N(\xi[t_1]^{-1})^s}
\int_{\sqrt{-1}(F\otimes \mathbb{R})_+^{\times}}e_F(\xi z)(\xi y)^{(s-1)t}\bigwedge_{\iota\in J_F} d\xi^{\iota}z_{\iota}\\
&\nonumber=
\tau(\eta^{-1})D(s,\bold{h}, \eta)
(2\pi)^{-sn} \sqrt{-1}^{n}\Gamma(s)^n. 
\end{align*}
Here we note that 
each integral is well-defined by using Proposition \ref{abs conv}, 
and we may regard $h_1(z+b)$ as a function on $\sqrt{-1}(F\otimes \mathbb{R})_+^{\times}/E$ since $h_1(uz+b)=h_1(z+b)$ for any $u\in E$. 
Furthermore, the integrals in the first line of this equation are independent of the choice of a lift $b$ of $\bar{b}$. 
Hence the integral depends only on the image $\bar{b}$ of $b$ in $(\ideal{m}_{\eta}^{-1}\ideal{d}_F^{-1}[t_1]^{-1}/\ideal{d}_F^{-1}[t_1]^{-1})^{\times}/\mathfrak{o}_{F,+}^{\times}$ and it shall be denoted by 
\[
\int_{\sqrt{-1}(F\otimes \mathbb{R})_+^{\times}/E}h_1(z+\bar{b})y^{(s-1)t}dz_{{}_{J_F}}. 
\]
Therefore we obtain the following Mellin transform:
\begin{prop}\label{Mellin}
Assume that $h_F^+=1$. 
Let $\mathbf{h}=\mathbf{f}$ or $\mathbf{E}_2(\varphi,\psi)$ as $(\ref{zero point})$ and $\eta$ a $\overline{\Q}$-valued narrow ray class character of $F$ whose conductor is denoted by $\ideal{m}_{\eta}$ such that $\ideal{m}_{\eta}$ is prime to $\ideal{d}_F[t_1]$ and $\ideal{n}|\ideal{m}_{\eta}$. 
Then we have
\begin{align*}
\sum_{b\in S}
\eta_1(\bar{b})^{-1}
\int_{\sqrt{-1}(F\otimes \mathbb{R})_+^{\times}/\mathfrak{o}_{F,\ideal{m}_{\eta},+}^{\times}}h_1(z+\bar{b})dz_{{}_{J_F}}
=\tau(\eta^{-1})D(1,\bold{h}, \eta)
(-2\pi \sqrt{-1})^{-n}. 
\end{align*}
\end{prop}

\begin{rem}
As mentioned above, 
the assumption $\ideal{n}|\ideal{m}_{\eta}$ and the conditions (\ref{zero point}) and (\ref{b equiv infty}) imply that the integrals are well-defined. 
If $\textbf{h}$ is a Hilbert cusp form, then the Mellin transform as Proposition \ref{Mellin} is satisfied without the assumption $\ideal{n}|\ideal{m}_{\eta}$.
\end{rem}

We consider a Mellin transform in the anti-holomorphic case. 
Let $W_G$ denote the Weyl group $K_{\infty}/K_{\infty,+}$, which is identified with $\{w_J\mid J\subset J_F\}$, 
where $w_J \in K_{\infty}$ such that 
$w_{J,\iota}=\begin{pmatrix}1&0\\0&1\end{pmatrix}$ if $\iota\in J$ and 
$w_{J,\iota}=\begin{pmatrix}-1&0\\0&1\end{pmatrix}$ if $\iota\in J_F \backslash J$. 
Then $W_G$ acts on the space of Hilbert modular forms via $\textbf{h}\mapsto \textbf{h}_J:=\textbf{h}|[K_{\infty}w_J K_{\infty}]$ for each subset $J$ of $J_F$. 

\begin{prop}\label{anti Mellin}
Under the same notation and assumptions as Proposition \ref{Mellin}, we have
\begin{align*}
&\sum_{b\in S}
\eta_1(\bar{b})^{-1}
\int_{\sqrt{-1}(F\otimes \mathbb{R})_+^{\times}/\mathfrak{o}_{F,\ideal{m}_{\eta},+}^{\times}}
h_{J,1}(z+\bar{b})dz_{{}_{J}}
=\tau(\eta^{-1})D(1,\mathbf{h},\eta)\eta_{\infty}(\nu_J)
(-2\pi \sqrt{-1})^{-n},  
\end{align*}
where $dz_{{}_{J}}$ is defined by $(\ref{dz_J})$ and $\nu_J\in \mathbb{A}_{F,\infty}$ such that $\nu_{J,\iota}=1$ if $\iota\in J$ and $\nu_{J,\iota}=-1$ if $\iota\in J_F\backslash J$. 
\end{prop}
\begin{proof}
Since $h_F^+=1$, we can take $a\in \mathfrak{o}_F^{\times}$ such that $\iota(a)>0$ if $\iota \in J$ and $\iota(a)<0$ if $\iota\in J_F\backslash J$. 
By putting $\gamma=\begin{pmatrix}a&0\\0&1\end{pmatrix}$, 
the action of $[K_{\infty}w_J K_{\infty}]$ on $Y(\ideal{n})=Y_1$ is given by $z\mapsto \gamma^{-1}z$. 
Then, by definition, we have 
$h_{J,1}(z) =h_1(\gamma^{-1}z)(-1)^{\sharp (J_F\backslash J)}$.
Hence we obtain
\begin{align*}
h_{1}(\gamma^{-1}z)
&=\sum_{\mu\in [t_1], \{\mu\}=J} c(\mu[t_1]^{-1},\textbf{h})|N(\mu)| 
e_F(\sqrt{-1}\mu y_{\infty}^{w_J})
e_F(\mu x_{\infty}).
\end{align*}
Here $\{\mu\}=\{\iota\in J_F |\mu^{\iota}>0\}$ and 
$y_{\infty,\iota}^{w_J}=y_{\infty,\iota}$ (resp. $-y_{\infty,\iota}$) if $\iota\in J$ (resp. $\iota\in J_F\backslash J$). 
Now our assertion follows from the same argument 
as in the proof of Proposition \ref{Mellin}. 
\end{proof}

%
\subsection{Relation between cohomology class and Dirichlet series}\label{subsection:modular symbol}
%

In this subsection, we give a cohomological description of special values of the $L$-functions. 

We keep the notation in \S \ref{twist Mellintrans}.
As the previous subsection, we assume that $h_F^+=1$. 
We fix a lift $b \in S$ of $\bar{b}\in (\ideal{m}_{\eta}^{-1}\ideal{d}_F^{-1}[t_1]^{-1}/\ideal{d}_F^{-1}[t_1]^{-1})^{\times}/\mathfrak{o}_{F,+}^{\times}$. 
We consider the Hilbert modular variety $Y(\ideal{n})$ defined by (\ref{adele HMV}). 
Let $C(\Gamma_{1}(\ideal{d}_F [t_1],\ideal{n}))$ denote the set of all cusps of $Y(\ideal{n})$. 
Let $C_{\infty}$ be the subset of $C(\Gamma_{1}(\ideal{d}_F [t_1],\ideal{n}))$ consisting of cusps $\Gamma_{0}(\ideal{d}_F [t_1],\ideal{n})$\nobreakdash-equivalent to the cusp $\infty$. 

We consider the following subset $H_{b}$ of $\mathfrak{H}^{J_F}$: 
\[
H_{b}:=b+\sqrt{-1}(F\otimes\mathbb{R})_+^{\times}\hookrightarrow \mathfrak{H}^{J_F}. 
\]

We define an action of $\mathfrak{o}_{F,\ideal{m}_{\eta},+}^{\times}$ on $H_{b}$ by 
$\varepsilon\ast(z_{\iota})_{\iota\in J_F}=(\varepsilon^{\iota}z_{\iota}-(\varepsilon^{\iota}-1)b)_{\iota\in J_F}$.
Since $(\varepsilon-1) b\in \ideal{d}_F^{-1}[t_1]^{-1}$ for any $\varepsilon\in \mathfrak{o}_{F,\ideal{m}_{\eta},+}^{\times}$, we see that 
$\varepsilon\ast(z_{\iota})_{\iota\in J_F}$ is $\Gamma_{1}(\ideal{d}_F [t_1],\ideal{n})$-equivalent to $(z_{\iota})_{\iota\in J_F}$. 
Therefore we have
$H_{b}/\mathfrak{o}_{F,\ideal{m}_{\eta},+}^{\times} \to Y(\ideal{n})$ and 
it induces
$H_c^n(Y(\ideal{n}),A)\to
H_c^n(H_{b}/\mathfrak{o}_{F,\ideal{m}_{\eta},+}^{\times},A)$
for $A=\integer{}$, $K$, or $\C$. 

We define subsets $H_b^{\mathrm{BS}}$, $\partial_{\infty}$, and $\partial_{b}$ of $X \times (\mathbb{R}_{\ge 0}\cup \{\infty\})$ as follows.
Let $X_b$ denote the image of $H_b$ in $X$ under the composition of the isomorphism (\ref{H^{J_F}}) and the projection to $X$. 
We have $H_b \simeq X_b \times \mathbb{R}_+^{\times}$. 
We define $H_b^{\mathrm{BS}}$, $\partial_{\infty}$, and $\partial_{b}$ by 
\begin{align*}
&H_b^{\mathrm{BS}}=X_b \times (\mathbb{R}_{\ge 0}\cup \{\infty\}),& 
&\partial_{\infty}=X_b \times \{\infty\}, &
&\partial_{b}=X_b \times \{0\}.&
\end{align*}
The action of $\mathfrak{o}_{F,\ideal{m}_{\eta},+}^{\times}$ on $H_{b}$ extends canonically to an action on $H_b^{\mathrm{BS}}$.
We put $\alpha_b=\small\begin{pmatrix}-b&1+b^2\\-1&b \end{pmatrix}$.
Note that $\alpha_b(\infty)=b$.
The embedding $H_b \hookrightarrow \mathfrak{H}^{J_F}$
(resp. the composition $H_b \xrightarrow{\alpha_b} H_b \hookrightarrow \mathfrak{H}^{J_F}\xrightarrow{\alpha_b}\mathfrak{H}^{J_F}$) induces an $\mathfrak{o}_{F,\ideal{m}_{\eta},+}^{\times}$-equivariant map
\begin{align}\label{extend to H_b^{BS}}
&H_b \cup \partial_{\infty} \to \mathfrak{H}^{J_F}\cup (X\times\{\infty\})&
&\text{(resp. $H_b \cup \partial_{b} \to \mathfrak{H}^{J_F}\cup \alpha_b(X\times\{\infty\})$)}&
\end{align}
because $\alpha_b(b+\sqrt{-1}y)=b+\sqrt{-1}/y$.
Therefore we have $H_{b}^{\mathrm{BS}}/\mathfrak{o}_{F,\ideal{m}_{\eta},+}^{\times} \to Y(\ideal{n})^{\mathrm{BS}}$ and it induces
\begin{align}\label{Gysin}
H^n(Y(\ideal{n})^{\text{BS}},&D_{C_{\infty}}(\ideal{n});A)
\to H^n(Y(\ideal{n})^{\text{BS}},D_{b,\infty}(\ideal{n});A)\\
\nonumber&\ \ \to H^n(H_{b}^{\mathrm{BS}}/\mathfrak{o}_{F,\ideal{m}_{\eta},+}^{\times},\partial_{b}/\mathfrak{o}_{F,\ideal{m}_{\eta},+}^{\times}\cup\partial_{\infty}/\mathfrak{o}_{F,\ideal{m}_{\eta},+}^{\times};A)
\simeq
H_c^n(H_{b}/\mathfrak{o}_{F,\ideal{m}_{\eta},+}^{\times},A)
\end{align}
for $A=\integer{}$, $K$, or $\C$. 
Here $D_{s}$ is the boundary of $Y(\ideal{n})^{\text{BS}}$ at a cusp $s$ as explained in \S \ref{Borel--Serre}, 
$D_{C_{\infty}}(\ideal{n})=\coprod_{s\in C_{\infty}}D_{s}$, 
and $D_{b,\infty}(\ideal{n})=D_{b}\cup D_{\infty}$.

We define the evaluation map
\begin{align}\label{ev map}
\mathrm{ev}_{b,1,A} :
\widetilde{H}^n(Y(\ideal{n})^{\text{BS}},D_{C_{\infty}}(\ideal{n});A)\to A
\end{align}
by the composition of (\ref{Gysin}) and the trace map 
$H_c^n(H_{b}/\mathfrak{o}_{F,\ideal{m}_{\eta},+}^{\times},A)\to A$, where 
\[
\widetilde{H}^n(Y(\ideal{n})^{\text{BS}},D_{C_{\infty}}(\ideal{n});A)
:=H^n(Y(\ideal{n})^{\text{BS}},D_{C_{\infty}}(\ideal{n});A)/\text{$A$-torsion}.
\]
Note that the definition of $\mathrm{ev}_{b,1,A}$ depends only on $\bar{b}$ (because if $\bar{b}=\bar{b}'$, then there is $\gamma \in \Gamma_1(\ideal{d}_F[t_1],\ideal{n})$ such that $b=\gamma (b')$ and $\gamma(\infty)=\infty$) and hence it shall be denoted by $\mathrm{ev}_{\bar{b},1,A}$.

In order to give a cohomological description of the $L$-functions, 
we recall the relative de Rham theory, which is proved by A. Borel \cite[Theorem 5.2]{Bo} for general locally symmetric spaces. 
Let $\Omega^{\bullet}(Y(\ideal{n}),\C)$ denote the complex of $\C$-valued $C^{\infty}$-differential $\Gamma_{1}(\ideal{d}_F[t_1],\ideal{n})$\nobreakdash-invariant forms in $\mathfrak{H}^{J_F}$. 
Let $\Omega_{\text{fd}}^{\bullet}(Y(\ideal{n}),D_{C_{\infty}}(\ideal{n});\C)$ denote the complex of forms in $\Omega^{\bullet}(Y(\ideal{n}),\C)$ which, together with their exterior differentials, are fast decreasing at every $s\in C_{\infty}$. 
By the proof of \cite[Theorem 5.2]{Bo} on the stalks at the boundary, we have 
\begin{align}\label{relative de Rham theory}
H_{\text{dR}}^n(Y(\ideal{n}),\Omega_{\text{fd}}^{\bullet}(Y(\ideal{n}),D_{C_{\infty}}(\ideal{n});\C))\simeq H^n (Y(\ideal{n})^{\text{BS}},D_{C_{\infty}}(\ideal{n});\C).
\end{align}

Let $\textbf{h}=\mathbf{f}$ or $\mathbf{E}_2(\varphi,\psi)$ as (\ref{zero point}). 
Let $[\omega_{\textbf{h}}]$ denote the Betti cohomology class in 
$H^n(Y(\ideal{n}),\C)$ attached to the de Rham cohomology class of $\omega(\mathbf{h})$.
By the isomorphism (\ref{relative de Rham theory}) and the condition (\ref{zero point}), we can define the relative cohomology class $[\omega_{\textbf{h}}]_{\text{rel}}$ in $H^n(Y(\ideal{n})^{\text{BS}},D_{C_{\infty}}(\ideal{n});\C)$ attached to the relative de Rham cohomology class of $\omega(\mathbf{h})$ whose image in $H^n(Y(\ideal{n}),\C)$ is $[\omega_{\textbf{h}}]$. 
Now, by combining these observations and Proposition \ref{Mellin}, we obtain the following: 
\begin{prop}\label{Modular symbol}
Let $\mathbf{h}=\mathbf{f}$ or $\mathbf{E}_2(\varphi,\psi)$ as $(\ref{zero point})$. 
Let $A=\integer{}$, $K$, or $\C$.
Assume that $h_F^+=1$ and
$a [\omega_{\mathbf{h}}]_{\mathrm{rel}}\in \widetilde{H}^n(Y(\ideal{n})^{\mathrm{BS}},D_{C_{\infty}}(\ideal{n});A)$ for some $a\in A$. 
Let $\eta$ be a $\overline{\Q}$-valued narrow ray class character of $F$ whose conductor is denoted by $\ideal{m}_{\eta}$ such that $\ideal{m}_{\eta}$ is prime to $\ideal{d}_F[t_1]$ and $\ideal{n}|\ideal{m}_{\eta}$. 
Then we have
\begin{align*}
A(\eta)\ni\sum_{b\in S}\eta_1(\bar{b})^{-1} \mathrm{ev}_{\bar{b},1,A}(a[\omega_{\mathbf{h}}]_{\mathrm{rel}})
=a\tau(\eta^{-1})D(1,\mathbf{h}, \eta)
(-2\pi \sqrt{-1})^{-n}. 
\end{align*}
\end{prop}

We treat the anti-holomorphic case. 
By the description of the action of $[K_{\infty}w_J K_{\infty}]$ on $Y(\ideal{n})$ mentioned in the proof of Proposition \ref{anti Mellin}, 
we see that the action of $[K_{\infty}w_J K_{\infty}]$ preserves the component $D_{C_{\infty}}(\ideal{n})$ and hence $[K_{\infty}w_J K_{\infty}]$ acts on $\widetilde{H}^n(Y(\ideal{n})^{\text{BS}},D_{b,\infty}(\ideal{n});A)$. 
We obtain the following proposition by the same argument as in the proof of Proposition \ref{Modular symbol}, using Proposition \ref{anti Mellin} instead of Proposition \ref{Mellin}: 
\begin{prop}\label{Modular symbol anti-hol}
Under the same notation and assumptions as Proposition \ref{anti Mellin} and Proposition \ref{Modular symbol}, we have 
\begin{align*}
A(\eta) \ni \sum_{b\in S}
\eta_1(\bar{b})^{-1} 
\mathrm{ev}_{\bar{b},1,A}\left(a[\omega_{\mathbf{h}}]_{\mathrm{rel}}|[K_{\infty}w_J K_{\infty}]\right)
=a\tau(\eta^{-1})D(1,\mathbf{h},\eta)\eta_{\infty}(\nu_J)
(-2\pi \sqrt{-1})^{-n}. 
\end{align*}
\end{prop}

\section{Integrality of cohomology classes at boundary}\label{subsection:Rationality}
%
The purpose of this section is to prove the integrality of the restriction of the Betti cohomology class associated to a Hilbert Eisenstein series to the boundary of the Borel\nobreakdash--Serre compactification of $Y(\ideal{n})$ (Proposition \ref{const}). 
The proof is based on the comparison theorem between Betti cohomology and group cohomology, and an explicit computation of the restriction of a group cocycle associated to a Hilbert Eisenstein series to the boundary.

\subsection{Cocycles associated to Hilbert modular forms}\label{Cocycle ass to HMF}
%
In this subsection, we construct a group cocycle associated to a Hilbert modular form, which is a generalization of the Eichler\nobreakdash--Shimura cocycle in the case where $F=\Q$. 
We strictly follow the argument in the method of H. Yoshida \cite{Yo}. 
We fix $i$ with $1\le i\le h_F^+$ and abbreviate $\Gamma_{1}(\ideal{d}_F [t_i],\ideal{n})$ to $\Gamma$. 

First we recall the definition of group cohomology. 
Let $R$ be a commutative ring and $M$ a left $R[\overline{\Gamma}]$-module. 
For a non-negative integer $q$, 
let $C^q$ denote the space of functions on $\overline{\Gamma}^q$ with values in $M$.
The differential map $d^q \colon C^q \rightarrow C^{q+1}$ 
is given by 
\begin{align*}
d^q u(\overline{\gamma_1},\cdots,\overline{\gamma_{q+1}})
&=\overline{\gamma_1} u(\overline{\gamma_2},\cdots,\overline{\gamma_{q+1}})\\
&\ \ \ +\sum_{1 \le j \le q}(-1)^j u(\overline{\gamma_1},\cdots,\overline{\gamma_j \gamma_{j+1}},\cdots,\overline{\gamma_q})
+(-1)^{q+1} u(\overline{\gamma_1},\cdots,\overline{\gamma_q}).
\end{align*}
The associated $q$-th cohomology group of $\overline{\Gamma}$ with coefficients in $M$ is given by
\[
H^q(\overline{\Gamma},M)=\ker(d^q \colon C^q \rightarrow C^{q+1})/\im(d^{q-1} \colon C^{q-1} \rightarrow C^{q}).
\]

For a subset $J$ of $J_F$, we put 
$z_{\iota}^J=z_{\iota}$ (resp. $\overline{z}_{\iota}$) if $\iota\in J$ (resp. $\iota \in J_F \backslash J$) and 
\begin{align}\label{dz_J}
dz_{{}_J}=\bigwedge_{\iota\in J_F} dz_{\iota}^J.
\end{align}

For a $\Z$-algebra $A$, $u,v\in A$, and a non-negative integer $\ell$, we put 
\[
\begin{bmatrix}
u\\
v
\end{bmatrix}^{\ell}
={}^{t}(u^{\ell},u^{\ell -1}v,\cdots,uv^{\ell -1},v^{\ell}).
\]
Let $L_{\ell}(A)$ denote the space of column vectors $A^{\ell+1}\simeq \Sym^{\ell}(A^2)$. 
We define the $\ell$-th symmetric tensor representation $\rho_{\ell}$ of 
$\GL_2(A)$ on $L_{\ell}(A)$ by 
\[
\rho_{\ell}(g)\begin{bmatrix}u\\v\end{bmatrix}^{\ell}
=\left[
g\begin{pmatrix}
u\\
v
\end{pmatrix}
\right]^{\ell}. 
\]

Let $k$ be an integer $\ge 2$.
Recall that $\widetilde{F}$ is the Galois closure of $F$ in $\overline{\Q}$ and 
$F'$ is the field generated by elements $\varepsilon^{t/2}$ for all $\varepsilon\in \mathfrak{o}_{F,+}^{\times}$ over $\widetilde{F}$ as defined in \S \ref{subsection:Geometric HMF}. 
For an $\mathfrak{o}_{F'}$\nobreakdash-algebra $A$, we define an 
$A[(\M_2(\mathfrak{o}_{F'})\cap \GL_2(F'))^{J_F}]$\nobreakdash-module $L_{kt}(A)$ as follows: let $L_{kt}(A)$ be the $A$-module $\otimes_{\iota \in J_F}L_{k-2}(A)$ with a left action by 
\begin{align*}
g\bullet P=\det(g)^{(2-k)t}\rho(g)P,
\end{align*}
where $\rho=\otimes_{\iota\in J_F}\rho_{k-2}$. 
Note that $G(\A_f)$ naturally acts on $L_{kt}(A\otimes_{\mathfrak{o}_{F'}}\A_{F',f})$. 
We consider the $i$-part $L_{kt,i}(A)$ of $L_{kt}(A)$ defined by 
\[
L_{kt,i}(A)=L_{kt}(A\otimes_{\mathfrak{o}_{F'}}F')
\cap x_i\bullet L_{kt}(A\otimes_{\mathfrak{o}_{F'}}\widehat{\mathfrak{o}}_{F'}).
\]

Hereafter, in this subsection, we abbreviate $L_{kt,i}(A)$ to $L_{kt}(A)$. 
For a holomorphic function $h$ on $\mathfrak{H}^{J_F}$, 
we define an $L_{kt}(\C)$-valued holomorphic differential $n$-form $\omega(h)$ on $\mathfrak{H}^{J_F}$ by 
\begin{align}\label{n-form}
\omega(h)=h(z)\otimes
\bigotimes_{\iota\in J_F}
\begin{bmatrix}
z_{\iota}\\
1
\end{bmatrix}^{k-2}
dz_{{}_{J_F}}.
\end{align}
If $h\in M_k(\Gamma,\C)$, then we have 
\[
g^{\ast}\omega(h)
=\det(g)^{(2-k)t}\rho(g)(h\vert g)(z)\otimes
\bigotimes_{\iota\in J_F}
\begin{bmatrix}
z_{\iota}\\
1
\end{bmatrix}^{k-2}
dz_{{}_{J_F}}
\]
for $g \in \GL_2(\mathbb{R})^{J_F}_+$.
Here $(h\vert g)(z):=\det(g)^{(k-1)t}j(g,z)^{-kt}h(gz)$.
Since $(h\vert \gamma) (z) =h(z)$ for $\gamma\in\Gamma$ and the center $\Gamma\cap F^{\times}$ of $\Gamma$ acts trivially on $L_{kt}(\C)$, we obtain the pull-back formula 
\begin{align}\label{pull-back formula}
\gamma^{\ast}\omega(h)=\gamma\bullet \omega(h)
\end{align}
for any $\overline{\gamma} \in \overline{\Gamma}$ and a lift $\gamma\in\Gamma$ of $\overline{\gamma}$. 

Hereafter, in this subsection, we put $J_F=\{\iota_1,\cdots,\iota_n\}$ and 
$z_i=z_{\iota_i}$ for $z\in \mathfrak{H}^{J_F}$.
We fix a base point $w=(w_i)_{1 \le i \le n}\in \mathfrak{H}^{J_F}$. 
We define an $L_{kt}(\C)$-valued holomorphic function by 
\begin{align}\label{hol function}
F(z)=\integral{w_1}{z_1}{w_n}{z_n}\omega(h). 
\end{align}
For $\overline{\gamma} \in \overline{\Gamma}$, we define a function $\overline{\gamma}\ast F$ on $\mathfrak{H}^{J_F}$ by $\overline{\gamma}\ast F(z)=\gamma \bullet F(\gamma^{-1}z)$. We note that 
\begin{align*}
&\frac{\partial}{\partial z_1}\cdots \frac{\partial}{\partial z_n} 
(\overline{\gamma} \ast F-F)(z)=0. 
\end{align*}
\begin{lem}$($\cite[Chapter V, Lemma 5.1]{Yo}$)$. \label{Yoshida lemma}
Let $D$ be an open contractible domain in $\C^n$. 
Let $f$ be a holomorphic function on $D$. 

$(1)$
Assume that 
\[
\frac{\partial}{\partial z_1}\cdots \frac{\partial}{\partial z_n} f(z)=0
\ \text{ for }\ z=(z_1,\cdots,z_n)\in D. 
\]
Then there exist holomorphic functions $g_i$ on $D$ such that $g_i(z)$ is independent of $z_i$ and $f$ is decomposed into 
$f(z)=\sum_{1 \le i \le n} g_i(z)$. 

$(2)$
Moreover, assume that $f$ is decomposed into $f(z)=\sum_{1 \le i \le n} g_i(z)$ as $(1)$ and $n\ge 2$. If $f(z)$ is independent of $z_1$, then there exist holomorphic functions $h_i$ on $D$ such that $h_i(z)$ is independent of $z_1$ and $z_i$ and $f$ is decomposed into 
$f(z)=\sum_{2 \le i \le n} h_i(z)$. 
\end{lem}
By applying Lemma \ref{Yoshida lemma} $(1)$ to $-(\overline{\gamma}\ast F-F)$, 
we obtain a decomposition 
\begin{align}\label{gamma^1}
-(\overline{\gamma}\ast F-F)(z)=\sum_{1 \le i \le n} g_i^{(1)}(\overline{\gamma})(z), 
\end{align}
where $g_i^{(1)}(\overline{\gamma})$ is a holomorphic function on $\mathfrak{H}^{J_F}$ and $g_i^{(1)}(\overline{\gamma})(z)$ is independent of $z_i$. 

We explicitly describe $g_n^{(1)}(\overline{\gamma})(z)$ as follows. We have 
\begin{align}\label{cal^1}
(\overline{\gamma} \ast F-&F)(z)
=\integral{\gamma w_1}{z_1}{\gamma w_{n-1}}{z_{n-1}}\left(\int_{w_n}^{z_n}+\int_{\gamma w_n}^{w_n}\right)\omega(h)-\integral{w_1}{z_1}{w_n}{z_n}\omega(h)\\
&\nonumber=\inti{\gamma w_1}{z_1}{\gamma w_{n-1}}{z_{n-1}}{\gamma w_n}{w_n}\omega(h)
+\left(\integral{\gamma w_1}{z_1}{\gamma w_{n-1}}{z_{n-1}}-\integral{w_1}{z_1}{w_{n-1}}{z_{n-1}}\right)\int_{w_n}^{z_n}\omega(h).
\end{align}
By applying Lemma \ref{Yoshida lemma} (1) to the second term in the second line of (\ref{cal^1}), 
we can choose $-g_n^{(1)}(\overline{\gamma})(z)$ as the first term in the second line of (\ref{cal^1}):
\begin{align}\label{g_n^1}
g_n^{(1)}(\overline{\gamma})(z)=\inti{\gamma w_1}{z_1}{\gamma w_{n-1}}{z_{n-1}}{w_n}{\gamma w_n}\omega(h). 
\end{align}
By regarding (\ref{gamma^1}) as a $1$-cochain, we obtain 
\begin{align*}
dg_n^{(1)}(\overline{\gamma_1},\overline{\gamma_2})(z)
=-\sum_{1 \le i \le n-1}dg_i^{(1)}(\overline{\gamma_1},\overline{\gamma_2})(z)
\end{align*}
for $\overline{\gamma_1},\overline{\gamma_2} \in \overline{\Gamma}$, 
where $d$ is the boundary map in group cohomology. 
Since the left-hand side is independent of $z_n$ and $dg_i^{(1)}(\overline{\gamma_1},\overline{\gamma_2})(z)$ is independent of $z_i$, by Lemma \ref{Yoshida lemma} (2), we obtain a decomposition 
\begin{align}\label{gamma^2}
dg_n^{(1)}(\overline{\gamma_1},\overline{\gamma_2})(z)
=\sum_{1 \le i \le n-1} g_i^{(2)}(\overline{\gamma_1},\overline{\gamma_2})(z), 
\end{align}
where $g_i^{(2)}(\overline{\gamma_1},\overline{\gamma_2})$ is a holomorphic function and $g_i^{(2)}(\overline{\gamma_1},\overline{\gamma_2})(z)$ is independent of $z_i$. 

We explicitly describe $g_{n-1}^{(2)}(\overline{\gamma_1},\overline{\gamma_2})(z)$ as follows. 
A direct calculation shows that 
\begin{align}\label{cal^2}
\overline{\gamma_1}\ast &g_n^{(1)}(\overline{\gamma_2})(z)-g_n^{(1)}(\overline{\gamma_1\gamma_2})(z)+g_n^{(1)}(\overline{\gamma_1})(z)\\
&\nonumber=\inti{\gamma_1\gamma_2 w_1}{z_1}{\gamma_1 \gamma_2 w_{n-1}}{z_{n-1}}{\gamma_1 w_n}{\gamma_1\gamma_2 w_n}\omega(h)\\
&\nonumber\hphantom{=}\ \ -\inti{\gamma_1\gamma_2 w_1}{z_1}{\gamma_1 \gamma_2 w_{n-1}}{z_{n-1}}{w_n}{\gamma_1\gamma_2 w_n}\omega(h)
+\inti{\gamma_1 w_1}{z_1}{\gamma_1 w_{n-1}}{z_{n-1}}{w_n}{\gamma_1 w_n}\omega(h)\\
&\nonumber=\inti{\gamma_1\gamma_2 w_1}{z_1}{\gamma_1 \gamma_2 w_{n-1}}{z_{n-1}}{\gamma_1 w_n}{w_n}\omega(h)
+\inti{\gamma_1 w_1}{z_1}{\gamma_1 w_{n-1}}{z_{n-1}}{w_n}{\gamma_1 w_n}\omega(h)\\
&\nonumber=\inti{\gamma_1 \gamma_2 w_1}{z_1}{\gamma_1 \gamma_2 w_{n-2}}{z_{n-2}}{\gamma_1 \gamma_2 w_{n-1}}{\gamma_1 w_{n-1}}\int_{\gamma_1 w_n}^{w_n}\omega(h)\\
&\nonumber\hphantom{=}\ \ 
+\left(\integral{\gamma_1\gamma_2 w_1}{z_1}{\gamma_1\gamma_2 w_{n-2}}{z_{n-2}}-\integral{\gamma_1 w_1}{z_1}{\gamma_1 w_{n-2}}{z_{n-2}}\right)
\int_{\gamma_1 w_{n-1}}^{z_{n-1}}\int_{\gamma_1 w_n}^{w_n}
\omega(h). 
\end{align}
By the same argument as above, we can choose $g_{n-1}^{(2)}(\overline{\gamma_1},\overline{\gamma_2})(z)$ as the first term in the last equation of (\ref{cal^2}):
\begin{align}\label{g_{n-1}^2}
g_{n-1}^{(2)}(\overline{\gamma_1},\overline{\gamma_2})(z)=\inti{\gamma_1\gamma_2w_1}{z_1}{\gamma_1\gamma_2 w_{n-2}}{z_{n-2}}{\gamma_1 w_{n-1}}{\gamma_1\gamma_2 w_{n-1}}\int_{w_n}^{\gamma_1 w_n}\omega(h). 
\end{align}
By repeating this argument, we obtain a decomposition 
\begin{align}
(-1)^{m}dg_{n-m+2}^{(m-1)}(\overline{\gamma_1},\cdots,\overline{\gamma_{m}})(z)=\sum_{1 \le i \le n-m+1} g_i^{(m)}(\overline{\gamma_1},\cdots,\overline{\gamma_m})(z)
\end{align}
for each $m$ with $2\le m\le n-1$ and $\overline{\gamma_1},\cdots \overline{\gamma_m} \in \overline{\Gamma}$, where 
\begin{align*}
g_{n-m+1}^{(m)}(\overline{\gamma_1},\cdots,\overline{\gamma_m})(z)
=\integral{\gamma_1\cdots\gamma_m w_1}{z_1}{\gamma_1\cdots\gamma_m w_{n-m}}{z_{n-m}}\inti{\gamma_1\cdots\gamma_{m-1} w_{n-m+1}}{\gamma_1\cdots\gamma_m w_{n-m+1}}{\gamma_1 w_{n-1}}{\gamma_1\gamma_2 w_{n-1}}{w_n}{\gamma_1 w_n}\omega(h).
\end{align*}
Hence we have an explicit formula
\begin{align*}
dg_2^{(n-1)}(\overline{\gamma_1},\cdots,\overline{\gamma_n})(z)=\inti{\gamma_1\cdots \gamma_{n-1} w_1}{\gamma_1\cdots \gamma_{n} w_1}{\gamma_1 w_{n-1}}{\gamma_1\gamma_2 w_{n-1}}{w_n}{\gamma_1 w_n}\omega(h).
\end{align*}
Therefore we obtain the following theorem (1) because it is a constant function: 
\begin{propdef}\label{def:cocycle}
Let $h\in M_{k}(\Gamma,\C)$ and $w=(w_i)_{1\le i\le n}\in \mathfrak{H}^{J_F}$ a base point. 

$(1)$
For $\overline{\gamma_i}\in \overline{\Gamma}$ and a lift $\gamma_i\in\Gamma$ of $\overline{\gamma_i}$, the following map 
$\pi_{h,w} \colon \overline{\Gamma}^n \to L_{kt}(\C)$ is an $n$-cocycle$:$
\begin{align*}
\pi_{h,w}(\overline{\gamma_1},\cdots,\overline{\gamma_n})=\inti{\gamma_1\cdots \gamma_{n-1} w_1}{\gamma_1\cdots \gamma_{n} w_1}{\gamma_1 w_{n-1}}{\gamma_1\gamma_2 w_{n-1}}{w_n}{\gamma_1 w_n}\omega(h).
\end{align*}

$(2)$ 
The cohomology class $[\pi_h]\in H^n(\overline{\Gamma},L_{kt}(\C))$ defined by $\pi_{h,w}$ does not depend on the choice of the base point $w\in \mathfrak{H}^{J_F}$.
\end{propdef}
\begin{proof}
The assertion (2) follows from \cite[Theorem 5.2]{Yo}. 
\end{proof}

%
\subsection{Integrality of $n$-cocycles at boundary}\label{constant term of E}
%

We keep the notation in \S \ref{Cocycle ass to HMF}.
In this subsection, for $E \in M_k(\Gamma,\integer{})$, 
we prove the integrality of the image of $[\pi_E]$ under the restriction map.

We define two maps $b_1$ and $b_2$ from $\overline{G(\Q)_+}^n$ to $L_{kt}(\C)$ by
\begin{align*}
&b_1(\overline{\gamma_1},\cdots,\overline{\gamma_n})\\
&=\gamma_1\bullet\inti{\gamma_2 \cdots \gamma_{n-1} w_1}{\gamma_2\cdots \gamma_{n} w_1}{w_{n-1}}{\gamma_2 w_{n-1}}{w_n}{\sqrt{-1}\infty}\omega(\widetilde{E})
-\gamma_1\bullet\inti{\gamma_2 \cdots \gamma_{n-1} w_1}{\gamma_2\cdots \gamma_{n} w_1}{w_{n-1}}{\gamma_2 w_{n-1}}{0}{w_n}
\omega(a_{\infty}(0,E)),\\
&b_2(\overline{\gamma_1},\cdots,\overline{\gamma_n})\\
&=\inti{\gamma_1 \cdots \gamma_{n-1} w_1}{\gamma_1\cdots \gamma_{n} w_1}{\gamma_1 w_{n-1}}{\gamma_1\gamma_2 w_{n-1}}{w_n}{\sqrt{-1}\infty}\omega(\widetilde{E})
-\inti{\gamma_1 \cdots \gamma_{n-1} w_1}{\gamma_1\cdots \gamma_{n} w_1}{\gamma_1 w_{n-1}}{\gamma_1\gamma_2 w_{n-1}}{0}{w_n}\omega(a_{\infty}(0,E)),
\end{align*}
where $\widetilde{E}(z)=E(z)-a_{\infty}(0,E)$. 
We note that the same argument as in the proof of Proposition \ref{abs conv} shows that $b_1(\overline{\gamma_1},\cdots,\overline{\gamma_n})$ and $b_2(\overline{\gamma_1},\cdots,\overline{\gamma_n})$ converge absolutely. 

We define a new map $\pi_{E,w}^{\mathrm{b}}:\overline{G(\Q)_+}^n \to L_{kt}(\C)$ by
\begin{align}\label{new n-cocycle}
\pi_{E,w}^{\mathrm{b}}(\overline{\gamma_1},\cdots,\overline{\gamma_n})
=\pi_{E,w}(\overline{\gamma_1},\cdots,\overline{\gamma_n})
+b_1(\overline{\gamma_1},\cdots,\overline{\gamma_n})-b_2(\overline{\gamma_1},\cdots,\overline{\gamma_n}).
\end{align}
\begin{prop}\label{prop:cohomologous to E}
For $E \in M_k(\Gamma,\C)$, 
the map $\pi_{E,w}^{\mathrm{b}}$ satisfies the following properties. 
\begin{enumerate}[$(1)$]
\item The value $\pi_{E,w}^{\mathrm{b}}(\overline{\gamma_1},\cdots,\overline{\gamma_n})$ is independent on $w_n$. 
\item $\pi_{E,w}^{\mathrm{b}}$ is cohomologous to $\pi_{E,w}$. 
\end{enumerate}
\end{prop}
\begin{proof}
The assertion (1) follows from a direct calculation.
For the proof of (2), we put 
\begin{align*}
v&(\overline{\gamma_1},\cdots,\overline{\gamma_{n-1}})\\
&=\inti{\gamma_{1}\cdots \gamma_{n-2} w_1}{\gamma_{1}\cdots \gamma_{n-1} w_1}{w_{n-1}}{\gamma_{1} w_{n-1}}{w_n}{\sqrt{-1}\infty}
\omega(\widetilde{E})
-\inti{\gamma_{1}\cdots \gamma_{n-2} w_1}{\gamma_{1}\cdots \gamma_{n-1} w_1}{w_{n-1}}{\gamma_{1} w_{n-1}}{0}{w_n}\omega(a_{\infty}(0,E)).
\end{align*}

Now the assertion follows from the following claim:
\begin{align}\label{claim at boundary}
dv(\overline{\gamma_1},\cdots,\overline{\gamma_n})
&=\pi_{E,w}^{\mathrm{b}}(\overline{\gamma_1},\cdots,\overline{\gamma_n})-\pi_{E,w}(\overline{\gamma_1},\cdots,\overline{\gamma_n}). 
\end{align}

For the proof of the claim (\ref{claim at boundary}), it suffices to prove the following:
\begin{align*}
&(\mathrm{i})\ \ \ \ \overline{\gamma_1}\bullet v(\overline{\gamma_2},\cdots,\overline{\gamma_{n}})
=b_1(\overline{\gamma_1}\cdots \overline{\gamma_{n}});\\
&(\mathrm{ii})\sum_{1\le n-j\le n-1}(-1)^{n-j}v(\overline{\gamma_1},\cdots,\overline{\gamma_{n-j}\gamma_{n-j+1}},\cdots,\overline{\gamma_n})+(-1)^{n} v(\overline{\gamma_1},\cdots,\overline{\gamma_{n-1}})
=-b_2(\overline{\gamma_1}\cdots \overline{\gamma_{n}}). 
\end{align*}

The assertion (i) follows from a direct calculation.
For the proof of (ii), it suffices to show the following ${(\ast)}_k$ by induction on $1\le n-k\le n-1$:
\begin{align*}
{(\ast)}_k\ \ \ \ 
&\sum_{n-k\le n-j\le n-1}(-1)^{n-j}v(\overline{\gamma_1},\cdots,\overline{\gamma_{n-j}\gamma_{n-j+1}},\cdots,\overline{\gamma_n})+(-1)^{n} v(\overline{\gamma_1},\cdots,\overline{\gamma_{n-1}})\\
&\ \ =(-1)^{n-k}\bigg\{
\int_{\gamma_1\cdots \gamma_{n-1} w_1}^{\gamma_1\cdots \gamma_n w_1}\cdots
\intii
{\gamma_{1}\cdots \gamma_{n-k} w_k}{\gamma_{1}\cdots \gamma_{n-k+1} w_k}
{\gamma_{1}\cdots \gamma_{n-k-2} w_{k+1}}{\gamma_{1}\cdots \gamma_{n-k-1} w_{k+1}}
{w_n}{\sqrt{-1}\infty}
\omega(\widetilde{E})\\
&\ \ \ \ \ \ \ \ \ \ \ \ \ \ \ \ \ \ \ \ 
-\int_{\gamma_1\cdots\gamma_{n-1} w_1}^{\gamma_1\cdots\gamma_n w_1}\cdots
\intii
{\gamma_{1}\cdots \gamma_{n-k} w_k}{\gamma_{1}\cdots \gamma_{n-k+1} w_k}
{\gamma_{1}\cdots \gamma_{n-k-2} w_{k+1}}{\gamma_{1}\cdots \gamma_{n-k-1}w_{k+1}}
{0}{w_n}
\omega(a_{\infty}(0,E))\bigg\}. 
\end{align*}

Suppose that $k=1$. A direct calculation shows $(\ast)_1$:
\begin{align*}
&(-1)^{n-1}v(\overline{\gamma_1},\cdots,\overline{\gamma_{n-1}\gamma_n})
+(-1)^{n}v(\overline{\gamma_1},\cdots,\overline{\gamma_{n-1}})\\
&=(-1)^{n-1}\bigg\{
\left(
\int_{\gamma_{1}\cdots \gamma_{n-2} w_{1}}^{\gamma_{1}\cdots \gamma_{n-1}\gamma_n w_{1}}
+\int_{\gamma_{1}\cdots \gamma_{n-1} w_{1}}^{\gamma_{1}\cdots \gamma_{n-2} w_{1}}
\right)
\int_{\gamma_1\cdots \gamma_{n-3} w_2}^{\gamma_1\cdots\gamma_{n-2} w_2}
\cdots\int_{w_n}^{\sqrt{-1}\infty}
\omega(\widetilde{E}) \\
&\ \ \ \ \ \ \ \ \ \ \ \ \ \ \ \ \ \ \ \  -\left(
\int_{\gamma_{1}\cdots \gamma_{n-2} w_{1}}^{\gamma_{1}\cdots \gamma_{n-1}\gamma_n w_{1}}
+\int_{\gamma_{1}\cdots \gamma_{n-1} w_{1}}^{\gamma_{1}\cdots \gamma_{n-2}w_{1}}
\right)
\int_{\gamma_1\cdots \gamma_{n-3} w_2}^{\gamma_1\cdots\gamma_{n-2} w_2}
\cdots\int_{0}^{w_n}
\omega(a_{\infty}(0,E))\bigg\} \\
&=(-1)^{n-1}\bigg\{
\intii
{\gamma_{1}\cdots \gamma_{n-1}w_1}{\gamma_{1}\cdots \gamma_{n}w_1}
{\gamma_{1}\cdots \gamma_{n-3}w_2}{\gamma_{1}\cdots \gamma_{n-2}w_2}
{w_n}{\sqrt{-1}\infty} \omega(\widetilde{E})\\
&\ \ \ \ \ \ \ \ \ \ \ \ \ \ \ \ \ \ 
-\intii
{\gamma_{1}\cdots \gamma_{n-1}w_1}{\gamma_{1}\cdots \gamma_{n}w_1}
{\gamma_{1}\cdots \gamma_{n-3}w_2}{\gamma_{1}\cdots \gamma_{n-2}w_2}
{0}{w_n}\omega(a_{\infty}(0,E))\bigg\}. 
\end{align*}

Suppose ${(\ast)}_k$. 
By adding $(-1)^{n-k-1}v(\overline{\gamma_1},\cdots,\overline{\gamma_{n-k-1}\gamma_{n-k}},\cdots,\overline{\gamma_n})$ to ${(\ast)}_k$, we get ${(\ast)}_{k+1}$:
\begin{align*}
&\sum_{n-k-1\le n-j\le n-1}(-1)^{n-j}v(\overline{\gamma_1},\cdots,\overline{\gamma_{n-j}\gamma_{n-j+1}},\cdots,\overline{\gamma_n})+(-1)^{n} v(\overline{\gamma_1},\cdots,\overline{\gamma_{n-1}})\\
&=(-1)^{n-k-1}\bigg\{
\int_{\gamma_{1}\cdots \gamma_{n-1}w_1}^{\gamma_{1}\cdots \gamma_{n}w_1}
\cdots
\left(
\int_{\gamma_{1}\cdots \gamma_{n-k-2}w_{k+1}}^{\gamma_{1}\cdots \gamma_{n-k-1}\gamma_{n-k}w_{k+1}}
+\int_{\gamma_{1}\cdots \gamma_{n-k-1}w_{k+1}}^{\gamma_{1}\cdots \gamma_{n-k-2}w_{k+1}}
\right)
\cdots\int_{w_n}^{\sqrt{-1}\infty}
\omega(\widetilde{E})\\
&\ \ \ \ \ \ \ \ \ \ \ \ \ \ \ 
-\int_{\gamma_{1}\cdots \gamma_{n-1}w_1}^{\gamma_{1}\cdots \gamma_{n}w_1}
\cdots
\left(
\int_{\gamma_{1}\cdots \gamma_{n-k-2}w_{k+1}}^{\gamma_{1}\cdots \gamma_{n-k-1}\gamma_{n-k}w_{k+1}}
+\int_{\gamma_{1}\cdots \gamma_{n-k-1}w_{k+1}}^{\gamma_{1}\cdots \gamma_{n-k-2}w_{k+1}}
\right)
\cdots\int_{0}^{w_n}
\omega(a_{\infty}(0,E))\bigg\} \\
&=(-1)^{n-k-1}\bigg\{
\int_{\gamma_{1}\cdots \gamma_{n-1}w_1}^{\gamma_{1}\cdots \gamma_{n}w_1}
\cdots
\int_{\gamma_{1}\cdots \gamma_{n-k-1}w_{k+1}}^{\gamma_{1}\cdots \gamma_{n-k}w_{k+1}}
\int_{\gamma_{1}\cdots \gamma_{n-k-3}w_{k+2}}^{\gamma_{1}\cdots \gamma_{n-k-2}w_{k+2}}
\cdots\int_{w_n}^{\sqrt{-1}\infty}
\omega(\widetilde{E})\\
&\ \ \ \ \ \ \ \ \ \ \ \ \ \ \ \ \ \ \ \ \ 
-\int_{\gamma_{1}\cdots \gamma_{n-1}w_1}^{\gamma_{1}\cdots \gamma_{n}w_1}
\cdots
\int_{\gamma_{1}\cdots \gamma_{n-k-1}w_{k+1}}^{\gamma_{1}\cdots \gamma_{n-k}w_{k+1}}
\int_{\gamma_{1}\cdots \gamma_{n-k-3}w_{k+2}}^{\gamma_{1}\cdots \gamma_{n-k-2}w_{k+2}}
\cdots\int_{0}^{\omega_n}
\omega(a_{\infty}(0,E))\bigg\}. 
\end{align*}
\end{proof}

By using Proposition \ref{prop:cohomologous to E}, we prove the main result of this section:
\begin{prop}\label{const}
Assume that $h_F^+=1$ and $\ideal{d}_F [t_1]=\ideal{o}_F$. 
Let $\Phi_p$ be the composite field of $\iota_p((F')^{\iota}(\sqrt{-1}))$ in $\overline{\Q}_p$ for all $\iota\in J_F$ and $\integer{}$ the ring of integers of a finite extension $K$ of $\Phi_p$. 
Here $\iota_p: \overline{\Q} \to \overline{\Q}_p$ is the fixed embedding and $F'$ is the field defined in \S \ref{subsection:Geometric HMF}. 
Let $k$ be an even integer such that $2\le k\le p$. 
Let $E\in M_k(\Gamma,\integer{})$.
Then we have the following properties:

$(1)$
The cohomology class $\res([\pi_E])$ is integral, that is, 
$\res([\pi_E]) \in \bigoplus_{s\in C(\Gamma)}\widetilde{H}^n(\overline{\Gamma_s},L_{kt,1}(\integer{}))$.
Here $\widetilde{H}^n(\overline{\Gamma_s},L_{kt,1}(\integer{}))$ denotes the image of $H^n(\overline{\Gamma_s},L_{kt,1}(\integer{}))\to H^n(\overline{\Gamma_s},L_{kt,1}(K))$. 

$(2)$
Suppose that $E$ vanishes at a cusp $s\in C(\Gamma)$. Then we have 
$\res([\pi_E])=0$ in $\widetilde{H}^n(\overline{\Gamma_s},L_{kt,1}(\integer{}))$.
\end{prop}
\begin{proof}
We treat the case $s=\infty$. 
Let $\overline{\gamma_1},\cdots,\overline{\gamma_n}\in \overline{\Gamma_{\infty}}$.
By Proposition \ref{prop:cohomologous to E} (1), the value
$\pi_{E,w}^{\mathrm{b}}(\overline{\gamma_1},\cdots,\overline{\gamma_n})$ defined by (\ref{new n-cocycle}) is independent on $w_n$. 
The first terms of $b_1(\overline{\gamma_1},\cdots,\overline{\gamma_n})$ and $b_2(\overline{\gamma_1},\cdots,\overline{\gamma_n})$ converge to $0$ when $w_n$ tends to $\sqrt{-1} \infty$. 
Hence we obtain 
\begin{align*}
&\pi_{E,(\sqrt{-1},\cdots,\sqrt{-1},w_n)}^{\mathrm{b}}(\overline{\gamma}_1,\cdots,\overline{\gamma}_n)\\
&=\lim_{w_n \to \sqrt{-1}\infty}\int_{\gamma_1\cdots\gamma_{n-1}\sqrt{-1}}^{\gamma_1\cdots\gamma_n\sqrt{-1}}\cdots \int_{\gamma_1\sqrt{-1}}^{\gamma_1\gamma_2\sqrt{-1}}
\left(\int_{w_n}^{\gamma_1 w_n}-\int_{\gamma_1 0}^{\gamma_1 w_n}+\int_0^{w_n}\right)\omega(a_{\infty}(0,E))\\
&=\inti{\gamma_1\cdots \gamma_{n-1}\sqrt{-1}}{\gamma_1\cdots \gamma_n\sqrt{-1}}{\gamma_1\sqrt{-1}}{\gamma_1\gamma_2\sqrt{-1}}{0}{\gamma_1 0}\omega(a_{\infty}(0,E)).
\end{align*}
Here the first equality follows from that 
$\gamma_1 \bullet \omega(a_{\infty}(0,E))=\gamma_1^{\ast} \omega(a_{\infty}(0,E))$, where we use the assumption that $k$ is even. 
Since $\gamma_i\in \GL_2(\ideal{o}_F)\cap B_{\infty,+}$ for each $i$, $\gamma_1\cdots \gamma_j c$ belongs to $\integer{}$ for each $j$ and $c\in \{\sqrt{-1},0\}$. 
Thus $\pi_{E,(\sqrt{-1},\cdots,\sqrt{-1},w_n)}^{\mathrm{b}}(\overline{\gamma}_1,\cdots,\overline{\gamma}_n)$ belongs to $L_{kt,1}(\integer{})$, where we use the assumption that $k\le p$. 
Hence the image of $\res([\pi_E])$ in the $\infty$-part is integral.

We treat the general case $s\in C(\Gamma)$.
By the assumption $h_F^+=1$, we can take $\alpha \in \GL_2(\ideal{o}_F)$ such that $s=\alpha(\infty)$.
Let $\overline{\gamma_1},\cdots,\overline{\gamma_n}\in \overline{\Gamma_{s}}$.
Let $B_1,\cdots,B_n\in \GL_2(\ideal{o}_F) \cap B_{\infty,+}$ such that $\gamma_i=\alpha B_i \alpha^{-1}$.
By the pull-back formula $\alpha^{\ast} \omega(E)=\alpha\bullet\omega(E|\alpha)$, we have 
\[
\pi_{E,(\alpha \sqrt{-1},\cdots,\alpha \sqrt{-1},\alpha w_n )}(\overline{\gamma_1},\cdots,\overline{\gamma_n})
=\alpha\bullet\pi_{E|\alpha,(\sqrt{-1},\cdots,\sqrt{-1},w_n )}(B_1,\cdots,B_n).
\]
Now the same argument as in the case $s=\infty$ replacing $E$ by $E|\alpha$ shows that the image of $\res([\pi_E])$ in the $s$-part is integral.
\end{proof}

\section{Eisenstein cohomology and Eichler--Shimura--Harder isomorphism}\label{subsection:Eisenstein cohomology and ESH}
%
The purpose of this section is to recall theory of Eisenstein cohomology and the Eichler\nobreakdash--Shimura--Harder isomorphism. 
We use the assumption $h_F^+=1$ to prove the Eichler\nobreakdash--Shimura--Harder isomorphism (\ref{+,+ decomp}) by using an explicit computation of the action of the Weyl group.

%
\subsection{Eisenstein cohomology}
%

The purpose of this subsection is to prove the following proposition, where $Y=Y(\ideal{n})$ and $H_{\mathrm{Eis}}^n(Y,\C)$ is the Eisenstein cohomology defined by (\ref{definition of Eisenstein cohomology}).

\begin{prop}\label{Eis Hodge number}
$(1)$ The Hodge number of $H_{\mathrm{Eis}}^n(Y,\C)$ is equal to $n$, that is, $H_{\mathrm{Eis}}^n(Y,\C)= F^n H_{\mathrm{Eis}}^n(Y,\C)$.

$(2)$ $H_{\mathrm{Eis}}^n(Y,\C)$ is stable under the Hecke correspondences.
\end{prop}
\begin{proof}
(1) The assertion for $H_{\mathrm{Eis}}^n(Y^1(\ideal{n}),\C)$ is obtained by E. Freitag (\cite[Chapter III, Proposition 3.5 and Theorem 4.9]{Fre}). 
We follow the arguments in the method of Freitag. 
We fix $i$ with $1\le i \le h_F^+$ and abbreviate $\Gamma_{1}(\ideal{d}_F [t_i],\ideal{n})$ to $\Gamma$.
We put $J_F=\{\iota_1,\cdots,\iota_n\}$ and $z_i=z_{\iota_i}$ for $z\in \mathfrak{H}^{J_F}$.
For $z\in \mathfrak{H}^{J_F}$, we put $N(z)=\prod_{1\le i\le n}z_i$.

The proof consists of three steps. 

\textbf{Step1}: To give a basis of $H^{n-1}(\overline{\Gamma_{t}},\C)$ and $H^{n}(\overline{\Gamma_{t}},\C)$ over $\C$ for each cusp $t\in C(\Gamma)$.
(Here we use a basis of $H^{n-1}(\overline{\Gamma_{t}},\C)$ to prove Proposition \ref{H^{n-1} vanishing}). 

We treat the case $t=\infty$. 
We prove that a basis of $H^{n-1}(\overline{\Gamma_{\infty}},\C)$ (resp. $H^{n}(\overline{\Gamma_{\infty}},\C)$) over $\C$ is given by 
\[
\omega_{\infty}^{n-1}=\frac{dy_1}{y_1}\wedge\cdots\wedge\frac{dy_{n-1}}{y_{n-1}} \ (\text{resp.}\ \omega_{\infty}^n=dx_1\wedge\cdots\wedge dx_{n}). 
\]

Put $D=\{z\in \mathfrak{H}^{J_F} \mid N(y)=1\}$.
The group $\Gamma_{\infty}$ which consists of transformations of the form 
$z \mapsto \varepsilon z+b$ with $N(\varepsilon)=1$ acts on $D$. 
We identify $D$ with $\mathbb{R}^{2n-1}$ by 
$z \mapsto (x_1,\cdots,x_n, u_1,\cdots,u_{n-1})$
with coordinates $\{x_i\}_{1 \le i \le n}$ and $\{u_i:=\text{log}(y_i)\}_{1 \le i \le n-1}$. 
Since 
$\overline{\Gamma_{\infty}}\backslash \mathfrak{H}^{J_F}$ is homeomorphic to $\mathbb{R}\times (\overline{\Gamma_{\infty}}\backslash D)$ by $z \mapsto (\text{log}(N(y)),N(y)^{-1/n}z)$, 
the canonical embedding $\overline{\Gamma_{\infty}}\backslash D  \hookrightarrow \overline{\Gamma_{\infty}}\backslash \mathfrak{H}^{J_F}$ is a homotopy equivalence. 
Hence it induces 
$H^{\ast}(\overline{\Gamma_{\infty}},\C) \simeq H^{\ast}(\overline{\Gamma_{\infty}}\backslash \mathfrak{H}^{J_F},\C) \simeq H^{\ast}(\overline{\Gamma_{\infty}}\backslash D,\C)$. 

We consider 
a $\overline{\Gamma_{\infty}}$-invariant harmonic differential $m$-form $\omega=\sum_{b,c} f_{b,c}(x,u)dx_b \wedge du_c$ on $D$.
By the same argument as in \cite[p.145, 146]{Fre}, 
the functions $f_{b,c}(x,u)$ are independent of $x$, and if $f_{b,c}(x,u)\neq 0$, then $b=\phi$ or $\{1,\cdots,n\}$. 
We treat the case $b=\phi$ (the case $b=\{1,\cdots,n\}$ is similar). 
Since $H^{n-1}(\overline{\Gamma_{\infty}}\backslash D,\C)$ is isomorphic to the de Rham cohomology of a lattice $\text{log}(\mathfrak{o}_{F,+}^{\times})$ of $\mathbb{R}^{n-1}$, 
the same argument as in \cite[p.146]{Fre} shows that $\omega_{\infty}^{n-1}$ is a basis of $H^{n-1}(\overline{\Gamma_{\infty}},\C)$. 

We treat the general case $t\in C(\Gamma)$.
Let $\alpha\in G(\Q)$ be such that $t=\alpha(\infty)$. 
The canonical map $\alpha:D_{\infty} \xrightarrow{\simeq} D_{t}$ induces 
$\overline{(\alpha^{-1}\Gamma \alpha)_{\infty}}\backslash D_{\infty} \xrightarrow{\simeq} \overline{\Gamma_t}\backslash D_t$.
Now our assertion follows from the same argument as in the case $t=\infty$ by replacing $\Gamma$ by $\alpha^{-1}\Gamma \alpha$ (cf. \cite[p.154]{Fre}).

\textbf{Step2}: To construct the Eisenstein operator 
$E: \bigoplus_{t\in C(\Gamma)}H^{n}(\overline{\Gamma_{t}},\C) \to H^{n}(\overline{\Gamma},\C)$.

We may assume $t=\infty$ by the same argument as in Step 1. 
As mentioned in the proof of \cite[Chapter III, Remark 3.1]{Fre}, 
$\omega_{\infty}^n$ is cohomologous to $dz_1\wedge\cdots\wedge dz_n$ up to a constant factor. 
We put $\omega_{\infty}=dz_1\wedge\cdots\wedge dz_n$. 
The Eisenstein operator $E$ is defined by symmetrization 
\[
E(\omega_{\infty})
=\sum_{M\in  \Gamma_{\infty}\backslash \Gamma}M^{\ast}\omega_{\infty}
:=\lim_{s\to 0}\sum_{M\in  \Gamma_{\infty}\backslash \Gamma}|N(j(M,z))|^{-2s}M^{\ast}\omega_{\infty}
\]
(cf. \cite[Chapter III, Proposition 3.3]{Fre}). 
Here, for $M=\begin{pmatrix}a&b\\c&d\end{pmatrix}\in \Gamma$, $j(M,z)=(c_iz_i+d_i)_{1\le i\le n}$ and $M^{\ast}\omega_{\infty}=N(j(M,z))^{-2}\omega_{\infty}$.
We note that, by using analytic continuation (\cite[Proposition 3.2]{Shi}) 
of Eisenstein series of the type 
\[
E_{2,0}^{\Gamma}(z,s)=\sum_{M\in \Gamma_{\infty}\backslash \Gamma}N(j(M,z))^{-2}|N(j(M,z))|^{-2s},
\]
$E(\omega_{\infty})$ is expressed by $E(\omega_{\infty})=E_{2,0}^{\Gamma}(z,0)\omega_{\infty}$.
Hence $E$ is well\nobreakdash-defined.

\textbf{Step3}: To show that $E$ is a section of the restriction map $H^{n}(\overline{\Gamma},\C) \to H^{n}(\overline{\Gamma_{t}},\C)$ for every $t\in C(\Gamma)$. 

As discussed in the proof of \cite[Chapter III, Proposition 3.3]{Fre},
our assertion follows from that the constant term of $E_{2,0}^{\Gamma}(z,0)$ at $t$ is equal to $1$ (resp. $0$) if $t$ is $\Gamma$-equivalent to $\infty$ (resp. $t$ is not $\Gamma$-equivalent to $\infty$). 

We define the Eisenstein cohomology $H_{\text{Eis}}^n(Y,\C)$ to be the image of $E$:
\begin{align}\label{definition of Eisenstein cohomology}
H_{\text{Eis}}^n(Y,\C)=\im (E). 
\end{align}
Hence the Hodge number of $H_{\text{Eis}}^n(Y,\C)$ is equal to $n$ because $E_{2,0}^{\Gamma}(z,0)$ is holomorphic.

(2) 
Let us fix $\Gamma'=\Gamma_1(\ideal{d}_F[t_j],\ideal{n})$ and $\alpha\in G(\Q)$ such that $\Gamma\alpha\Gamma'$ is expressed as a finite disjoint union $\Gamma\alpha\Gamma'=\coprod_{i\in I}\Gamma \alpha_i$. 
For the proof, it suffices to show that
\begin{align}\label{Hecke on Eis}
E(\omega_t)|[\Gamma \alpha \Gamma']=E(\omega_t|[\Gamma \alpha \Gamma']). 
\end{align}

We may assume $t=\infty$ by the same argument as in Step 1 and 2. 
By the definition of $E$, the left-hand side of (\ref{Hecke on Eis}) is equal to 
\begin{align}\label{claim of LHS}
E(\omega_{\infty})|[\Gamma \alpha \Gamma']
&=\lim_{s\to 0}
\sum_{\gamma\in \Gamma_{\infty}\backslash \Gamma\alpha \Gamma'} 
|N(j(\gamma,z))|^{-2s}
\gamma^{\ast} \omega_{\infty}.
\end{align}

We consider the right-hand side of (\ref{Hecke on Eis}). 
For each $s\in \mathbb{P}^1(F)$, 
we put $\mathscr{S}_s=\{\gamma\in \Gamma_{\infty}\backslash \Gamma \alpha \Gamma' \mid \gamma(s)=\infty \}$.
Note that $\Gamma_{\infty}\backslash \Gamma\alpha\Gamma'
=\coprod_{s\in \mathbb{P}^1(F)} \mathscr{S}_s$.
Since, for each $s\in \mathbb{P}^1(F)$, 
there exist a unique $t\in C(\Gamma')$ and a unique $M\in \Gamma_t'\backslash \Gamma'$ such that $M(s)=t$, we have 
\begin{align}\label{bijective}
\Gamma_{\infty}\backslash \Gamma\alpha\Gamma'=
\coprod_{t\in C(\Gamma')}\coprod_{M\in \Gamma_t'\backslash \Gamma'}
\mathscr{S}_{M^{-1}(t)}.
\end{align}
We put $(\omega_t')_{t\in C(\Gamma')}=\omega_{\infty}|[\Gamma\alpha\Gamma']$. 
We claim that 
\begin{align}\label{w_t'}
\omega_t'=\sum_{\gamma\in \mathscr{S}_t} \gamma^{\ast}\omega_{\infty}.
\end{align}
For the moment, we admit the claim (\ref{w_t'}). 
Hence we obtain 
\begin{align*}
E(\omega_{\infty}|[\Gamma \alpha \Gamma'])
&=\sum_{t\in C(\Gamma')}\sum_{\gamma\in \mathscr{S}_t}
E(\gamma^{\ast}\omega_{\infty})\\
\nonumber&=\lim_{s \to 0}\sum_{t\in C(\Gamma')}
\sum_{\gamma\in \mathscr{S}_t}
\sum_{M\in  \Gamma_t'\backslash \Gamma'} 
|N(j(\gamma M,z))|^{-2s}(\gamma M)^{\ast}\omega_{\infty}.
\end{align*}
Here the first equality follows from (\ref{w_t'}) and 
the second equality follows from the definition of $E$. 
Therefore our assertion (\ref{Hecke on Eis}) follows from $\mathscr{S}_t\cdot M=\mathscr{S}_{M^{-1}(t)}$, (\ref{bijective}) and (\ref{claim of LHS}). 

Thus it remains to prove the claim (\ref{w_t'}).
We decompose $\Gamma\alpha\Gamma'$ into a disjoint union:
$\Gamma \alpha \Gamma'
=\coprod_{i\in I^{t}} \Gamma\beta_i \Gamma_t'$ and 
$\Gamma \beta_i \Gamma_t'=\coprod_{j\in J_i}\Gamma\beta_i\delta_{i,j}$ with $\delta_{i,j}\in \Gamma_t'$.
By the definition of the Hecke operator acting on the boundary cohomology \cite[(3.1c)]{Hida93}, we have 
$\omega_t'
=\sum_{i\in I_{\infty}^t}\sum_{j\in J_i} (\beta_i\delta_{i,j})^{\ast}\omega_{\beta_i(t)}$.
Here $I_{\infty}^t=\{i\in I^t\mid \text{$\beta_i(t)$ is $\Gamma$-equivalent to $\infty$}\}$. 
For each $i\in I_{\infty}^t$, 
we may assume that $\beta_i(t)=\infty$.
Now, for the proof of claim, it suffices to show that 
$\mathscr{S}_t
=\coprod_{i\in I_{\infty}^t}\coprod_{j\in J_i}\Gamma_{\infty}\beta_i\delta_{i,j}$.
The inclusion $\supset$ (resp. $\subset$) follows from $\beta_i\delta_{i,j}(t)=\infty$ (resp. the decomposition
$\Gamma_{\infty}\backslash\Gamma \alpha \Gamma'
=\coprod_{i\in I^t}\coprod_{j\in J_i} 
\Gamma_{\infty}\backslash \Gamma\beta_i\delta_{i,j}$).
\end{proof}

%
\subsection{Partial Eichler--Shimura--Harder isomorphism}\label{subsection:Eichler--Shimura--Harder}
%

In this subsection, we prove the Eichler--Shimura--Harder isomorphism (\ref{+,+ decomp}), where we use the assumption $h_F^+=1$. 

Let $K$ be a finite extension of $\Qp$ and $\integer{}$ the ring of integers of $K$. 
For $A=\integer{}$, $K$, or $\C$, let $H_c^{\ast}(Y(\ideal{n}),A)$ denote the compact support cohomology of $Y(\ideal{n})$ with coefficients in $A$, and let $H_{\pa}^{\ast}(Y(\ideal{n}),A)$ denote the parabolic cohomology of $Y(\ideal{n})$ with coefficients in $A$, that is, $H_{\pa}^{\ast}(Y(\ideal{n}),A)=\im\left(H_c^{\ast}(Y(\ideal{n}),A)\to H^{\ast}(Y(\ideal{n}),A)\right)$. 
We have the decomposition 
\begin{align}\label{split}
H^n(Y(\ideal{n}),\C)\simeq 
H_{\pa}^n (Y(\ideal{n}),\C)
\oplus H_{\mathrm{Eis}}^n (Y(\ideal{n}),\C)
\end{align}
(see Step 3 and (\ref{definition of Eisenstein cohomology}) in the proof of Proposition \ref{Eis Hodge number}).
Let $\widetilde{H}_{\pa}^n(Y(\ideal{n}),\integer{})$ denote the image of 
$H_{\pa}^n(Y(\ideal{n}),\integer{})\to H_{\pa}^n(Y(\ideal{n}),K)$.

By \cite[Theorem 1.1]{Hida93}, if the degree $n=[F:\Q]$ is even, 
then the $\C$-vector space 
$H_{\pa}^n (Y(\ideal{n}),\C)/H_{\mathrm{cusp}}^n (Y(\ideal{n}),\C)$ is spanned by the cohomology classes of the invariant forms $\omega_{J'}=\bigwedge_{\iota \in J'}y_{\iota}^{-2}dx_{\iota}\wedge dy_{\iota}$ for all subsets $J'$ of $J_F$ such that $\sharp J'=n/2$, where $H_{\mathrm{cusp}}^n (Y(\ideal{n}),\C)$ denotes the cuspidal cohomology of $Y(\ideal{n})$. 
Both $H_{\pa}^n (Y(\ideal{n}),\C)$ and $H_{\mathrm{cusp}}^n (Y(\ideal{n}),\C)$ are $W_G$\nobreakdash-modules (\cite[$\S$7]{Hida88}). 
We assume that $h_F^+=1$.
As mentioned after Proposition \ref{Modular symbol}, 
for each subset $J$ of $J_F$, the action of $((1_\iota)_{\iota\in J},(-1_{\iota})_{\iota\in J_F\backslash J})\in W_G$ on $Y(\ideal{n})$ is given by 
\begin{align*}
((x_\iota+\sqrt{-1}y_\iota)_{\iota\in J},
(x_\iota+\sqrt{-1}y_\iota)_{\iota\in J_F \backslash J})
\mapsto
\left(\xi^{\iota}(x_\iota+\sqrt{-1}y_\iota)_{\iota\in J}, 
(-\xi)^{\iota}(-x_\iota+\sqrt{-1}y_\iota)_{\iota\in J_F \backslash J}\right)
\end{align*} 
for some $\xi\in \mathfrak{o}_F^{\times}$. 
Thus, in the case $n$ is even, 
if a character $\epsilon$ of $W_G$ satisfies $\sharp \{\iota\in J_F \mid \epsilon(-1_{\iota})=-1\} \neq n/2$, 
then $H_{\pa}^n (Y(\ideal{n}),\C)[\epsilon]=H_{\mathrm{cusp}}^n (Y(\ideal{n}),\C)[\epsilon]$. 
Here, for a $W_G$-module $V$, $V[\epsilon]$ denotes the $\epsilon$-isotypic part $\{v\in V \mid w\cdot v=\epsilon(w)v\ \text{for all}\ w \in W_G \}$. 
Hence we obtain 
\begin{align}\label{+,+ decomp}
H_{\pa}^n (Y(\ideal{n}),\C)[\epsilon]
\simeq H_{\mathrm{cusp}}^n (Y_1(\ideal{n}),\C)[\epsilon]
\simeq S_2(\ideal{n},\C)
\end{align}
as Hecke modules (cf. \cite[\S2, \S3]{Hida94}). 
Thus the Hecke algebra $\mathscr{H}_2(\ideal{n},\integer{})$ is isomorphic to the $\integer{}$-subalgebra of 
$\mathrm{End}_{\integer{}}\left(\widetilde{H}_{\pa}^n(Y(\ideal{n}),\integer{})[\epsilon]\right)$. 
We have the decomposition 
\begin{align*}
H^n(Y(\ideal{n}),\C)[\epsilon] &\simeq 
H_{\pa}^n (Y(\ideal{n}),\C)[\epsilon]
\oplus H_{\mathrm{Eis}}^n (Y(\ideal{n}),\C)[\epsilon].
\end{align*}
By Proposition \ref{Eis Hodge number},
we have a homomorphism $\mathbb{H}_2(\ideal{n},\integer{}) \to \mathrm{End}_{\integer{}}\left(\widetilde{H}^n(Y(\ideal{n}),\integer{})[\epsilon]\right)$. 
For every ideal $I$ of $\mathbb{H}_2(\ideal{n},\integer{})$, 
let $I[\epsilon]$ denote the image of $I$ under this homomorphism.

\section{Rationality and Integrality of cohomology classes}\label{subsection:Rationality and Integrality}
%
The purpose of this section is to prove the rationality (Proposition \ref{prop:rational}) and integrality (Corollary \ref{thm:integral}) of the cohomology class associated to the Hilbert Eisenstein series $\mathbf{E}$ attached to a pair of Hecke characters of $F$ satisfying the following (Eis condition).
We use the assumption $h_F^+=1$ to prove a vanishing result on the cohomology of $D_{C_{\infty}}(\ideal{n})$ (Proposition \ref{H^{n-1} vanishing}). 

Let $\Phi_p$ be the field introduced in Proposition \ref{const}. 
We fix a finite extension $K$ of $\Phi_p$. 
Let $\integer{}$ be the ring of integers of $K$, $\varpi$ a uniformizer, and $\kappa$ the residue field. 

We assume that $h_F^+=1$.
Let $\ideal{n}$ be a non-zero ideal of $\ideal{o}_F$ such that $\ideal{n}$ is prime to $6p\Delta_F $ and $\ideal{d}_F[t_1]$.
Let us fix narrow ray class characters $\varphi$ and $\psi$ of $F$ satisfying $\ideal{m}_{\varphi\psi}=\ideal{m}_{\varphi}\ideal{m}_{\psi}=\ideal{n}$ and 
\begin{align*}
&\text{(Eis condition)}&
&\text{$\varphi$ and $\psi$ are $\cal{O}$-valued and totally even (resp. totally odd), }&\\
&&&\text{$\varphi$ is non-trivial, and the algebraic Iwasawa $\mu$-invariants of}&\\
&&&\text{the splitting fields $\overline{\Q}^{\ker(\varphi)}$ and $\overline{\Q}^{\ker(\psi)}$ are equal to $0$ (see Remark \ref{rem:Fe--Wa}).}&
\end{align*}

Let $\textbf{E}$ denote the Hilbert Eisenstein series $\textbf{E}_2(\varphi,\psi)\in M_2(\ideal{n},\C)$ attached to $\varphi$ and $\psi$ as Proposition \ref{Hilbert Eisenstein}. 
Note that $\textbf{E}$ satisfies (\ref{zero point}) by Proposition \ref{Const of Eisenstein}. 
We define the character $\epsilon_{{}_{\textbf{E}}}$ of $W_G$ by $\epsilon_{{}_{\textbf{E}}}=\sgn^{J_F}$ (resp. $\epsilon_{{}_{\textbf{E}}}=\textbf{1}$) if both $\varphi$ and $\psi$ are totally even (resp. totally odd). 
Here we identify $W_G=K_{\infty}/K_{\infty,+}$ with $\{\pm1\}^{J_F}$ by the determinant map. 
Put $\chi=\varphi\psi$. 

\begin{rem}\label{parity of Eis}
We note that 
\[
[\omega_{\textbf{E}}]^{\epsilon_{{}_{\textbf{E}}}}=[\omega_{\textbf{E}}]\neq 0
\ \text{in}\ H^n(Y(\ideal{n}),\C),
\]
where $[\omega_{\textbf{E}}]^{\epsilon_{{}_{\textbf{E}}}}$ stands for the projection of $[\omega_{\textbf{E}}]$ to the $\epsilon_{{}_{\textbf{E}}}$-part. 
Indeed, for a narrow ray class character $\theta$ of $F$ such that $(\ideal{m}_{\theta},\ideal{n})=1$ and $\theta=\epsilon_{{}_{\textbf{E}}}$ on $W_G\simeq \A_{F,\infty}^{\times}/\A_{F,\infty,+}^{\times}$, under the same notation as \S \ref{subsection:modular symbol}, we have 
\begin{align}\label{denominator'}
\sum_{b\in S}\eta_1(\bar{b})^{-1}&\mathrm{ev}_{b,1,\C}([\omega_{\textbf{E}}]_{\mathrm{rel}}^{\epsilon_{\textbf{E}}})
= \tau(\eta^{-1})
\frac{\sqrt{-1}^n}{(2\pi)^n}
D(1,\textbf{E}, \eta) \\
&\nonumber\ \ \ \ \ \ \ \ =
\frac{(-1)^n}{2^n \Delta_F^{1/2}} 
\cdot 
\frac{\tau(\varphi\psi)\varphi\psi(\ideal{m}_{\theta})\theta(\ideal{m}_{\psi})}{\tau(\psi)\psi(\ideal{m}_{\theta})\theta(\ideal{m}_{\varphi\psi})}
\cdot
L(0,\theta^{-1}\psi) L(0,\theta\varphi^{-1})\neq 0,
\end{align}
where $\eta$ denotes $\theta\varphi^{-1}\psi^{-1}$. 
Here the first equality follows from $\ideal{n}|\ideal{m}_{\eta}$, Proposition \ref{Modular symbol}, and Proposition \ref{Modular symbol anti-hol}, 
the second equality follows from Proposition \ref{Hilbert Eisenstein} (1), the functional equation for Hecke $L$\nobreakdash-functions (see, for example, \cite[Theorem 3.3.1]{Mi}), and the fact that $\eta\varphi=\theta\psi^{-1}$ is totally odd and \cite[(3.3.11)]{Mi}, and $L(0,\theta^{-1}\psi) L(0,\theta\varphi^{-1})\neq 0$ follows from the fact that both $\theta\psi^{-1}$ and $\theta\varphi^{-1}$ are totally odd and the functional equation for Hecke $L$\nobreakdash-functions (see, for example, \cite[Lemma 1.1]{Da--Da--Po}). 
Hence $[\omega_{\textbf{E}}]^{\epsilon_{\textbf{E}}}\neq 0$. 
Now our assertion follows from Proposition \ref{Eis Hodge number} and the $q$-expansion principle over $\C$.
\end{rem}

%
\subsection{Rationality of cohomology classes}\label{subsection:rationality of Eis}
%

In this subsection, we prove the rationality of the cohomology classes of $\mathbf{E}$ in $H^n(Y(\ideal{n}),\C)$ and $H^n(Y(\ideal{n})^{\mathrm{BS}},D_{C_{\infty}}(\ideal{n});\C)$. 

\begin{prop}\label{prop:rational'}
The cohomology class 
$[\omega_{\mathbf{E}}]$ is rational, that is, 
$[\omega_{\mathbf{E}}] \in H^n(Y(\ideal{n}),K)$. 
\end{prop}
\begin{proof}
Let $\mathfrak{p}_{\textbf{E}}$ denote the maximal ideal of $\mathbb{H}_2(\mathfrak{n},\integer{})\otimes K$ generated by $T(\ideal{q})-C(\ideal{q},\textbf{E}),S(\ideal{q})-\chi^{-1}(\ideal{q})$ for all non-zero prime ideals $\ideal{q}$ of $\ideal{o}_F$ prime to $\ideal{n}$ and $U(\ideal{q})-C(\ideal{q},\textbf{E})$ for all non\nobreakdash-zero prime ideals $\ideal{q}$ of $\ideal{o}_F$ dividing $\ideal{n}$. 
By Proposition \ref{const} (1), $\res([\omega_{\textbf{E}}])$ is rational. 
By Remark \ref{parity of Eis}, 
$[\omega_{\textbf{E}}]=[\omega_{\textbf{E}}]^{\epsilon_{{}_{\textbf{E}}}}$.
Hence there is $c \in H^n(Y(\ideal{n}),K)_{\mathfrak{p}_{\textbf{E}}}[\epsilon_{{}_{\textbf{E}}}]$ mapping to $\res([\omega_{\textbf{E}}])$. 
We have $[\omega_{\textbf{E}}]-c\in H_{\pa}^n(Y(\ideal{n}),\C)_{\mathfrak{p}_{\textbf{E}}}[\epsilon_{{}_{\textbf{E}}}]$. 
The isomorphism (\ref{+,+ decomp}) and the $q$\nobreakdash-expansion principle over $\C$ imply $H_{\pa}^n(Y(\ideal{n}),\C)_{\mathfrak{p}_{\textbf{E}}}[\epsilon_{{}_{\textbf{E}}}]=0$. 
Hence $[\omega_{\textbf{E}}]=c \in H^n(Y(\ideal{n}),K)$. 
\end{proof}

In order to prove the rationality of the relative cohomology class, we need to show a vanishing result on the cohomology of $D_{C_{\infty}}(\ideal{n})$. 

We abbreviate $\Gamma_1(\ideal{d}_F[t_1],\ideal{n})$ to $\Gamma$ and $\Gamma_0(\ideal{d}_F[t_1],\ideal{n})$ to $\Gamma_0$.
Let $\ideal{q}$ be a non-zero prime ideal of $\mathfrak{o}_F$ dividing $\ideal{n}$.
Since $h_F^+=1$, we can choose and fix a totally positive generator $g_{\ideal{q}}$ (resp. $e$) of $\ideal{q}$ (resp. $\ideal{d}_F[t_1]$). 
For the proof, we need the following:
\begin{align}
\label{U(q) explicit decomp}
\Gamma\begin{pmatrix}1&0\\0&g_{\ideal{q}}\end{pmatrix}\Gamma
&=\coprod_{b\in \ideal{d}_F^{-1}[t_1]^{-1}/\ideal{d}_F^{-1}[t_1]^{-1}\ideal{q}}\Gamma\begin{pmatrix}1&b\\0&g_{\ideal{q}}\end{pmatrix};\\
\label{invariant under Gamma_0}
\gamma\Gamma\begin{pmatrix}1&0\\0&g_{\ideal{q}}\end{pmatrix}\Gamma \gamma^{-1} 
&=\Gamma \gamma \begin{pmatrix}1&0\\0&g_{\ideal{q}}\end{pmatrix}
\gamma^{-1} \Gamma
=\Gamma\begin{pmatrix}1&0\\0&g_{\ideal{q}}\end{pmatrix}\Gamma 
\text{ for $\gamma\in\Gamma_0$,}
\end{align}
where $b$ runs over a complete set of representatives of $\ideal{d}_F^{-1}[t_1]^{-1}/\ideal{d}_F^{-1}[t_1]^{-1}\ideal{q}$. 
In order to show (\ref{U(q) explicit decomp}) and (\ref{invariant under Gamma_0}), we may assume $\ideal{d}_F[t_1]=\ideal{o}_F$ because
$\begin{pmatrix}1&0\\ 0&e^{-1}\end{pmatrix}\Gamma_1(\ideal{d}_F[t_1],\ideal{n})\begin{pmatrix}1&0\\0&e\end{pmatrix}=\Gamma_1(\ideal{o}_F,\ideal{n})$ and $\begin{pmatrix}1&0\\ 0&e^{-1}\end{pmatrix}\Gamma_0(\ideal{d}_F[t_1],\ideal{n})\begin{pmatrix}1&0\\0&e\end{pmatrix}=\Gamma_0(\ideal{o}_F,\ideal{n})$.

First we show (\ref{U(q) explicit decomp}).
By taking the inverse and multiplying $g_{\ideal{q}}$, it suffices to show that \[
\Gamma\begin{pmatrix}g_{\ideal{q}}&0\\0&1\end{pmatrix}\Gamma
=\coprod_{b\in \ideal{o}_F/\ideal{q}}\begin{pmatrix}g_{\ideal{q}}&b\\0&1\end{pmatrix}\Gamma.
\]
For any $\beta=\begin{pmatrix}a&b\\ c&d\end{pmatrix}\in \Gamma\begin{pmatrix}g_{\ideal{q}}&0\\0&1\end{pmatrix}\Gamma$, 
we have $a,b,c,d\in \ideal{o}_F$, 
$c\equiv 0\ (\mod \ \ideal{n})$, $d\equiv 1\ (\bmod \ \ideal{n})$, and 
$\det(\beta)=g_{\ideal{q}}u$ for some $u\in\mathfrak{o}_{F,+}^{\times}$. 
Since $\ideal{q}$ divides $\ideal{n}$, we have $(c,d)=1$.
Hence there is 
$\gamma_1=\begin{pmatrix}d&\ast\\ -c&\ast\end{pmatrix}\in \Gamma$ with $\det(\gamma_1)=1$ such that 
\begin{align*}
\beta\gamma_1\begin{pmatrix}u^{-1}&0\\ 0&1\end{pmatrix}
&=\begin{pmatrix}\det(\beta)&\ast\\ 0&1\end{pmatrix}\begin{pmatrix}u^{-1}&0\\ 0&1\end{pmatrix}=\begin{pmatrix}g_{\ideal{q}}&b'\\ 0&1\end{pmatrix}.
\end{align*}

We show (\ref{invariant under Gamma_0}).
The first equality of (\ref{invariant under Gamma_0}) follows from the fact that $\Gamma$ is a normal subgroup of $\Gamma_0$ and the second equality of (\ref{invariant under Gamma_0}) follows from the same argument as in the proof of (\ref{U(q) explicit decomp}).
Indeed, for 
$\beta=\begin{pmatrix}a&b\\ c&d\end{pmatrix}=\gamma\begin{pmatrix}g_{\ideal{q}}&0\\0&1\end{pmatrix}\gamma^{-1}$, 
we have $a,b,c,d\in \ideal{o}_F$, $c\equiv 0\ (\bmod \ \ideal{n})$, $d\equiv 1\ (\bmod \ \ideal{n})$, 
$\det(\beta)=g_{\ideal{q}}$, 
and $\ideal{q}$ divides $\ideal{n}$.

We put $\alpha_b=\begin{pmatrix}1&b\\0&g_{\ideal{q}}\end{pmatrix}$. 
Note that, for $s\in F$, the condition $s\in C_{\infty}$ is equivalent to the condition $es=a/c$ for some $a,c\in\ideal{o}_F$ such that $(a,c)=1$ and $c\equiv 0\ (\bmod \ \ideal{n})$.
Hence, if $s\in C_{\infty}$, then $\alpha_b(s)\in C_{\infty}$.
Therefore $U(\ideal{q})$ preserves the component $D_{C_{\infty}}(\ideal{n})$. 
Let $\mathbb{H}_2(\ideal{n},\integer{})'$ be the commutative $\integer{}$\nobreakdash-subalgebra of $\End_{\integer{}}({H}^{n-1}(D_{C_{\infty}}(\ideal{n}),\integer{})) \oplus \End_{\integer{}}({H}^n (Y(\ideal{n})^{\text{BS}},D_{C_{\infty}}(\ideal{n});\integer{}))\oplus \End_{\integer{}}({H}^n (Y(\ideal{n}),\integer{}))\oplus \End_{\integer{}}({H}^n(D_{C_{\infty}}(\ideal{n}),\integer{}))$ generated by $U(\ideal{q})$ for all non-zero prime ideals $\ideal{q}$ of $\mathfrak{o}_F$ dividing $\ideal{n}$, and 
$\ideal{m}_{\textbf{E}}'$ the maximal ideal of $\mathbb{H}_2(\ideal{n},\integer{})'$ generated by $\varpi$ and $U(\ideal{q})-C(\ideal{q},\textbf{E})$ for all non-zero prime ideals $\ideal{q}$ of $\mathfrak{o}_F$ dividing $\ideal{n}$.

\begin{prop}\label{H^{n-1} vanishing}
Assume that $C(\ideal{q},\mathbf{E})\not\equiv N(\ideal{q})\ (\bmod \varpi)$ for some prime ideal $\ideal{q}$ dividing $\ideal{n}$.
Then 
$H^{n-1}(D_{C_{\infty}}(\ideal{n}),\C)_{\ideal{m}_{\mathbf{E}}'}=0$.
\end{prop}
\begin{proof}
By \textbf{Step1} in the proof of Proposition \ref{Eis Hodge number}, for each $t\in C_{\infty}$ such that $t=\gamma (\infty)$ with $\gamma\in\Gamma_0$, a basis of $H^{n-1}(D_{t},\C)$ is given by $\omega_t:=(\gamma^{-1})^{\ast}(\omega_{\infty}^{n-1})$. 
We claim that 
\begin{align}\label{U(q)-eigenvalue of omega_t}
\omega_t|U(\ideal{q})=N(\ideal{q})\omega_t
\end{align}
for any $t\in C_{\infty}$ and any prime ideal $\ideal{q}$ of $\mathfrak{o}_F$ dividing $\ideal{n}$.

For the moment, we admit the claim (\ref{U(q)-eigenvalue of omega_t}). 
We have 
\[
H^{n-1}(D_{C_{\infty}}(\ideal{n}),\C)_{\ideal{m}_{\mathbf{E}}'}
 \simeq 
\prod_{\ideal{p}\cap \mathbb{H}_2(\ideal{n},\integer{})' \subset \ideal{m}_{\mathbf{E}}'} H^{n-1}(D_{C_{\infty}}(\ideal{n}),\C) \otimes_{\mathbb{H}_2(\ideal{n},\integer{})'} K(\ideal{p}), 
\]
where $\ideal{p}$ runs over the set of maximal ideals of $\mathbb{H}_2(\ideal{n},\integer{})'\otimes K$ such that $\ideal{p}\cap \mathbb{H}_2(\ideal{n},\integer{})' \subset \ideal{m}_{\mathbf{E}}'$, and 
$K(\ideal{p})$ denotes the residue field of $\ideal{p}$.
Let $\varphi_{\ideal{p}}$ denote the mod $\ideal{p}$ map $\mathbb{H}_2(\ideal{n},\integer{})'\otimes K \twoheadrightarrow K(\ideal{p})$.
The condition $\ker(\varphi_{\ideal{p}})\cap \mathbb{H}_2(\ideal{n},\integer{})'\subset \ideal{m}_{\mathbf{E}}'$ is equivalent to the condition $\varphi_{\ideal{p}}(U(\ideal{q}))\equiv C(\ideal{q},\mathbf{E})\ (\bmod \ \ideal{m}_{K(\ideal{p})})$ for all non-zero prime ideals $\ideal{q}$ of $\ideal{o}_F$ dividing $\ideal{n}$. 
By (\ref{U(q)-eigenvalue of omega_t}), 
$\varphi_{\ideal{p}}(U(\ideal{q}))=N(\ideal{q})$. 
Now our assumption implies $H^{n-1}(D_{C_{\infty}}(\ideal{n}),\C)_{\ideal{m}_{\mathbf{E}}'}=0$ as desired.

Thus it remains to prove the claim (\ref{U(q)-eigenvalue of omega_t}).
In order to do it, under the canonical isomorphism 
\[
H^{n-1}(\partial(Y(\ideal{n})^{\mathrm{BS}}),\C)\simeq \{(c_s)_s\in \bigoplus_{s\in \mathbb{P}^1(F)}H^{n-1}(D_s,\C)\mid \gamma^{\ast}(\omega_{\gamma(s)})=\omega_s \ \text{for } \gamma\in\Gamma,s\in \mathbb{P}^1(F)\},
\]
we explicitly describe the action of $U(\ideal{q})$ on the right-hand side.

We first treat the case $t=\infty$. 
By using the decomposition (\ref{U(q) explicit decomp}) and the definition of the Hecke operator acting on the boundary cohomology \cite[(3.1c)]{Hida93}, we have 
\begin{align}\label{act of Hecke on boundary}
\left(\omega_{\infty}|[\Gamma\begin{pmatrix}1&0\\0&g_{\ideal{q}}\end{pmatrix}\Gamma
]\right)_{s}=\sum_{\alpha_b(s)\sim_{\Gamma} \infty}  (\alpha_b)^{\ast} (\omega_{\alpha_b (s)}),
\end{align}
where $b$ runs over a complete set of representative of $\ideal{d}_F^{-1}[t_1]^{-1}/\ideal{d}_F^{-1}[t_1]^{-1}\ideal{q}$ such that $\alpha_b(s)$ is $\Gamma$-equivalent to $\infty$. 
By (\ref{U(q) explicit decomp}) and (\ref{invariant under Gamma_0}), for $s\in C_{\infty}$, $\alpha_b(s)$ is $\Gamma$-equivalent to $s$. 
Indeed, for $s=\gamma_0 (\infty)$ with $\gamma_0\in\Gamma_0$, we have $\Gamma\begin{pmatrix}1&0\\0&g_{\ideal{q}}\end{pmatrix}\Gamma=\coprod_{b}\Gamma \gamma_0^{-1}\alpha_b\gamma_0=\coprod_{b}\gamma_0^{-1}\Gamma\alpha_b\gamma_0$. 
Then $\gamma_0^{-1}\gamma\alpha_b\gamma_0=\alpha_{b'}$ for some $\gamma\in \Gamma$ and $b'\in \ideal{d}_F^{-1}[t_1]^{-1}/\ideal{d}_F^{-1}[t_1]^{-1}\ideal{q}$ and hence $\alpha_{b}(s)=\gamma^{-1}(s)$. 
Therefore, if $s$ is not $\Gamma$-equivalent to $\infty$, then (\ref{act of Hecke on boundary}) is $0$ and 
\begin{align*}
\left(\omega_{\infty}|[\Gamma\begin{pmatrix}1&0\\0&g_{\ideal{q}}\end{pmatrix}\Gamma
]\right)_{\infty}
&=\sum_{b\in \ideal{d}_F^{-1}[t_1]^{-1}/\ideal{d}_F^{-1}[t_1]^{-1}\ideal{q}}(\alpha_b)^{\ast}(\omega_{\infty})\\
&=N(\ideal{q}) \omega_{\infty}. 
\end{align*}
Here the last equality follows from that $\omega_{\infty}$ is invariant under the action of the standard Borel subgroup $B_{\infty}$.

We treat the general case $t\in C_{\infty}$. 
Let $\gamma\in\Gamma_0$ such that $t=\gamma(\infty)$.
The canonical map $\gamma :D_{C_{\infty}}(\ideal{n})\to D_{C_{\infty}}(\ideal{n})$ induces 
$\gamma^{\ast}: H^{n-1}(D_{C_{\infty}}(\ideal{n}),\C) \to H^{n-1}(D_{C_{\infty}}(\ideal{n}),\C)$. 
We have $\gamma^{\ast}(\omega_t)\in H^{n-1}(D_{\infty},\C)$.
Hence $(\gamma^{-1})^{\ast}(\gamma^{\ast}(\omega_t)|U(\ideal{q}))=N(\ideal{q})\omega_t$. 
Now the assertion follows from (\ref{invariant under Gamma_0}).
\end{proof}

\begin{prop}\label{prop:rational}
Under the same notation and assumptions as Proposition \ref{H^{n-1} vanishing}, $[\omega_{\mathbf{E}}]_{\mathrm{rel}}$ is rational, that is, 
$[\omega_{\mathbf{E}}]_{\mathrm{rel}}\in H^n(Y(\ideal{n})^{\mathrm{BS}},D_{C_{\infty}}(\ideal{n});K)$.
\end{prop}
\begin{proof}
By Proposition \ref{prop:rational'}, 
there is $c \in H^n(Y(\ideal{n})^{\text{BS}},D_{C_{\infty}}(\ideal{n});K)_{\ideal{m}_{\textbf{E}}'}$ mapping to $[\omega_{\textbf{E}}] \in H^n(Y(\ideal{n}),K)_{\ideal{m}_{\textbf{E}}'}$. 
The difference $c-[\omega_{\textbf{E}}]_{\mathrm{rel}}$ is in the image of $H^{n-1}(D_{C_{\infty}}(\ideal{n}),\C)_{\ideal{m}_{\textbf{E}}'}$ and Proposition \ref{H^{n-1} vanishing} implies $[\omega_{\textbf{E}}]_{\mathrm{rel}}=c$. 
\end{proof}

%
\subsection{Denominator ideal}\label{subsection:Denominator ideal}
%

In this subsection, we recall the definition of the denominator ideal in the sense of T. Berger (\cite[$\S$4.1]{Be}). 

Let $\widetilde{H}^n(Y(\ideal{n}),\integer{})$ denote 
the image of $H^n(Y(\ideal{n}),\integer{})\to H^n(Y(\ideal{n}),K)$. 
For $c \in H^n(Y(\ideal{n}),K)$, 
let $\delta(c)$ denote the denominator ideal of $c$, that is,
\[
\delta(c)=\left\{a\in \integer{}\ \big|\  ac \in \widetilde{H}^n(Y(\ideal{n}),\integer{})\right\}.
\]

%
\subsection{Congruence modules and integrality of cohomology classes}\label{congruence module}
%

In this subsection, we determine the structure of the congruence module associated to $\mathbf{E}$ by using the denominator ideal of $[\omega_{\mathbf{E}}]$. 
As an application, we prove the integrality of $[\omega_{\mathbf{E}}]$. 
The proof is based on the method of T. Berger \cite[\S 4]{Be} and M. Emerton \cite[Proposition 4, Theorem 5]{Eme}. 

We abbreviate $\Gamma_{1}(\ideal{d}_F [t_1], \mathfrak{n})$ to $\Gamma$. 
Let $\ideal{p}_{\mathbf{E}}$ be the prime ideal of $\mathbb{H}_2(\mathfrak{n},\integer{})$ generated by $T(\ideal{q})-C(\ideal{q},\textbf{E}),S(\ideal{q})-\chi^{-1}(\ideal{q})$ for all non-zero prime ideals $\ideal{q}$ of $\ideal{o}_F$ prime to $\ideal{n}$ and $U(\ideal{q})-C(\ideal{q},\textbf{E})$ for all non-zero prime ideals $\ideal{q}$ of $\ideal{o}_F$ dividing $\ideal{n}$. 
Let $\mathscr{P}_{\mathbf{E}}$ denote the image of $\ideal{p}_{\mathbf{E}}$ under the canonical surjection $\mathbb{H}_2(\mathfrak{n},\integer{}) \twoheadrightarrow \mathscr{H}_2(\mathfrak{n},\integer{})$. 
The module $\mathscr{H}_2(\mathfrak{n},\integer{})/\mathscr{P}_{\mathbf{E}}$ is the congruence module associated to $\textbf{E}$.

In order to determine the structure of $\mathscr{H}_2(\mathfrak{n},\integer{})/\mathscr{P}_{\mathbf{E}}$, we use an element $A_{1,s}\in \mathbb{H}_2(\mathfrak{n},\integer{})$ for a cusp $s\in C(\Gamma)$ defined as follows.
The space $M_2(\ideal{n},\integer{})$ of modular forms introduced in \S \ref{Duality theorem} can be identified with the space $M_2(M,\integer{})$ of geometric modular forms defined in \S \ref{subsection:Geometric HMF} (see, for example, \cite[p.329--333]{Hida88} and Definition \ref{definition:GHMF}). 
Hence, if $\mathbf{f}=f_1\in M_2(\ideal{n},\integer{})$, then the constant term of $f_1$ at $s$ belongs to $\integer{}$ by the $q$-expansion principle. 
Now, by using the duality theorem (Theorem \ref{Duality}), 
we can define $A_{1,s}\in \mathbb{H}_2(\mathfrak{n},\integer{})$ as the element corresponding to an $\integer{}$-linear map $\textbf{f} \mapsto a_s(0,f_1)$ from $M_2(\ideal{n},\integer{})$ to $\integer{}$, where $a_s(0,f_1)$ denotes the constant term of $f_1$ at $s$. 

Let $s_0\in C(\Gamma)$ such that $v_p(a_{s_0}(0,E_{1}))\le v_p(a_{s}(0,E_{1}))$ for every $s\in C(\Gamma)$, where $v_p$ denotes the $p$-adic valuation. 
We put 
\[
C=a_{s_0}(0,E_{1}). 
\]
Let $\mathcal{H}_2(\ideal{n},\integer{})$ be the commutative $\integer{}$\nobreakdash-subalgebra of 
$\End_{\integer{}}(H_c^n(Y(\ideal{n}),\integer{})) \oplus \End_{\integer{}}(H^n (Y(\ideal{n}), \integer{}))$ $ \oplus \End_{\integer{}}(H^n (\partial(Y(\ideal{n})^{\mathrm{BS}}),\integer{})) \oplus \End_{\integer{}}(H_c^{n+1}(Y(\ideal{n}),\integer{}))$ generated by $T(\ideal{q}),S(\ideal{q})$ for all non-zero prime ideals $\ideal{q}$ of $\ideal{o}_F$ prime to $\ideal{n}$ and $U(\ideal{q})$ for all non-zero prime ideals $\ideal{q}$ of $\ideal{o}_F$ dividing $\ideal{n}$, and 
$\ideal{m}$ the maximal ideal of $\mathcal{H}_2(\ideal{n},\integer{})$ generated by $\varpi$ and $T(\ideal{q})-C(\ideal{q},\textbf{E}),S(\ideal{q})-\chi^{-1}(\ideal{q})$ for all non-zero prime ideals $\ideal{q}$ of $\ideal{o}_F$ prime to $\ideal{n}$ and $U(\ideal{q})-C(\ideal{q},\textbf{E})$ for all non-zero prime ideals $\ideal{q}$ of $\ideal{o}_F$ dividing $\ideal{n}$. 

\begin{thm}\label{cong mod}
Let $p$ be a prime number $>3$ such that $p$ is prime to $\ideal{n}$ and $\Delta_F$. 
We assume the following two conditions $\mathrm{(a)}$ and $\mathrm{(b)}$$:$ 

$\mathrm{(a)}$
$H^{n}(\partial \left(Y(\ideal{n})^{\mathrm{BS}}\right),\integer{})_{\ideal{m}}$, $H_c^{n+1}(Y(\ideal{n}),\integer{})_{\ideal{m}}$, and $H^n(D_{C_{\infty}}(\ideal{n}),\integer{})_{\ideal{m}_{\mathbf{E}}'}$ are torsion-free, 
where 

\ \ \ \ $\ideal{m}_{\mathbf{E}}'$ is the maximal ideal of $\mathbb{H}_2(\ideal{n},\integer{})'$ defined before Proposition \ref{H^{n-1} vanishing}$;$

$\mathrm{(b)}$
$C(\ideal{q},\mathbf{E})\not\equiv N(\ideal{q})\ (\bmod \varpi)$ for some prime ideal $\ideal{q}$ dividing $\ideal{n}$. 

Then there are isomorphisms of $\integer{}$-modules
\[
\mathbb{H}_{2}(\ideal{n},\integer{})[\epsilon_{{}_{\mathbf{E}}}]/(\ideal{p}_{\mathbf{E}}+\sum_{s\in C(\Gamma)}\integer{}A_{1,s})[\epsilon_{{}_{\mathbf{E}}}]\simeq \mathscr{H}_{2}(\ideal{n},\integer{})/\mathscr{P}_{\mathbf{E}}\simeq \integer{}/C.
\]
Here the notion $[\epsilon_{{}_{\mathbf{E}}}]$ is defined at the end of \S \ref{subsection:Eichler--Shimura--Harder}.
\end{thm}
\begin{proof}
We prove the assertion by constructing the following surjective $\integer{}$-linear morphisms 
(\ref{Step0}), (\ref{Step1}), (\ref{Step2}), and (\ref{Step3}) whose composition is the identity:
\begin{align*}
\integer{}/C &
\xrightarrow{(\ref{Step0})} \mathbb{H}_2(\ideal{n},\integer{})[\epsilon_{{}_{\textbf{E}}}]/(\ideal{p}_{\mathbf{E}}+\sum_{s\in C(\Gamma)} \integer{}A_{1,s})[\epsilon_{{}_{\mathbf{E}}}]
\xrightarrow{(\ref{Step1})}\mathscr{H}_{2}(\ideal{n},\integer{})/\mathscr{P}_{\mathbf{E}}
\xrightarrow{(\ref{Step2})}\integer{}/\delta_{\textbf{G}}
\xrightarrow{(\ref{Step3})} \integer{}/C.
\end{align*}

First we construct the surjective morphisms (\ref{Step0}) and (\ref{Step1}).
By the definition of $A_{1,s}$, we have $A_{1,s}=a_s(0,E_1)$ in 
$\mathbb{H}_2(\ideal{n},\integer{})/\ideal{p}_{\mathbf{E}}\simeq \integer{}$.
Hence we obtain a surjective $\integer{}$-linear map by $1 \mapsto 1$:
\begin{align}\label{Step0}
\integer{}/C \twoheadrightarrow \mathbb{H}_2(\ideal{n},\integer{})[\epsilon_{{}_{\textbf{E}}}]/(\ideal{p}_{\mathbf{E}}+\sum_{s\in C(\Gamma)} \integer{}A_{1,s})[\epsilon_{{}_{\textbf{E}}}]. 
\end{align}
The canonical surjection $\mathbb{H}_2(\mathfrak{n},\integer{}) \twoheadrightarrow \mathscr{H}_2(\mathfrak{n},\integer{})$ induces 
\begin{align}\label{Step1}
\mathbb{H}_{2}(\ideal{n},\integer{})[\epsilon_{{}_{\textbf{E}}}]/(\ideal{p}_{\mathbf{E}}+\sum_{s\in C(\Gamma)}\integer{}A_{1,s})[\epsilon_{{}_{\textbf{E}}}]\twoheadrightarrow\mathscr{H}_{2}(\ideal{n},\integer{})/\mathscr{P}_{\mathbf{E}}. 
\end{align}

Put $\textbf{G}=\textbf{E}/C \in M_2(\ideal{n},\integer{})$. 
Consider the cohomology class $[\omega_{\textbf{G}}]^{\epsilon_{{}_{\textbf{E}}}}\in H^n(Y(\ideal{n}),\C)[\epsilon_{{}_{\textbf{E}}}]_{\ideal{m}}$. 
By Remark \ref{parity of Eis}, we have 
$[\omega_{\textbf{G}}]^{\epsilon_{{}_{\textbf{E}}}}=[\omega_{\textbf{G}}]\neq 0$.
By Proposition \ref{prop:rational'}, 
$[\omega_{\textbf{G}}]\in H^n(Y(\ideal{n}),K)$. 
Let $\delta_{\textbf{G}}$ denote the denominator ideal $\delta([\omega_{\textbf{G}}]^{\epsilon_{{}_{\textbf{E}}}})$ of $[\omega_{\textbf{G}}]^{\epsilon_{{}_{\textbf{E}}}}$ defined in \S \ref{subsection:Denominator ideal}. 
Next we construct the surjective morphism 
\begin{align}\label{Step2}
\mathscr{H}_{2}(\ideal{n},\integer{})/\mathscr{P}_{\mathbf{E}}\twoheadrightarrow\integer{}/\delta_{\textbf{G}}.
\end{align} 
By Proposition \ref{const} (1), 
$\res([\omega_{\textbf{G}}]^{\epsilon_{{}_{\textbf{E}}}})
\in \widetilde{H}^n(\partial (Y(\ideal{n})^{\mathrm{BS}}),\integer{})[\epsilon_{{}_{\textbf{E}}}]_{\ideal{m}}$. 
The image of $\res([\omega_{\textbf{G}}]^{\epsilon_{{}_{\textbf{E}}}})$ 
under the connecting homomorphism $H^{n}(\partial(Y(\ideal{n})^{\text{BS}}),K)[\epsilon_{{}_{\textbf{E}}}]_{\ideal{m}}\to H_c^{n+1}(Y(\ideal{n}),K)[\epsilon_{{}_{\textbf{E}}}]_{\ideal{m}}$ is $0$. 
Hence, by the assumption (a) on $H_c^{n+1}(Y(\ideal{n}),\integer{})_{\ideal{m}}$, 
there is $c\in\widetilde{H}^n(Y(\ideal{n}),\integer{})[\epsilon_{{}_{\textbf{E}}}]_{\ideal{m}}$ such that 
$\res(c)=\res([\omega_{\textbf{G}}]^{\epsilon_{{}_{\textbf{E}}}})$.
We have 
$c-[\omega_{\textbf{G}}]^{\epsilon_{{}_{\textbf{E}}}}
\in H_{\pa}^n(Y(\ideal{n}),K)[\epsilon_{{}_{\textbf{E}}}]_{\ideal{m}}$. 
Fix a generator $d$ of $\delta_{\textbf{G}}$. 
Put $e_0=d(c-[\omega_{\textbf{G}}]^{\epsilon_{{}_{\textbf{E}}}})\in \widetilde{H}^n(Y(\ideal{n}),\integer{})[\epsilon_{{}_{\textbf{E}}}]_{\ideal{m}}$. 
The assumption (a) on $H^{n}(\partial\left(Y(\ideal{n})^{\mathrm{BS}}\right),\integer{})_{\ideal{m}}$ implies $e_0 \in \widetilde{H}_{\pa}^n(Y(\ideal{n}),\integer{})[\epsilon_{{}_{\textbf{E}}}]_{\ideal{m}}$.
We may assume $e_0 \neq 0$. 
Indeed, if $e_0=0$, then $c=[\omega_{\textbf{G}}]^{\epsilon_{{}_{\textbf{E}}}}$ and hence $\delta_{\textbf{G}}=\integer{}$. 
Let $e_0,\cdots,e_v$ be an $\integer{}$-basis of $\widetilde{H}_{\pa}^n(Y(\ideal{n}),\integer{})[\epsilon_{{}_{\textbf{E}}}]$. 
For $t \in \mathscr{H}_2(\ideal{n},\integer{})$, we write 
\[
t(e_0)=\sum_{0\le i \le v} \lambda_i(t)e_i
\]
with $\lambda_i(t) \in \integer{}$. 
Thus the $\integer{}$-linear surjective morphism defined by 
\[
\mathscr{H}_2(\ideal{n},\integer{})\twoheadrightarrow\integer{}/\delta_{\textbf{G}} ;
t\mapsto \lambda_0(t)
\]
induces the required morphism (\ref{Step2}). 

Finally we construct the surjective morphism
\begin{align}\label{Step3}
\integer{}/\delta_{\textbf{G}}\twoheadrightarrow\integer{}/C.
\end{align}
In order to do it, it suffices to show that $\delta_{\mathbf{G}}\subset (C)$.
We fix a generator $d$ of $\delta_{\textbf{G}}$. 
Then we have 
$d[\omega_{\textbf{G}}]\in \widetilde{H}^n(Y(\ideal{n}),\integer{})$.
Moreover, under the assumption (b), Proposition \ref{prop:rational} implies $d[\omega_{\textbf{G}}]_{\mathrm{rel}} \in H^n(Y(\ideal{n})^{\text{BS}},D_{C_{\infty}}(\ideal{n});K)$. 
We claim that $d[\omega_{\textbf{G}}]_{\mathrm{rel}}$ is integral, that is, 
\begin{align}\label{key claim}
d[\omega_{\textbf{G}}]_{\mathrm{rel}} \in \widetilde{H}^n(Y(\ideal{n})^{\text{BS}},D_{C_{\infty}}(\ideal{n});\integer{}).
\end{align}

For the moment, we admit the claim (\ref{key claim}). 
Let $\eta_p$ be a non-trivial primitive narrow ray class character of $F$ 
corresponding to a finite order character of $\Gal(F(\zeta_{p^{\infty}})/F)$ such that $\eta_p=\epsilon_{{}_{\textbf{E}}}$ on $W_G\simeq \A_{F,\infty}^{\times}/\A_{F,\infty,+}^{\times}$. 
Put $\eta=\eta_p\varphi^{-1}\psi^{-1}$. 
Note that $\ideal{n}|\ideal{m}_{\eta}$. 
The condition $\ideal{n}|\ideal{m}_{\eta}$ implies the following:
\begin{align}\label{denominator}
\integer{}(\eta) \ni &\sum_{b\in S}\eta_1(\bar{b})^{-1}\mathrm{ev}_{b,1,\integer{}}(d[\omega_{\textbf{G}}]_{\mathrm{rel}}^{\epsilon_{\textbf{E}}})\\
&\nonumber=
\frac{d}{C}
\cdot
\frac{(-1)^n}{2^n \Delta_F^{1/2}} 
\cdot 
\frac{\tau(\varphi\psi)\varphi\psi(\ideal{m}_{\eta_p})\eta_p(\ideal{m}_{\psi})}{\tau(\psi)\psi(\ideal{m}_{\eta_p})\eta_p(\ideal{m}_{\varphi\psi})}
\cdot
L(0,\eta_p^{-1}\psi)L(0,\eta_p\varphi^{-1}). 
\end{align}
Here the equality follows from $C[\omega_{\mathbf{G}}]=[\omega_{\mathbf{E}}]$ and the same argument as in the proof of (\ref{denominator'}), and the integrality of the value follows from (\ref{key claim}), Proposition \ref{Modular symbol}, and Proposition \ref{Modular symbol anti-hol}. 
Note that the second and third terms in the second line of (\ref{denominator}) are prime to $p$. 
Moreover, by the condition on the $\mu$-invariants in (Eis condition) with the help of the Iwasawa main conjecture for totally real number fields proved by A. Wiles \cite{Wil}, 
the $p$-adic valuation of $L(0,\eta_p^{-1}\psi)$ and $L(0,\eta_p\varphi^{-1})$ are smaller than that of $\varpi$ for all but finitely many narrow ray class character $\eta_p$ of $F$ such that $\eta_p=\epsilon_{{}_{\textbf{E}}}$ on $W_G$. 
Therefore we obtain $C\mid d$ as required. 

Thus it remains to prove the claim (\ref{key claim}). 
We have an exact sequence 
\[
H^{n-1}(D_{C_{\infty}}(\ideal{n}),\integer{})_{\ideal{m}_{\textbf{E}}'}\to H^n(Y(\ideal{n})^{\mathrm{BS}},D_{C_{\infty}}(\ideal{n});\integer{})_{\ideal{m}_{\textbf{E}}'}\to H^n(Y(\ideal{n}),\integer{})_{\ideal{m}_{\textbf{E}}'}\to H^n(D_{C_{\infty}}(\ideal{n}),\integer{})_{\ideal{m}_{\textbf{E}}'}.
\]
By Proposition \ref{H^{n-1} vanishing}, $H^{n-1}(D_{C_{\infty}}(\ideal{n}),\integer{})_{\ideal{m}_{\textbf{E}}'}$ is torsion. 
By the assumption (a), $H^n(D_{C_{\infty}}(\ideal{n}),\integer{})_{\ideal{m}_{\textbf{E}}'}$ is torsion-free. 
Therefore we obtain an exact sequence
\[
0\to \widetilde{H}^n(Y(\ideal{n})^{\mathrm{BS}},D_{C_{\infty}}(\ideal{n});\integer{})_{\ideal{m}_{\textbf{E}}'}\to \widetilde{H}^n(Y(\ideal{n}),\integer{})_{\ideal{m}_{\textbf{E}}'}\to H^n(D_{C_{\infty}}(\ideal{n}),\integer{})_{\ideal{m}_{\textbf{E}}'}.
\]
Now the claim (\ref{key claim}) follows from this exact sequence. 
\end{proof}
By the proof of Theorem \ref{cong mod}, we get $\delta_{\textbf{G}}=(C)$ and hence we obtain the following: 
\begin{cor}\label{thm:integral}
Under the same assumptions as Theorem \ref{cong mod}, we have 
\begin{align*}
[\omega_{\mathbf{E}}] \in &\widetilde{H}^n(Y(\ideal{n}),\integer{})\backslash\varpi\widetilde{H}^n(Y(\ideal{n}),\integer{}), \\
[\omega_{\mathbf{E}}]_{\mathrm{rel}} \in &\widetilde{H}^n(Y(\ideal{n})^{\mathrm{BS}},D_{C_{\infty}}(\ideal{n});\integer{})\backslash \varpi\widetilde{H}^n(Y(\ideal{n})^{\mathrm{BS}},D_{C_{\infty}}(\ideal{n});\integer{}).
\end{align*}
\end{cor}

%
\subsection{Real quadratic field case}\label{Example}
%
%
In this subsection, 
we give an example of a congruence between a Hilbert cusp form and a Hilbert Eisenstein series.

We use the same notation as in the proof of Theorem \ref{cong mod}. 
We abbreviate $\Gamma_{1}(\ideal{d}_F[t_1],\ideal{n})$ to $\Gamma$ and $\Gamma\cap \SL_2(F)$ to $\Gamma^1$. 
Hereafter, in this subsection, we assume that 
$F$ is a real quadratic field with $h_F^+=1$. 
First we show the following lemma.
\begin{lem}\label{criterion}
Assume the following four conditions $(1)$, $(2)$, $(3)$, and $(4)$$:$
\begin{enumerate}[$(1)$]

\item \label{H_c^3}$H_{c}^3(Y(\ideal{n}),\integer{})$ is torsion-free$;$

\item \label{H^2(partial)}$H^2(\partial\left(Y(\ideal{n})^{\mathrm{BS}}\right),\integer{})$ is torsion-free$;$

\item \label{Eis eigenvalue}$C(\ideal{q},\mathbf{E})\not\equiv N(\ideal{q})\ (\bmod \varpi)$ for some prime ideal $\ideal{q}$ dividing $\ideal{n};$

\item \label{order}the ideal $(C)\neq 0,\integer{}$.
\end{enumerate}
Then there exist a finite extension $K'$ of $K$ with the ring of integer $\integer{}\hookrightarrow \cal{O}'$ and a uniformizer $\varpi'$ such that $(\varpi')\cap \integer{}=(\varpi)$ and a Hecke eigenform $\mathbf{f}\in S_2(\ideal{n},\cal{O}')$ for all $T(\ideal{q})$ and $U(\ideal{q})$ with character $\chi$ such that
$\mathbf{f}\equiv \mathbf{E}\ (\bmod \varpi')$.
\end{lem}
\begin{proof}
As discussed in the proof of (\ref{Step2}) of Theorem \ref{cong mod}, 
the assumption (4) implies 
$e_0\neq 0\in \widetilde{H}_{\pa}^2(Y(\ideal{n}),\integer{})[\epsilon_{{}_{\textbf{E}}}]$.
Hence $e_0$ is cohomologous to $-[\omega_{\textbf{E}}]$ modulo $\varpi$ and 
the Hecke eigenvalues of $e_0$ are the same as those of $-[\omega_{\textbf{E}}]$ modulo $\varpi$ for all $t\in \mathscr{H}_2(\ideal{n},\integer{})$. 
Now the Deligne\nobreakdash-Serre lifting lemma (\cite[Lemma 6.11]{Del--Se}) in the case 
$R=\integer{}$, $M=\widetilde{H}_{\pa}^2(Y(\ideal{n}),\integer{})[\epsilon_{{}_{\textbf{E}}}]$, and $\mathbb{T}=\mathscr{H}_2(\ideal{n},\cal{O})$ 
says that there exist a finite extension $K'$ of $K$ with the ring of integer $\integer{}\hookrightarrow \cal{O}'$ and a uniformizer $\varpi'$ such that $(\varpi')\cap \integer{}=(\varpi)$ and a non\nobreakdash-zero eigenvector 
$e\in \widetilde{H}_{\pa}^2(Y(\ideal{n}),\integer{})[\epsilon_{{}_{\textbf{E}}}]\otimes \cal{O}'$ 
for all $t\in \mathscr{H}_2(\ideal{n},\cal{O})$ with eigenvalues $\lambda(t)$ such that 
$\lambda(V(\ideal{q}))\equiv C(\ideal{q},\textbf{E})\ (\bmod\ \varpi')$ for all non-zero prime ideals $\ideal{q}$ of $\ideal{o}_F$ prime to $\ideal{n}$ (resp. dividing $\ideal{n}$) and $V(\ideal{q})=T(\ideal{q})$ (resp. $U(\ideal{q})$). 
By the isomorphism (\ref{+,+ decomp}), 
we obtain a Hecke eigenform $\textbf{f}\in S_2(\ideal{n},\C)$ for all $T(\ideal{q})$ and $U(\ideal{q})$ such that 
$e=[\omega_{\textbf{f}}]$. 
By using the relation between Hecke eigenvalues and Fourier coefficients, 
we may assume that $\textbf{f}\in S_2(\ideal{n},\cal{O}')$ with character $\chi$. 
Therefore we obtain the congruence 
$\textbf{f}\equiv \textbf{E}\ (\bmod \varpi')$.
\end{proof}

In order to give an example of a congruence between a Hilbert cusp form and a Hilbert Eisenstein series, 
we prove (1) and (2) of Lemma \ref{criterion} in certain case (Proposition \ref{prop:ab torsion-free} and \ref{prop:bou torsion-free})
and give a Hilbert Eisenstein series satisfying (3) and (4) of Lemma \ref{criterion} based on a numerical table in \cite{Oka} (Example \ref{Example cong}). 

\begin{prop}\label{prop:ab torsion-free}
Assume that $\ideal{n}$ is prime to $6\Delta_F$. 
If $p$ is prime to $6\ideal{n}$ and $\sharp(\mathfrak{o}_{F,+}^{\times}/\mathfrak{o}_{F,\ideal{n}}^{\times 2})$, 
then the assumption $(\ref{H_c^3})$ of Lemma \ref{criterion} is satisfied. 
\end{prop}
\begin{proof}
The Poincar\'e--Lefschetz duality theorem says that $H_c^3(Y(\ideal{n}),\integer{})\simeq H_1(Y(\ideal{n}),\integer{})$.
Hence it suffices to show that the maximal abelian quotient $\overline{\Gamma}^{\mathrm{ab}}$ of $\overline{\Gamma}$ is $p$-torsion-free. 
Since $\ideal{n}$ is prime to $2$, we have $\overline{\Gamma^1}=\Gamma^1$ and $\overline{\Gamma}/\Gamma^1 \simeq \mathfrak{o}_{F,+}^{\times}/\mathfrak{o}_{F,\ideal{n}}^{\times 2}$. 
Thus, by our assumption, it suffices to show that $\left(\Gamma^1\right)^{\text{ab}}$ is $p$-torsion-free. 
By taking conjugation, 
we may assume $\Gamma^1=\Gamma_1(\mathfrak{o}_F,\ideal{n})\cap \SL_2(\mathfrak{o}_F)$. 
Then $\left(\Gamma^1\right)^{\text{ab}}$ is torsion (\cite[Theorem 3]{Se}) and there is a non\nobreakdash-zero ideal $\ideal{m}$ of $\ideal{o}_F$ such that the principal congruence subgroup $\Gamma(\ideal{m})$ satisfies $\Gamma(\ideal{m})\subset [\Gamma^1:\Gamma^1]\subset \Gamma^1$ (\cite[Corollary 3 of Theorem 2]{Se}). 
We have 
$\left(\Gamma^1\right)^{\text{ab}}\simeq \left(\Gamma^1/\Gamma(\ideal{m})\right)^{\text{ab}}$.
We estimate the order of the right\nobreakdash-hand side. 
Put $H=\Gamma^1/\Gamma(\ideal{m})$. 
We have decompositions
$\SL_2(\mathfrak{o}_F)/\Gamma(\ideal{m})$ $=\prod_{i}\SL_2(\mathfrak{o}_F/\ideal{q}_i^{r_i})$ and $H=\prod_{i}H_{\ideal{q}_i}$. 
For each $i$, we define $\widehat{H}_{\ideal{q}_i}$ by the cartesian diagram
\[
\xymatrix{
H_{\ideal{q}_i} \ar@{^{(}->}[r]\ar@{}[rd]|{\square}& \SL_2(\mathfrak{o}_F/\ideal{q}_i^{r_i})\\
\widehat{H}_{\ideal{q}_i}\ar@{->>}[u] \ar@{^{(}->}[r]&\SL_2(\mathfrak{o}_{F_{{}_{\ideal{q}_i}}}). \ar@{->>}[u]
}
\]

We fix a prime ideal $\ideal{q}=\ideal{q}_i$ of $\mathfrak{o}_F$ and a positive integer $r=r_i$. 
Let $l$ denote the prime number such that $(l)=\ideal{q}\cap \Z$. 
The assertion follows from the following:

\textbf{Claim} 
(a) $\widehat{H}_{\ideal{q}}^{\text{ab}}=1$ in the case $\widehat{H}_{\ideal{q}}=\SL_2(\mathfrak{o}_{F_{{}_{\ideal{q}}}})$ and $(\ideal{q},6)=1$;

(b) $\widehat{H}_{\ideal{q}}^{\text{ab}}$ is an $l$-group in the case $\widehat{H}_{\ideal{q}}=\left\{\begin{pmatrix}a&b\\c&d\end{pmatrix}\equiv\begin{pmatrix}1&\ast\\0&1\end{pmatrix} \bmod \ideal{q}^r\right\}$. 

The assertion (a) is obtained by \cite[Proposition 2.6]{Fe--Si}. 
The assertion (b) follows from the following facts (i), (ii), and (iii):
(i) $\widehat{H}_{\ideal{q}}$ is generated by all elementary unipotents in $\widehat{H}_{\ideal{q}}$;
(ii) $\widehat{\Gamma}(\ideal{q}^{4r}) \subset \mathrm{EL}_2(\ideal{q}^{2r})$; 
(iii) $\mathrm{EL}_2(\ideal{q}^{2r})
\subset [\widehat{H}_{\ideal{q}}:\widehat{H}_{\ideal{q}}]$.
Here $\mathrm{EL}_2(\mathfrak{o}_{F_{{}_{\ideal{q}}}})$ denotes the subgroup of $\SL_2(\mathfrak{o}_{F_{{}_{\ideal{q}}}})$ generated by all elementary unipotents, and for a non-negative integer $m$, 
$\widehat{\Gamma}(\ideal{q}^m)=\ker (\SL_2(\mathfrak{o}_{F_{{}_{\ideal{q}}}})\twoheadrightarrow \SL_2(\mathfrak{o}_F/\ideal{q}^m) )$ and $\mathrm{EL}_2(\ideal{q}^{m})=\mathrm{EL}_2(\mathfrak{o}_{F_{{}_{\ideal{q}}}})\cap \widehat{\Gamma}(\ideal{q}^{m})$. 
Indeed, 
(i) implies that the image of $\widehat{H}_{\ideal{q}}/(\widehat{H}_{\ideal{q}}\cap \widehat{\Gamma}(\ideal{q}))$ in $\SL_2(\mathfrak{o}_{F_{{}_{\ideal{q}}}}/\ideal{q})$ is generated by $\begin{pmatrix}1&1\\0&1\end{pmatrix}$ and hence it is an $l$-group. 
Since $(\widehat{H}_{\ideal{q}}\cap \widehat{\Gamma}(\ideal{q}^m))/(\widehat{H}_{\ideal{q}}\cap\widehat{\Gamma}(\ideal{q}^{m+1})$ is an $l$-group for a non-negative integer $m$, (ii) and (iii) implies $\widehat{H}_{\ideal{q}}/\widehat{\Gamma}(\ideal{q}^{4r})$ is an $l$-group. 
Hence $\widehat{H}_{\ideal{q}}^{\text{ab}}$ is an $l$-group. 
The facts (i) and (ii) follow from 
\begin{align*}
\scriptsize 
\begin{pmatrix}a&b\\c&d\end{pmatrix}
=\begin{pmatrix}1&0\\a^{-1}c&1\end{pmatrix}
&\scriptsize 
\begin{pmatrix}1&-a^{-1}\\0&1\end{pmatrix}
\begin{pmatrix}1&0\\a^{-1}-1&1\end{pmatrix}\\
&\scriptsize 
\times 
\begin{pmatrix}1&1\\0&1\end{pmatrix}
\begin{pmatrix}1&0\\a-1&1\end{pmatrix}
\begin{pmatrix}1&-a^{-1}(1-a^{-2})\\0&1\end{pmatrix}
\begin{pmatrix}1&a^{-1}b\\0&1\end{pmatrix}.
\end{align*}
The fact (iii) follows from the same argument as in the proof of \cite[Proposition 2.6]{Fe--Si} in the case $n=2$ and general $\sharp \mathfrak{o}_{F_{{}_{\ideal{q}}}}/\ideal{q}$ and the commutator relation \cite[(1)]{Fe--Si}.
\end{proof}

Let $\varepsilon_0$ denote the fundamental unit of $F$. 
We put $\varepsilon_+=\varepsilon_0$ (resp. $\varepsilon_0^2$) if $N(\varepsilon_0)=1$ (resp. $N(\varepsilon_0)=-1$).

\begin{prop}\label{prop:bou torsion-free}
If $p\nmid N(\varepsilon_+-1)$ and $\ideal{n}$ is a prime ideal $\ideal{q}$ of $\mathfrak{o}_F$ such that $\ideal{q}$ is prime to $6\Delta_F$, 
then the assumption $(\ref{H^2(partial)})$ of Lemma \ref{criterion} is satisfied.
\end{prop}
\begin{proof}
We may assume $\Gamma=\Gamma_1(\mathfrak{o}_F,\ideal{n})$ by taking conjugation. 
Fix a cusp $s\in C(\Gamma)$. 
As mentioned in \cite[p.260]{Gha}, 
$H^2(\overline{\Gamma_s},\cal{O})$ is torsion-free if and only if 
$H^1(\overline{\Gamma_s},K/\cal{O})$ is divisible. 
Main tools for the proof of the divisibility are 
the description 
$H^1(\overline{\Gamma_s},K/\cal{O})=H^1(\overline{\alpha^{-1}\Gamma \alpha \cap B_{\infty}},K/\cal{O})$ 
and the Hochschild--Serre spectral sequence 
\begin{align*}
E_2^{i,j}=H^i(\overline{\alpha^{-1}\Gamma \alpha\cap B_{\infty}}/\overline{\alpha^{-1}\Gamma \alpha\cap U_{\infty}},H^j(\overline{\alpha^{-1}\Gamma \alpha\cap U_{\infty}},K/\cal{O}))
\Rightarrow 
H^{i+j}(\overline{\alpha^{-1}\Gamma \alpha\cap B_{\infty}},K/\cal{O}),
\end{align*}
where $\alpha\in \SL_2(\mathfrak{o}_F)$ such that $\alpha(\infty)=s$, $B_{\infty}$ denotes the standard Borel subgroup of upper triangular matrices, $U_{\infty}$ denotes the unipotent radical of $B_{\infty}$, and the bar ${}^{-}$ means image in $\GL_2(F)/(\GL_2(F)\cap F^{\times})$. 
Let $T_{\infty}$ denote the standard torus of $B_{\infty}$. 
By the same argument as in \cite[\S 3.4.2]{Gha}, 
our assertion follows from the following (\ref{U}), (\ref{T}), and (\ref{extension}):
\begin{eqnarray}\label{U}
\text{$\overline{\alpha^{-1}\Gamma \alpha\cap U_{\infty}} \simeq \ideal{q}^{1-e}$\ \  if $(y,\ideal{q})=\ideal{q}^e$;}
\end{eqnarray}
\begin{align}\label{T}
\overline{\alpha^{-1}\Gamma \alpha\cap T_{\infty}} \simeq \mathfrak{o}_{F,+}^{\times};
\end{align}
\begin{align}\label{extension}
1
\to
\overline{\alpha^{-1}\Gamma \alpha\cap U_{\infty}}
\to
\overline{\alpha^{-1}\Gamma \alpha\cap B_{\infty}}
\to 
\overline{\alpha^{-1}\Gamma \alpha\cap T_{\infty}}
\to
1.
\end{align}

Fix $\alpha=\begin{pmatrix}x&\beta\\y&\delta\end{pmatrix}\in\SL_2(\mathfrak{o}_F)$ such that $\alpha(\infty)=s$. 
We may assume that if $(y,\ideal{q})=1$, then $(\delta,\ideal{q})=\ideal{q}$. Indeed, since $(x\ideal{q},y)=1$, there is $\begin{pmatrix}x&\beta\\y&\delta\end{pmatrix}\in\SL_2(\mathfrak{o}_F)$ with $(\delta,\ideal{q})=\ideal{q}$. 

First we prove (\ref{U}). 
Suppose that $\begin{pmatrix}1&b\\0&1\end{pmatrix} \in \alpha^{-1}\Gamma \alpha\cap U_{\infty}$. 
The direct calculation shows that the condition 
$\alpha \begin{pmatrix}1&b\\0&1\end{pmatrix} \alpha^{-1}\in \Gamma$ is equivalent to the condition 
$bx^2\in \mathfrak{o}_F$, $by^2 \in \ideal{q}$, and $bxy \in \ideal{q}$. 
Since $(x,y)=1$, we have $b\in \mathfrak{o}_F$. 
If $(y,\ideal{q})=\ideal{q}^{e}$, then $b\in \ideal{q}^{1-e}$ as desired. 

Next we prove (\ref{T}). 
Suppose that $\begin{pmatrix}a&0\\0&d\end{pmatrix} \in \alpha^{-1}\Gamma\alpha\cap T_{\infty}$. 
As in the proof of (\ref{U}), the direct calculation shows that 
if $(y,\ideal{q})=1$ (resp. $(y,\ideal{q})=\ideal{q}$), then 
$a \equiv 1\ (\bmod\ \ideal{q})$ (resp. $d \equiv 1\ (\bmod\ \ideal{q})$) and hence 
$\overline{\begin{pmatrix}a&0\\0&d\end{pmatrix}}=\overline{\begin{pmatrix}1&0\\0&a^{-1}d\end{pmatrix}}$
$\left(\text{resp.}\  \overline{\begin{pmatrix}a&0\\0&d\end{pmatrix}}=\overline{\begin{pmatrix}ad^{-1}&0\\0&1\end{pmatrix}}\right)$.

Finally we prove (\ref{extension}). 
For $\begin{pmatrix}a&b\\0&d\end{pmatrix} \in \alpha^{-1}\Gamma \alpha\cap B_{\infty}$, it suffices to show that $\begin{pmatrix}1&-a^{-1} b\\0&1\end{pmatrix} \in \alpha^{-1}\Gamma \alpha\cap U_{\infty}$.
As in the proof of (\ref{U}),
it follows from 
the condition $\alpha \begin{pmatrix}a&b\\0&d\end{pmatrix}\alpha^{-1} \in \Gamma$.
\end{proof}

\begin{ex}\label{Example cong}
We give an example satisfying the assumptions of Lemma \ref{criterion} 
in the case $F=\Q(\sqrt{2})$ with $h_F^+=1$, $\Delta_F=8$, and $\varepsilon_0=1+\sqrt{2}$. 
By \cite[\S4, p.1137]{Oka}, 
for the non-trivial character $\chi$ of $\Gal(F(\sqrt{5})/F)$ whose conductor is a prime ideal $(5)$ of $\mathfrak{o}_F$, 
we have 
\begin{align*}
L(-1,\chi)=\frac{28}{5}.
\end{align*}

A pair of characters $\varphi=\chi^{-1}$ and the trivial character $\psi=\textbf{1}$ satisfies (Eis condition). 
One can see that $p=7$ with $(p,6\Delta_F)=1$ and the Eisenstein series $\textbf{E}_2(\varphi,\psi)$ of level $\Gamma_{1}(\ideal{o}_F,(5))$ satisfy all the assumptions of Lemma \ref{criterion}. 
\end{ex}

\section{Congruences between $L$-values}\label{congruence}
%
The purpose of this section is to prove the main theorem (Theorem \ref{main theorem}=Theorem \ref{thm:congruence}) of this paper. 
We use the assumption $h_F^+=1$ in the proof of the isomorphism (\ref{key isomorphism}) between a relative cohomology and the corresponding partial parabolic cohomology.
We keep the notation in \S \ref{subsection:Rationality and Integrality}.

%
%
\subsection{Canonical periods}\label{subsection:Canonical periods}
%
%

Let $\textbf{f}\in S_2(\ideal{n},\integer{})$ be a normalized Hecke eigenform for all $T(\ideal{q})$ and $U(\ideal{q})$ with character $\chi$. 
Let $\epsilon$ denote $\epsilon_{{}_{\textbf{E}}}$ defined at the beginning of \S\ref{subsection:Rationality and Integrality}. 

We define the canonical period $\Omega_{\textbf{f}}^{\epsilon}$ of $\mathbf{f}$. 
Let $\mathfrak{p}_{\textbf{f}}$ denote the prime ideal of $\mathscr{H}_2(\ideal{n},\integer{})$ generated by $T(\mathfrak{q})-C(\ideal{q},\textbf{f})$ and $S(\mathfrak{q})-\chi^{-1}(\mathfrak{q})$ for all non-zero prime ideals $\mathfrak{q}$ of $\mathfrak{o}_F$ prime to $\ideal{n}$ and $U(\ideal{q})-C(\ideal{q},\textbf{f})$ for all non-zero prime ideals $\ideal{q}$ of $\mathfrak{o}_F$ dividing $\ideal{n}$. 
The isomorphism (\ref{+,+ decomp}) and the $q$-expansion principle over $\C$ imply that $\dim_{\C}\left( H_{\pa}^n(Y(\ideal{n}),\C)[\mathfrak{p}_{\textbf{f}},\epsilon]\right)=1$ and $\rank_{\integer{}}\left( \widetilde{H}_{\pa}^n(Y(\ideal{n}),\integer{})[\mathfrak{p}_{\textbf{f}},\epsilon]\right)=1$.
Choose a generator $[\delta_{\textbf{f}}]^{\epsilon}$ of $\widetilde{H}_{\pa}^n(Y(\ideal{n}),\integer{})[\mathfrak{p}_{\textbf{f}},\epsilon]$. 
Let $[\omega_{\textbf{f}}]^{\epsilon}$ denote the projection of $[\omega_{\textbf{f}}]$ to the $\epsilon$-part. 
We define the canonical period $\Omega_{\textbf{f}}^{\epsilon}\in \C^{\times}$ of $\textbf{f}$ by 
\begin{align*}
[\omega_{\textbf{f}}]^{\epsilon}=\Omega_{\textbf{f}}^{\epsilon}[\delta_{\textbf{f}}]^{\epsilon}. 
\end{align*}

%
%
\subsection{Congruences between special values}
%
%
For modular forms $\textbf{f},\textbf{g}\in M_2(\ideal{n},\integer{})$, 
we define the congruence $\textbf{f}\equiv \textbf{g}\ (\bmod \varpi)$ by $C(\ideal{m},\textbf{f})\equiv C(\ideal{m},\textbf{g})\ (\bmod \varpi)$ for all non-zero ideals $\mathfrak{m}$ of $\ideal{o}_F$. 

\begin{thm}\label{thm:congruence}
Let $p$ be a prime number such that $p\ge n+2$ and $p$ is prime to $\mathfrak{n}$ and $6\Delta_F$. 
Assume that $h_F^+=1$. 
Let $\varphi$ and $\psi$ be primitive narrow ray class characters of $F$ satisfying $\mathrm{(Eis\  condition)}$ at the beginning of \S\ref{subsection:Rationality and Integrality} and $\epsilon$ the character $\epsilon_{{}_{\mathbf{E}}}$ of the Weyl group $W_G$ defined after $\mathrm{(Eis\ condition)}$. 
Put $\chi=\varphi\psi$. 
Let $\mathbf{f}\in S_2(\ideal{n},\integer{})$ be a normalized Hecke eigenform for all $T(\ideal{q})$ and $U(\ideal{q})$ with character $\chi$. 
We assume the following conditions $(a)$, $(b)$, $(c)$$:$ 
\begin{enumerate}[$(a)$] 
\item $\mathbf{f} \equiv \mathbf{E}\ (\bmod \varpi)$$;$

\item $H^{n}(\partial \left(Y(\ideal{n})^{\mathrm{BS}}\right),\integer{})_{\ideal{m}}$, $H_c^{n+1}(Y(\ideal{n}),\integer{})_{\ideal{m}}$, and $H^n(D_{C_{\infty}}(\ideal{n}),\integer{})_{\ideal{m}_{\mathbf{E}}'}$ are torsion-free, 
where $\ideal{m}$ $($resp. $\ideal{m}_{\mathbf{E}}'$$)$ is the maximal ideal of $\cal{H}_2(\ideal{n},\integer{})$ $($resp. $\mathbb{H}_2(\ideal{n},\integer{})'$$)$ defined before Theorem \ref{cong mod} $($resp. Proposition \ref{H^{n-1} vanishing}$)$$;$

\item $C(\ideal{q},\mathbf{E})\not\equiv N(\ideal{q})\ (\bmod \varpi)$ for some prime ideal $\ideal{q}$ dividing $\ideal{n}$. 

\end{enumerate}
Then there exists $u \in \integer{}^{\times}$ such that, 
for every narrow ray class character $\eta$ of $F$, whose conductor is denoted by $\ideal{m}_{\eta}$, such that $\ideal{n}|\ideal{m}_{\eta}$ and 
$\eta=\epsilon$ on $W_G\simeq \A_{F,\infty}^{\times}/\A_{F,\infty,+}^{\times}$, 
the both values $\tau({\eta}^{-1})D(1,\mathbf{f},\eta)/(2\pi \sqrt{-1})^n\Omega_{\mathbf{f}}^{\epsilon}$ and $\tau(\eta^{-1})D(1,\mathbf{E},\eta)/(2\pi \sqrt{-1})^n$ belong to $\integer{}(\eta)$ and the following congruence holds$:$ 
\begin{align*}
&\tau(\eta^{-1})\frac{D(1,\mathbf{f},\eta)}{(2\pi\sqrt{-1})^n\Omega_{\mathbf{f}}^{\epsilon}}
= u \tau(\eta^{-1})\frac{D(1,\mathbf{E},\eta)}{(2\pi\sqrt{-1})^n} 
\ \ \text{in} \ \ \integer{}(\eta)/\varpi.
\end{align*}
Here $\tau(\eta^{-1})$ denotes the Gauss sum attached to $\eta^{-1}$ (cf. (\ref{Gauss})), $D(1,\ast,\eta)$ is the Dirichlet series recalled in \S\ref{Diri}, $\integer{}(\eta)$ denotes the ring of integers of the field $K(\eta)$ generated by $\im(\eta)$ over $K$. 
\end{thm}
\begin{rem}
In general, 
the value $\tau(\eta^{-1}) D(1,\mathbf{E},\eta)/(2\pi\sqrt{-1})^n$ is non-zero in $\integer{}(\eta)/\varpi$ by the proof of Theorem \ref{cong mod} (see after (\ref{denominator})).
\end{rem}
\begin{proof}
For the moment, we admit Theorem \ref{partial one}, which shall be proved in \S \ref{Pf of main}.
We abbreviate $Y(\ideal{n})$ to $Y$ and $D_{C_{\infty}}(\ideal{n})$ to $D_{C_{\infty}}$. 
For $A=\integer{}$ or $K$, 
we define the partial parabolic cohomology $H_{\pa}^n(Y,D_{C_{\infty}};A)$ to be the image of 
\[
H^n(Y^{\text{BS}},D_{C_{\infty}};A)
\to 
H^n(Y,A) 
\]
and put 
\begin{align*}
\widetilde{H}^n(Y^{\text{BS}},D_{C_{\infty}};\integer{})
&=\im \left(H^n(Y^{\text{BS}},D_{C_{\infty}};\integer{})\to H^n(Y^{\text{BS}},D_{C_{\infty}};K)\right),\\
\widetilde{H}_{\pa}^n(Y,D_{C_{\infty}};\integer{})
&=\im \left(H_{\pa}^n(Y,D_{C_{\infty}};\integer{})\to H_{\pa}^n(Y,D_{C_{\infty}};K)\right).
\end{align*}
By Proposition \ref{H^{n-1} vanishing}, 
$H^n(Y^{\text{BS}},D_{C_{\infty}};K)_{\ideal{m}_{\mathbf{E}}'} \to H_{\pa}^n(Y,D_{C_{\infty}};K)_{\ideal{m}_{\mathbf{E}}'}$ is an isomorphism and induces an isomorphism 
\begin{align}\label{key isomorphism}
\widetilde{H}^n(Y^{\text{BS}},D_{C_{\infty}};\integer{})_{\ideal{m}_{\mathbf{E}}'} 
\simeq
\widetilde{H}_{\pa}^n(Y,D_{C_{\infty}};\integer{})_{\ideal{m}_{\mathbf{E}}'}.
\end{align}
Hence, by Corollary \ref{thm:integral} and the definition of $[\delta_{\textbf{f}}]^{\epsilon}$ in \S\ref{subsection:Canonical periods}, 
$[\omega_{\textbf{E}}]_{\text{rel}}^{\epsilon}$ and $[\delta_{\textbf{f}}]_{\text{rel}}^{\epsilon}:=[\omega_{\textbf{f}}]_{\text{rel}}^{\epsilon}/\Omega_{\textbf{f}}^{\epsilon}$ belong to $\widetilde{H}^n(Y^{\text{BS}},D_{C_{\infty}};\integer{})$.
Furthermore Theorem \ref{partial one} and the isomorphism (\ref{key isomorphism}) imply 
\[
\text{
$[\delta_{\textbf{f}}]_{\text{rel}}^{\epsilon}=u [\omega_{\textbf{E}}]_{\text{rel}}^{\epsilon}$  
\ in \ 
$\widetilde{H}^n(Y^{\text{BS}},D_{C_{\infty}};\integer{})/\varpi$} 
\]
for some $u\in \integer{}^{\times}$.
Now our assertion follows from Proposition \ref{Modular symbol} and Proposition \ref{Modular symbol anti-hol}.  
\end{proof}

\section{Congruences between cohomology classes}\label{comparison}
%

The purpose of this section is to prove Theorem \ref{partial one},
that is, a congruence between a Hilbert eigenform and a Hilbert Eisenstein series gives rise to corresponding congruence between the associated cohomology classes under certain assumptions. 

In this section, we assume that $2\le n \le p-2$ and $K$ is a finite extension of the composite field of $\iota_p(F')$ and $\Phi_p$. 
Here $\iota_p :\overline{\Q} \to \overline{\Q}_p$ is the fixed embedding and $F'$ (resp. $\Phi_p$) is the field defined in \S \ref{subsection:Geometric HMF} (resp. Proposition \ref{const}). 
Let $\integer{}$ be the ring of integers of $K$, $\varpi$ a uniformizer of $\integer{}$, and $\kappa$ the residue field of $\integer{}$.

%
%
\subsection{Comparison theorem for torsion cohomology}
%
%

In this subsection, 
we briefly review the fully faithful functor from the category of finitely generated filtered $\varphi$-modules to the category of representations of $G_{\Qp}=\Gal(\overline{\Q}_p/\Qp)$ on $\integer{}$-modules of finite length, and state the comparison theorem between the parabolic \'etale cohomology and the parabolic log\nobreakdash-crystalline cohomology for Hilbert modular varieties, which we shall use in the following subsections.

For a non-negative integer $r$, 
let $\textbf{MF}_{\integer{}}^r$ denote the category whose objects are the following triples $(M,(\mathrm{Fil}^i M)_{i\in\Z},(\varphi_M^i)_{i\in \Z})$: 
\begin{enumerate}[(1)]
\item $M$ is a finitely generated $\integer{}$-module;
\item $(\mathrm{Fil}^i M)_{i\in \Z}$ is a decreasing filtration on $M$ by $\integer{}$-submodules such that $\mathrm{Fil}^0 M=M$ and $\mathrm{Fil}^{r+1} M=0$;
\item $\varphi_M^i \colon \mathrm{Fil}^i M \to M$ is an $\integer{}$-linear homomorphism such that $\varphi_M^i\mid_{\mathrm{Fil}^{i+1} M}=p\varphi_M^{i+1}$ and $\sum_{i=0}^r\varphi_M^i(\mathrm{Fil}^iM)=M$.
\end{enumerate}

A morphism in $\textbf{MF}_{\integer{}}^r$ is a homomorphism of filtered $\integer{}$-modules compatible with $\varphi^{\bullet}$. 
It is known that any morphism $\eta : M\to M'$ in $\textbf{MF}_{\integer{}}^r$ is strict with respect to the filtrations, that is, $\eta(\mathrm{Fil}^i M)=\mathrm{Fil}^i M'\cap \eta(M)$ for each $i\in \Z$ (\cite[1.10 (b)]{Fo--La}). 
This implies that $\textbf{MF}_{\integer{}}^r$ is an abelian category as follows. 
Let $\eta:M\to M'$ be a morphism in $\textbf{MF}_{\integer{}}^r$, and let $\underline{\eta}$ denote $\eta$ regarded as a homomorphism of underlying $\integer{}$-modules. 
Then the $\integer{}$\nobreakdash-module $N:=\ker(\underline{\eta})$ with $\mathrm{Fil}^i N$ and $\varphi_N^i$ defined by 
$\mathrm{Fil}^i N=N \cap \mathrm{Fil}^i M$ and $\varphi_N^i=\varphi_M^i|_{N}$, respectively, belongs to $\textbf{MF}_{\integer{}}^r$ and gives the kernel of $\eta$ in $\textbf{MF}_{\integer{}}^r$. 
Let $N'$ denote $\mathrm{coker}(\underline{\eta})$.
We define a filtration $\mathrm{Fil}^i N'$ and an $\integer{}$-linear homomorphism $\varphi_{N'}^i$ by 
$\mathrm{Fil}^i N'= \mathrm{Fil}^i M'/\eta(\mathrm{Fil}^i M)$ and the homomorphism induced by $\varphi_M^i$ and $\varphi_{M'}^i$, respectively. 
Note that $\mathrm{Fil}^i N' \to N'$ is injective because $\eta$ is strict, and hence $\mathrm{Fil}^i N'$ may be regarded as an $\integer{}$-submodule of $N'$. The triple $(N',(\mathrm{Fil}^i N')_{i\in\Z},(\varphi_{N'}^i)_{i\in \Z})$ belongs to $\textbf{MF}_{\integer{}}^r$ and gives the cokernel of $\eta$ in $\textbf{MF}_{\integer{}}^r$. 
The strictness of $\eta$ further shows that we have 
$\mathrm{Fil}^i (\im(\eta)) = \eta (M) \cap \mathrm{Fil}^i M' = \eta(\mathrm{Fil}^i M) \simeq \mathrm{Fil}^i (\mathrm{coim}(\eta))$ and hence
$\im(\eta)=\mathrm{coim}(\eta)$ in $\textbf{MF}_{\integer{}}^r$.

Let $\textbf{MF}_{\kappa}^r$ denote the full subcategory of $\textbf{MF}_{\cal{O}}^r$ consisting of objects $M$ satisfying $\varpi M=0$. 
Let $\textbf{Rep}_{\integer{}}(G_{\Qp})$ denote the category of representations of $G_{\Q_p}$ on $\integer{}$-modules of finite length. 
For an integer $r$ such that $0\le r\le p-2$, there exists a fully faithful functor 
\begin{align*}
T_{\mathrm{cris}}&\colon\textbf{MF}_{\integer{}}^r \rightarrow \textbf{Rep}_{\integer{}}(G_{\Qp}).
\end{align*}
given by J-M. Fontaine and G. Laffaille (\cite{Fo--La}, \cite{Br--Me}, \cite{Wach}). 
Let $\textbf{Rep}_{\integer{},\mathrm{cris}}^{r}(G_{\Q_p})$ denote 
the essential image of $\textbf{MF}_{\cal{O}}^r$ by $T_{\mathrm{cris}}$. 
For an object $T$ of $\textbf{Rep}_{\integer{},\mathrm{cris}}^{r}(G_{\Q_p})$, the Hodge\nobreakdash-Tate weights of $T$ mean $s\in \Z$ for which $\Gr^s M\neq 0$, 
where $M$ is an object of $\textbf{MF}_{\cal{O}}^r$ such that $T_{\mathrm{cris}}(M)\simeq T$. 

Now we can state the comparison theorem for the Hilbert modular varieties.

By the comparison theorem for torsion cohomology (in the good reduction case) proved by G. Faltings (\cite[Theorem 5.3]{Fa} 
(see also \cite[Theorem 3.2.4.6]{Br}=\cite[Theorem 5.1]{Tsu} and \cite[Theorem 3.2.4.7]{Br} for an alternative proof with an extension to the log-smooth reduction case, but without compact support), 
for $(X^{\mathrm{tor}},X)=(M^{1,\mathrm{tor}},M^1)$ or $(M^{\mathrm{tor}},M)$ defined in \S \ref{subsection:HMV}, 
there are canonical $G_{\Q_p}$-equivariant $\cal{O}$-linear isomorphisms 
\begin{align}
\label{comparison log}
H_{\mathrm{\acute et}}^n({X}_{\overline{\Q}_p},\integer{})
&\simeq  
T_{\mathrm{cris}}\left(H_{\text{log-cris}}^n(X_{\Z_p}^{\mathrm{tor}})\otimes_{\Z_p} \integer{}\right),\\
\label{comparison cpt}
H_{\mathrm{\acute et},c}^n({X}_{\overline{\Q}_p},\integer{})
&\simeq  
T_{\mathrm{cris}}\left(H_{\text{log-cris},c}^n(X_{\Z_p}^{\mathrm{tor}})\otimes_{\Z_p} \integer{}\right).
\end{align}

For $?=\phi$ or $c$ and $A=\integer{}$ or $K$, 
we simply write $H_{\text{log-cris},?}^n(X_{\Zp}^{\mathrm{tor}})\otimes_{\Zp}A$ for $H_{\text{log-cris},?}^n(X^{\mathrm{tor}})_{A}$.
For $A=\cal{O}$ or $K$, 
we define $H_{\mathrm{\acute et},\pa}^n({X}_{\overline{\Q}_p},A)$ and $H_{\text{log-cris},\pa}^n(X^{\mathrm{tor}})_A$ by
\begin{align*}
H_{\mathrm{\acute et},\pa}^n({X}_{\overline{\Q}_p},A)
&=\im\left(H_{\mathrm{\acute et},c}^n({X}_{\overline{\Q}_p},A)
\to 
H_{\mathrm{\acute et}}^n({X}_{\overline{\Q}_p},A)
\right),\\
H_{\text{log-cris},\pa}^n(X^{\mathrm{tor}})_A
&=\im\left(H_{\text{log-cris},c}^n(X^{\mathrm{tor}})_A\to 
H_{\text{log-cris}}^n(X^{\text{tor}})_A\right). 
\end{align*}

We obtain the following $G_{\Q_p}$-equivariant $\cal{O}$-linear isomorphisms from (\ref{comparison log}) and (\ref{comparison cpt}):
\begin{align}\label{comparison par}
H_{\mathrm{\acute et},\pa}^n({X}_{\overline{\Q}_p},\integer{})
&\simeq  
T_{\mathrm{cris}}\left(H_{\text{log-cris},\pa}^n(X^{\mathrm{tor}})_{\integer{}}\right).
\end{align}

For $?=\phi$ or $\pa$, we put
\begin{align*}
\widetilde{H}_{\mathrm{\acute et},?}^n(M_{\overline{\Q}},\integer{})
&=\im \left(H_{\mathrm{\acute et},?}^n(M_{\overline{\Q}},\integer{}) \to H_{\mathrm{\acute et},?}^n(M_{\overline{\Q}},K)\right), \\
\widetilde{H}_{\text{log-cris},?}^n(M^{\text{tor}})_{\integer{}}
&=\im \left(H_{\text{log-cris},?}^n(M^{\text{tor}})_{\integer{}} \to H_{\text{log-cris},?}^n(M^{\text{tor}})_K\right). 
\end{align*}
We define objects 
$\widetilde{H}_{\mathrm{\acute et},\pa}^n({M}_{\overline{\Q}},\kappa)$ of $\textbf{Rep}_{\integer{},\mathrm{cris}}^{p-2}(G_{\Q_p})$ and 
$\widetilde{H}_{\text{log-cris},\pa}^n({M}^{\text{tor}})_{\kappa}$ of $\textbf{MF}_{\kappa}^{p-2}$ by
\begin{align*}
\widetilde{H}_{\mathrm{\acute et},\pa}^n({M}_{\overline{\Q}},\kappa)
=\widetilde{H}_{\mathrm{\acute et},\pa}^n({M}_{\overline{\Q}},\integer{})/\varpi,\ \ 
\widetilde{H}_{\text{log-cris},\pa}^n({M}^{\text{tor}})_{\kappa}
=\widetilde{H}_{\text{log-cris},\pa}^n({M}^{\text{tor}})_{\integer{}}/\varpi.
\end{align*}

%
%
\subsection{Analogue of a multiplicity one theorem}\label{main result}
%
%

In this subsection, we state the main theorem of \S \ref{comparison}, 
which shall be proved in \S \ref{Pf of main}.

\begin{thm}\label{partial one}
Under the same notation and assumptions as Theorem \ref{thm:congruence}, 
$[\delta_{\mathbf{f}}]^{\epsilon} (\bmod \varpi)$ and $[\omega_{\mathbf{E}}]^{\epsilon} (\bmod \varpi)$ belong to $\widetilde{H}_{\mathrm{\acute et},\pa}^n(M_{\overline{\Q}},\kappa)$ and 
there exists $u\in \integer{}^{\times}$ such that 
\[
\text{$[\delta_{\mathbf{f}}]^{\epsilon}=u [\omega_{\mathbf{E}}]^{\epsilon}$\ \ in   $\widetilde{H}_{\mathrm{\acute et},\pa}^n(M_{\overline{\Q}},\kappa)$.}
\]
\end{thm}
\begin{rem}
M. Dimitrov \cite[Theorem 6.7]{Dim2} proved that a multiplicity one theorem holds for the $\textbf{f}$-parts of $H_{\mathrm{\acute et},\pa}^n({M}_{\overline{\Q}},\kappa)$ and $H_{\mathrm{\acute et},\pa}^n({M}_{\overline{\Q}},\integer{})$ 
if the residual Galois representation $\bar{\rho_{\textbf{f}}}$ is irreducible under some assumptions. 
\end{rem}

In the rest of this subsection, we introduce some objects of $\textbf{Rep}_{\integer{},\mathrm{cris}}^{p-2}(G_{\Q_p})$ and $\textbf{MF}_{\integer{}}^{p-2}$ associated to $\mathbf{E}$ and $\mathbf{f}$, which will be used in the proof of Theorem \ref{partial one}.

For $?=\phi$ or $c$, 
let $T(\ideal{a})_{\mathrm{\acute et}}$ and $U(\ideal{a})_{\mathrm{\acute et}}$ (resp. $T(\ideal{a})_{\mathrm{dR}}$ and $U(\ideal{a})_{\mathrm{dR}}$) be the Hecke operators on $H_{\mathrm{\acute et},?}^n(M_{\overline{\Q}_p},\Qp)$ 
(resp. $H_{\text{log-dR},?}^n({M}^{\mathrm{tor}})_{\Qp}$) induced by the Hecke correspondences $T(\ideal{a})$ and $U(\ideal{a})$ on $M_{\overline{\Q}_p}^{1,\mathrm{tor}}$ (resp. $M_{\Qp}^{1,\mathrm{tor}}$) (see \S 1.9), respectively.
We define the Hecke operators $T(\ideal{a})_{\mathrm{cris}}$ and $U(\ideal{a})_{\mathrm{cris}}$ on $H_{\text{log-cris},?}^n({M}^{\mathrm{tor}})_{\Qp}$ via the isomorphism $H_{\text{log-cris},?}^n({M}^{\mathrm{tor}})_{\Qp} \simeq H_{\text{log-dR},?}^n({M}^{\mathrm{tor}})_{\Qp}$. 
By the de Rham conjecture (\cite[\S VIII]{Fa}), the comparison isomorphism $\mathrm{c}_{M^1,\mathrm{dR}}$ between $H_{\mathrm{\acute et},?}^n(M_{\overline{\Q}_p}^1,\Qp)$ and $H_{\text{log-dR},?}^n({M}^{1,\mathrm{tor}})_{\Qp}$ is Hecke-equivariant (cf. \cite[Theorem 1.2]{Fa--Jo}).
Since $\mathrm{c}_{X,\mathrm{dR}}$ is compatible with the the comparison isomorphism $\mathrm{c}_{X,\text{cris}}$ between $H_{\mathrm{\acute et},?}^n(X_{\overline{\Q}_p},\Qp)$ and $H_{\text{log-cris},?}^n(X^{\mathrm{tor}})_{\Qp}$ for $X=M^1$ or $M$, the isomorphism $\mathrm{c}_{M,\mathrm{cris}}$ is Hecke-equivariant and $\widetilde{H}_{\text{log-cris},?}^n({M}^{\mathrm{tor}})_{\Zp}$ is stable under $T(\ideal{a})_{\mathrm{cris}}$ and $U(\ideal{a})_{\mathrm{cris}}$.

Set $D=M^{\mathrm{tor}}-M$.
Let $\cal{H}_2(\ideal{n},\integer{})$ be the commutative $\integer{}$-subalgebra of $\End_{\integer{}}(\widetilde{H}_{\mathrm{\acute et},c}^n(M_{\overline{\Q}},\integer{}))$ 
$\oplus \End_{\integer{}}(\widetilde{H}_{\mathrm{\acute et}}^n(M_{\overline{\Q}},\integer{}))\oplus \End_{\integer{}}(\widetilde{H}_{\mathrm{\acute et}}^n(D_{\overline{\Q}},\integer{}))\oplus \End_{\integer{}}(\widetilde{H}_{\mathrm{\acute et},c}^{n+1}(M_{\overline{\Q}},\integer{}))$
generated by $T(\ideal{q})_{\mathrm{\acute et}}$ for all non\nobreakdash-zero prime ideals $\ideal{q}$ of $\ideal{o}_F$ prime to $\ideal{n}$ and $U(\ideal{q})_{\mathrm{\acute et}}$ for all non-zero prime ideals $\ideal{q}$ of $\ideal{o}_F$ dividing $\ideal{n}$. 
Let $\mathfrak{p}_{\textbf{E}}$ (resp. $\mathfrak{p}_{\textbf{f}}$) be the prime ideal of $\mathcal{H}_2(\ideal{n},\integer{})$ generated by $T(\ideal{q})_{\mathrm{\acute et}}-C(\ideal{q},\textbf{E})$ (resp. $T(\ideal{q})_{\mathrm{\acute et}}-C(\ideal{q},\textbf{f})$) for all non\nobreakdash-zero prime ideals $\ideal{q}$ of $\ideal{o}_F$ prime to $\ideal{n}$ and $U(\ideal{q})_{\mathrm{\acute et}}-C(\ideal{q},\textbf{E})$ (resp. $U(\ideal{q})_{\mathrm{\acute et}}-C(\ideal{q},\textbf{f})$) for all non-zero prime ideals $\ideal{q}$ of $\ideal{o}_F$ dividing $\ideal{n}$. 
Let $\ideal{m}$ denote the maximal ideal $(\varpi, \ideal{p}_{\textbf{E}})$.
We may regard $\cal{H}_2(\ideal{n},\integer{})$ as the $\integer{}$-subalgebra of  $\End_{\textbf{MF}_{\integer{}}^{p-2}}(\widetilde{H}_{\text{log-cris},c}^n(M^{\mathrm{tor}})_{\integer{}})
\oplus \End_{\textbf{MF}_{\integer{}}^{p-2}}(\widetilde{H}_{\text{log-cris}}^n(M^{\mathrm{tor}})_{\integer{}}) 
\oplus \End_{\textbf{MF}_{\integer{}}^{p-2}}(\widetilde{H}_{\text{log-cris}}^n(D)_{\integer{}})
\oplus \End_{\textbf{MF}_{\integer{}}^{p-2}}(\widetilde{H}_{\text{log-cris},c}^{n+1}(M^{\mathrm{tor}})_{\integer{}})$
via the comparison theorem.

We shall consider the $\textbf{f}$-parts of 
$\widetilde{H}_{\mathrm{\acute et},\pa}^n(M_{\overline{\Q}},\integer{})$ and 
$\widetilde{H}_{\text{log-cris},\pa}^n({M}^{\mathrm{tor}})_{\integer{}}$ etc. defined by 
\begin{align*}
\widetilde{\overline{T}}=\widetilde{H}_{\mathrm{\acute et},\pa}^n({M}_{\overline{\Q}},\kappa)[\ideal{m}], \ \ \ \ \ \ \ \ \,
\widetilde{T_{\textbf{f}}}=\widetilde{H}_{\mathrm{\acute et},\pa}^n({M}_{\overline{\Q}},\integer{})[\mathfrak{p}_{\textbf{f}}],\ \ \ \ \ \ \ \ \,
\widetilde{\overline{T}}_{\textbf{f}}=\widetilde{T_{\textbf{f}}}/\varpi,\ \\
\widetilde{\overline{M}}=\widetilde{H}_{\text{log-cris},\pa}^n({M}^{\mathrm{tor}})_{\kappa}[\ideal{m}],\ \ 
\widetilde{M}_{\textbf{f}}=\widetilde{H}_{\text{log-cris},\pa}^n({M}^{\text{tor}})_{\integer{}}[\mathfrak{p}_{\textbf{f}}],\ \ 
\widetilde{\overline{M}}_{\textbf{f}}=\widetilde{M}_{\textbf{f}}/\varpi.
\end{align*}

A main tool for our proof is the $\textbf{E}$-parts of 
$\widetilde{H}_{\mathrm{\acute et}}^n(M_{\overline{\Q}},\integer{})$ and 
$\widetilde{H}_{\text{log-cris}}^n({M}^{\mathrm{tor}})_{\integer{}}$ defined by 
\begin{align*}
\widetilde{T}_{\textbf{E}}=\widetilde{H}_{\mathrm{\acute et}}^n(M_{\overline{\Q}},\integer{})[\mathfrak{p}_{\textbf{E}}],
\ \ \ \ \widetilde{M}_{\textbf{E}}=\widetilde{H}_{\text{log-cris}}^n(M^{\text{tor}})_{\integer{}}[\mathfrak{p}_{\textbf{E}}]. 
\end{align*}
We obtain the following isomorphisms from (\ref{comparison par}) and (\ref{comparison log}):
\begin{align*}
\widetilde{\overline{T}}
\simeq 
T_{\mathrm{cris}}(\widetilde{\overline{M}}),\ \ \ \ 
\widetilde{T}_{\textbf{f}}
\simeq 
T_{\mathrm{cris}}(\widetilde{M}_{\textbf{f}}),\ \ \ \ 
\widetilde{\overline{T}}_{\textbf{f}}
\simeq 
T_{\mathrm{cris}}(\widetilde{\overline{M}}_{\textbf{f}}),\ \ \ \ 
\widetilde{T}_{\textbf{E}}
\simeq 
T_{\mathrm{cris}}(\widetilde{M}_{\textbf{E}}).
\end{align*}

\begin{lem}\label{inj 0}
The canonical morphisms $\widetilde{\overline{T}}_{\mathbf{f}} \to \widetilde{\overline{T}}$ and $\widetilde{\overline{M}}_{\mathbf{f}} \to \widetilde{\overline{M}}$ are injective.
\end{lem}
\begin{proof}
The assertion follows from the fact that $\widetilde{H}_{\mathrm{\acute et},\pa}^n(M_{\overline{\Q}},\integer{})/(\widetilde{H}_{\mathrm{\acute et},\pa}^n(M_{\overline{\Q}},\integer{})[\ideal{p}_\textbf{f}])$ and $\widetilde{H}_{\text{log-cris},\pa}^n(M^{\mathrm{tor}})_{\integer{}}/(\widetilde{H}_{\text{log-cris},\pa}^n(M^{\mathrm{tor}})_{\integer{}}[\ideal{p}_\textbf{f}])$ are torsion-free.
\end{proof}

%
%
\subsection{Multiplicity one for $\text{Fil}^n (\widetilde{\overline{M}})$}\label{Rank Fil}
%
%

\begin{thm}\label{theorem:bottom}
There exists a canonical isomorphism $\mathrm{Fil}^n(\widetilde{H}_{\mathrm{log}\text{-}\mathrm{cris},\pa}^n(M^{\mathrm{tor}})_{\integer{}})\simeq S_2(\ideal{n},\integer{})$.
\end{thm}
\begin{proof}
By the degeneration of the Hodge spectral sequences for $H^{\ast}(M_{\Zp}^{\mathrm{tor}},\Omega_{M_{\Zp}^{\mathrm{tor}}}^n(\text{log}(D)))$ and $H^{\ast}(M_{\Zp}^{\mathrm{tor}},\Omega_{M_{\Zp}^{\mathrm{tor}}}^n(-\mathrm{log}(D)))$ (\cite[Theorem 4.1 and 4.1$^{\ast}$]{Fa}), we have canonical isomorphisms 
\begin{align}
\label{Fil bottom open}
\mathrm{Fil}^n(\widetilde{H}_{\mathrm{log}\text{-}\mathrm{cris}}^n(M^{\mathrm{tor}})_{\Zp})&\simeq \mathrm{Fil}^n(H_{\mathrm{log}\text{-}\mathrm{cris}}^n(M^{\mathrm{tor}})_{\Zp})\simeq H^0(M_{\Zp}^{\mathrm{tor}},\Omega_{M_{\Zp}^{\mathrm{tor}}}^n(\text{log}(D))),\\
\label{Fil bottom !}
\mathrm{Fil}^n(\widetilde{H}_{\mathrm{log}\text{-}\mathrm{cris},c}^n(M^{\mathrm{tor}})_{\Zp})&\simeq \mathrm{Fil}^n(H_{\mathrm{log}\text{-}\mathrm{cris},c}^n(M^{\mathrm{tor}})_{\Zp})\simeq H^0(M_{\Zp}^{\mathrm{tor}},\Omega_{M_{\Zp}^{\mathrm{tor}}}^n).
\end{align}
Note $\Omega_{M_{\Zp}^{\mathrm{tor}}}^n(-\mathrm{log}(D))=\Omega_{M_{\Zp}^{\mathrm{tor}}}^n$ for the last isomorphism.
Therefore $\mathrm{Fil}^n(\widetilde{H}_{\mathrm{log}\text{-}\mathrm{cris},\pa}^n(M^{\mathrm{tor}})_{\integer{}})$ is canonically isomorphic to the image of the homomorphism
\[
H^0(M_{\integer{}}^{\mathrm{tor}},\Omega_{M_{\integer{}}^{\mathrm{tor}}/\integer{}}^n) \to 
H^0(M_{\integer{}}^{\mathrm{tor}},\Omega_{M_{\integer{}}^{\mathrm{tor}}/\integer{}}^n(\text{log}(D))), 
\]
which is identified with $S_2(\ideal{n},\integer{})$ by the Koecher's principle and the $q$-expansion principle because $S_2(\ideal{n},\C)$ is identified with $\im(H^0(M_{\C}^{\mathrm{tor}},\Omega_{M_{\C}^{\mathrm{tor}}/\C}^{n}) \to H^0(M_{\C}^{\mathrm{tor}},\Omega_{M_{\C}^{\mathrm{tor}}/\C}^{n}(\mathrm{log}(D))))$ (\cite[Chapter II, \S 4]{Fre}).
\end{proof}

\begin{prop}\label{Fil f mod pi}
$(1)$ The dimension of $\mathrm{Fil}^n (\widetilde{\overline{M}})$ over $\kappa$ is equal to $1$. 

$(2)$ The homomorphism $\mathrm{Fil}^n (\widetilde{\overline{M}}_{\mathbf{f}})
\to
\mathrm{Fil}^n (\widetilde{\overline{M}})$ is an isomorphism.
\end{prop}
\begin{proof}
(1) We have $\mathrm{Fil}^n (\widetilde{\overline{M}})
=\left(S_2(\ideal{n},\integer{})/\varpi\right)[\ideal{m}]$ by Theorem \ref{theorem:bottom}.
By the duality theorem $S_2(\ideal{n},\integer{}) \simeq \Hom_{\integer{}}(\mathscr{H}_2(\ideal{n},\integer{}),\integer{})$ (\cite[Theorem 5.1]{Hida88}),
we obtain $\left(S_2(\ideal{n},\integer{})/\varpi\right)[\ideal{m}] 
\simeq \Hom_{\kappa}(\mathscr{H}_2(\ideal{n},\integer{})/\ideal{m},\kappa)$
and the dimension of the right-hand side over $\kappa$ is equal to $1$. 

(2) By Theorem \ref{theorem:bottom}, we have $\mathrm{Fil}^n (\widetilde{M}_{\mathbf{f}}) \simeq S_2(\ideal{n},\integer{})[\ideal{p}_{\mathbf{f}}]$ and the right-hand side is a free $\integer{}$-module of rank $1$.
Hence the claim follows from (1), $\mathrm{Fil}^n (\widetilde{\overline{M}}_{\mathbf{f}})=\mathrm{Fil}^n (\widetilde{M}_{\mathbf{f}})/\varpi \mathrm{Fil}^n (\widetilde{M}_{\mathbf{f}})$, and Lemma \ref{inj 0}.
\end{proof}

%
%
\subsection{Multiplicity one for $\widetilde{M}_{\mathbf{E}}$}\label{subsection:HT of Eis}
%
%

\begin{prop}\label{Eis HT}
Assume that $h_F^+=1$ and $C(\ideal{q},\mathbf{E})\neq N(\ideal{q})$ for some prime ideal $\ideal{q}$ dividing $\ideal{n}$. 
Then $\widetilde{M}_{\mathbf{E}}$ is free of rank $1$ over $\integer{}$ and $\mathrm{Fil}^n(\widetilde{M}_{\mathbf{E}})=\widetilde{M}_{\mathbf{E}}$. 
\end{prop}
\begin{proof}
The $U(\ideal{q})$-eigenvalue of each invariant form $\omega_{J'}$ defined in \S\ref{subsection:Eichler--Shimura--Harder} is equal to $N(\ideal{q})$ because 
\[
\omega_{J'}|U(\ideal{q})=\sum_{b\in \ideal{o}_F/\ideal{q}}
\begin{pmatrix}1&b\\0&g_{\ideal{q}}\end{pmatrix}^{\ast}\omega_{J'}
=N(\ideal{q})\omega_{J'}.
\]
Here the first equality follows from (\ref{U(q) explicit decomp}) and the second equality follows from that $\omega_{J'}$ is invariant under the action of the standard Borel subgroup $B_{\infty}$. 
Then, by the same arguments as in the proof of (\ref{+,+ decomp}), we see that 
$H_{\pa}^n (Y(\ideal{n}),\C)[\ideal{p}_{\mathbf{E}}]
\simeq H_{\mathrm{cusp}}^n (Y(\ideal{n}),\C)[\ideal{p}_{\mathbf{E}}]$ and the right\nobreakdash-hand side is $0$ because $H_{\mathrm{cusp}}^n (Y(\ideal{n}),\C) \simeq \bigoplus_{w\in W_G}S_2(\ideal{n},\C)$ as Hecke modules (cf. \cite[Corollary 2.2]{Hida94}) and the $q$-expansion principle.
Hence, by (\ref{split}), we obtain $H^n(Y(\ideal{n}),\C)[\mathfrak{p}_{\textbf{E}}] = H_{\mathrm{Eis}}^n(Y(\ideal{n}),\C)[\mathfrak{p}_{\textbf{E}}]$.
By combining with Proposition \ref{Eis Hodge number} (1), we have 
$H^n(Y(\ideal{n}),\C)[\mathfrak{p}_{\textbf{E}}] = \mathrm{Fil}^n (H^n(Y(\ideal{n}),\C)[\mathfrak{p}_{\textbf{E}}])$.
Since $H^n(Y(\ideal{n}),\C) \simeq H^n(M_{\C}^{\mathrm{tor}},\Omega_{M_{\C}^{\mathrm{tor}}/\C}^{\bullet}(\mathrm{log}(D)))$ as filtered modules, we see that 
$\widetilde{M}_{\mathbf{E}}=\mathrm{Fil}^n(\widetilde{M}_{\mathbf{E}}) \simeq H^0(M_{\integer{}}^{\mathrm{tor}},\Omega_{M_{\integer{}}^{\mathrm{tor}}/\integer{}}^{n}(\mathrm{log}(D)))[\ideal{p}_{\mathbf{E}}]$. Here the last isomorphism follows from (\ref{Fil bottom open}). 
The last term is free of rank $1$ over $\integer{}$ by the $q$-expansion principle. 
\end{proof}

Combining with Corollary \ref{thm:integral}, we obtain the following:
\begin{cor}\label{Eis generated by omega_E}
Under the same assumptions as Theorem \ref{cong mod},
$\widetilde{T}_{\mathbf{E}}$ is free $\integer{}$-module of rank $1$ generated by $[\omega_{\mathbf{E}}]$. 
\end{cor}

%
%
\subsection{The morphism $\widetilde{T}_{\mathbf{E}} \to \widetilde{T}$} \label{The morphism}
%
%

\begin{lem}\label{lemma:tilde mod pi exact}
Assume that both $H^n(\partial(Y(\ideal{n})^{\mathrm{BS}}),\integer{})_{\ideal{m}}$ and $H_c^{n+1}(Y(\ideal{n}),\integer{})_{\ideal{m}}$ are torsion-free, where $\ideal{m}$ is the maximal ideal of $\cal{H}_2(\ideal{n},\integer{})$ defined before Theorem \ref{cong mod}. 
Then the exact sequence $0 \to H_{\pa}^n(Y(\ideal{n}), \integer{})\to H^n(Y(\ideal{n}), \integer{})\to H^n(\partial(Y(\ideal{n})^{\mathrm{BS}}),\integer{})$ induces an exact sequence $0 \to \widetilde{H}_{\pa}^n(Y(\ideal{n}), \integer{})_{\ideal{m}}/\varpi\to \widetilde{H}^n(Y(\ideal{n}), \integer{})_{\ideal{m}}/\varpi \to \widetilde{H}^n(\partial(Y(\ideal{n})^{\mathrm{BS}}),\integer{})_{\ideal{m}}/\varpi$.
\end{lem}
\begin{proof}
We omit the coefficient $\integer{}$ of the cohomology groups to simplify the notation. 
Let $N$ be the image of $H^n(Y(\ideal{n}))$ in $H^n(\partial(Y(\ideal{n})^{\text{BS}}))$. 
Then we have an exact sequence $0 \to  N_{\ideal{m}} \to H^n(\partial(Y(\ideal{n})^{\text{BS}}))_{\ideal{m}} \to H_c^{n+1}(Y(\ideal{n}))_{\ideal{m}}$, whose last term is torsion-free by assumption. Therefore $N_{\ideal{m}}/\varpi \to H^n(\partial(Y(\ideal{n})^{\text{BS}}))_{\ideal{m}}/\varpi$ is injective. Since $N_{\ideal{m}}$ is torsion-free by assumption, we have $H_{\pa}^n(Y(\ideal{n}))_{\ideal{m},\mathrm{torsion}} \simeq H^n(Y(\ideal{n}))_{\ideal{m},\mathrm{torsion}}$ and obtain an exact sequence $0 \to \widetilde{H}_{\pa}^n(Y(\ideal{n}))_{\ideal{m}} \to \widetilde{H}^n(Y(\ideal{n}))_{\ideal{m}} \to N_{\ideal{m}} \to 0$.
By taking the reduction modulo $\varpi$, we obtain an exact sequence $0 \to \widetilde{H}_{\pa}^n(Y(\ideal{n}))_{\ideal{m}}/\varpi \to \widetilde{H}^n(Y(\ideal{n}))_{\ideal{m}}/\varpi \to N_{\ideal{m}}/\varpi \to 0$ because $N_{\ideal{m}}$ is torsion-free. 
This completes the proof.
\end{proof}

Note that $M[\ideal{m}]=M_{\ideal{m}}[\ideal{m}]$ for a $\cal{H}_2(\ideal{n},\integer{})$-module $M$.
Hence, by Lemma \ref{lemma:tilde mod pi exact}, we may regard $\widetilde{\overline{T}}$ as a submodule of $(\widetilde{H}_{\mathrm{\acute et}}^n(M_{\overline{\Q}},\integer{})/\varpi)[\ideal{m}]$ under the assumption of the lemma. 

\begin{prop}\label{prop:Eis mod pi}
Under the same assumptions as Theorem \ref{partial one}, the natural homomorphism $\widetilde{T}_{\textbf{E}}/\varpi \to (\widetilde{H}_{\mathrm{\acute et}}^n(M_{\overline{\Q}},\integer{})/\varpi)[\ideal{m}]$ is injective and its image is contained in $\widetilde{\overline{T}}$.
\end{prop}
\begin{proof}
One can prove the first claim in the same way as in Lemma \ref{inj 0}.
The assumption $\mathbf{E} \equiv \mathbf{f} (\bmod \varpi)$ is equivalent to $\mathbf{E}-\mathbf{f} \in \varpi M_2(\ideal{n},\integer{})$ by the $q$-expansion principle (\cite[Proposition 1.10 (i)]{Dim2}).
Applying Proposition \ref{const} (1) (resp. (2)) to the cohomology class of $\varpi^{-1}(\mathbf{E}-\mathbf{f})\in M_2(\ideal{n},\integer{})$ (resp. $\mathbf{f}\in S_2(\ideal{n},\integer{})$), we obtain $\varpi^{-1}\res([\omega_{\mathbf{E}}])=\res([\omega_{\varpi^{-1}(\mathbf{E}-\mathbf{f})}])\in \widetilde{H}^n(\partial(Y(\ideal{n})^{\mathrm{BS}}),\integer{})$.
Now the second claim follows from Corollary \ref{Eis generated by omega_E} and Lemma \ref{lemma:tilde mod pi exact}.
\end{proof}

%
%
\subsection{Proof of Theorem \ref{partial one}} \label{Pf of main}
%
%

In this subsection, we prove Theorem \ref{partial one}.
Let $\widetilde{\overline{T}}_{\mathbf{E}}$ (resp. $\widetilde{\overline{M}}_{\mathbf{E}}$) denote the quotient $\widetilde{T}_{\mathbf{E}}/\varpi$ (resp. $\widetilde{M}_{\mathbf{E}}/\varpi$) in $\mathbf{Rep}_{\integer{},\text{cris}}^{p-2}(G_{\Qp})$ (resp. $\mathbf{MF}_{\integer{}}^{p-2}$).
By Lemma \ref{inj 0} and Proposition \ref{prop:Eis mod pi}, we have the following monomorphisms in $\mathbf{Rep}_{\integer{},\text{cris}}^{p-2}(G_{\Qp})$ and $\mathbf{MF}_{\integer{}}^{p-2}$:
\[
\xymatrix{
\widetilde{\overline{T}}_{\mathbf{E}} 
\ar@{^{(}->}[r]^{\alpha_T}
&\widetilde{\overline{T}}
&\widetilde{\overline{T}}_{\mathbf{f}} \ar@{_{(}->}[l]_{\beta_T},
&\widetilde{\overline{M}}_{\mathbf{E}} 
\ar@{^{(}->}[r]^{\alpha_M}
&\widetilde{\overline{M}}
&\widetilde{\overline{M}}_{\mathbf{f}} \ar@{_{(}->}[l]_{\beta_M}.
}
\]
We define the action of $W_G$ on the underlying $\integer{}$-modules of $\widetilde{\overline{T}}_{\mathbf{E}}$, $\widetilde{\overline{T}}$, and $\widetilde{\overline{T}}_{\mathbf{f}}$ via the comparison isomorphism between \'etale and Betti cohomologies induced by the fixed embedding $\overline{\Q} \hookrightarrow \C$. 
Then the morphisms $\alpha_T$ and $\beta_T$ are $W_G$-equivariant.
By Proposition \ref{Fil f mod pi} (1) and Proposition \ref{Eis HT}, we have 
$\alpha_M(\widetilde{\overline{M}}_{\mathbf{E}})=\mathrm{Fil}^n(\widetilde{\overline{M}})$.
Hence, by Proposition \ref{Fil f mod pi} (2), we see that there exists a subobject $L$ of $\widetilde{\overline{T}}_{\mathbf{f}}$ in $\mathbf{Rep}_{\integer{},\text{cris}}^{p-2}(G_{\Qp})$ such that $\beta_T(L)=\alpha_T(\widetilde{\overline{T}}_{\mathbf{E}})$.
By Remark \ref{parity of Eis} and Corollary \ref{Eis generated by omega_E}, we have $\widetilde{\overline{T}}_{\mathbf{E}}=\widetilde{\overline{T}}_{\mathbf{E}}[\epsilon]$, which implies that $L$ is $W_G$-stable and $L=L[\epsilon]$.
Since the isomorphism (\ref{+,+ decomp}) over $\C$ says that the dimension of 
$\widetilde{\overline{T}}_{\textbf{f}}[\epsilon]$ over $\kappa$ is equal to $1$, we obtain $L=\widetilde{\overline{T}}_{\textbf{f}}[\epsilon]$. 
This completes the proof because $[\delta_{\mathbf{f}}]^{\epsilon}(\bmod \varpi)$ and $[\omega_{\mathbf{E}}]^{\epsilon}(\bmod \varpi)$ are bases of $\widetilde{\overline{T}}_{\textbf{f}}[\epsilon]$ and $\widetilde{\overline{T}}_{\textbf{E}}$, respectively (cf. \S \ref{subsection:Canonical periods}, Corollary \ref{Eis generated by omega_E}).


\end{document}